\RequirePackage{etex}
\documentclass[12pt]{amsart}

\usepackage{amsmath, amsthm, amsfonts, amssymb, amscd,comment,mathtools}
\allowdisplaybreaks[4]
\usepackage[mathscr]{eucal}
\usepackage[dvips, usenames, dvipsnames]{xcolor}
\definecolor{orange2}{HTML}{FF7F2A}
\usepackage[all, knot, arc]{xy}
\usepackage{graphicx}
\usepackage{pinlabel}
\usepackage{textcomp}
\usepackage{float}
\usepackage{stmaryrd}
\usepackage{enumitem}
\usepackage{tikz}
\usetikzlibrary{arrows,automata}
\usetikzlibrary{matrix,arrows}
\tikzset{cdlabel/.style={above,sloped,%
    execute at begin node=$\scriptstyle,execute at end node=$}}
\tikzset{algarrow/.style={->, thick}}   
\tikzset{alb/.style={->, bend right=25, thick}}
\tikzset{arb/.style={->, bend left=25, thick}}
\tikzset{al/.style={->, bend right=20, thick}}
\tikzset{ar/.style={->, bend left=20, thick}}
\tikzset{als/.style={->, bend right=15, thick}}
\tikzset{ars/.style={->, bend left=15, thick}}
\tikzset{blgarrow/.style={->, thick}}
\tikzset{clgarrow/.style={->, thick}}
\tikzset{tensoralgarrow/.style={double, double equal sign distance, -implies}}
\tikzset{tensorblgarrow/.style={double, double equal sign distance, -implies}}
\tikzset{tensorclgarrow/.style={double, double equal sign distance, -implies}}
\tikzset{tensorelgarrow/.style={double, double equal sign distance, -implies}}
\tikzset{modarrow/.style={->, dashed}}
\tikzset{Amodar/.style={->, dashed}}
\tikzset{Dmodar/.style={->, dashed}}
\tikzset{DAmodar/.style={->, dashed}}

\usepackage{tikz-cd}
\usepackage[noadjust]{cite}
\usepackage{xifthen}
\usepackage{slashbox}
\usepackage{multirow}
\usepackage{makecell}
\usepackage{longtable}
\usepackage[para]{threeparttable}
\usepackage[referable]{threeparttablex}
\usepackage{afterpage}
\usepackage{adjustbox}
\usepackage{caption}
\usepackage{footnote}
\makesavenoteenv{tabular}
\makesavenoteenv{table}
\usepackage{subcaption}

\usepackage{tabularx}
\usepackage{thmtools, thm-restate}

\usepackage[final, colorlinks=true, linkcolor=blue, citecolor=blue, urlcolor=blue]{hyperref}
\usepackage[capitalise,noabbrev,nameinlink]{cleveref}
\usepackage{soul}

\usepackage{graphbox}

\usepackage[lmargin=1in,rmargin=1in,tmargin=1in,bmargin=1in]{geometry}

\theoremstyle{definition}
\declaretheorem[name=Definition,style=definition,numberwithin=section]{definition}

\theoremstyle{plain}
\declaretheorem[name=Lemma,style=plain,numberlike=definition]{lemma}
\declaretheorem[name=Theorem,style=plain,numberlike=definition]{theorem}

\declaretheorem[name=Proposition,style=plain,numberlike=definition]{proposition}

\theoremstyle{remark}
\declaretheorem[name=Remark,style=remark,numberlike=definition]{remark}

\numberwithin{equation}{section}
\numberwithin{figure}{section}
\numberwithin{table}{section}

\newcommand{\comp}{\DOTSB\circ}
\let\oldphi\phi

\DeclareMathOperator{\id}{id}

\newcommand{\optilde}[1]{\widetilde{#1}}

\newcommand{\opunorr}[1]{#1 {}^u}
\DeclareMathOperator{\Kh}{Kh}

\newcommand{\CTu}{\opunorr{\CTu}}

\newcommand{\curves}[1]{\overrightarrow{#1}}
\newcommand{\alphas}[1][]{%
  \ifthenelse{\equal{#1}{}}{\curves{\alpha}}{\curves{\alpha^{#1}}}
}
\newcommand{\betas}[1][]{%
  \ifthenelse{\equal{#1}{}}{\curves{\beta}}{\curves{\beta_{#1}}}
}
\newcommand{\Xs}[1][]{%
  \ifthenelse{\equal{#1}{}}{\curves{\mathbb{X}}}{\curves{\mathbb{X}{#1}}}
}
\newcommand{\Os}[1][]{%
  \ifthenelse{\equal{#1}{}}{\curves{\mathbb{O}}}{\curves{\mathbb{O}{#1}}}
}

\newcommand{\emptypoly}[2][]{%
  #2^{\ifthenelse{\equal{#1}{}}{\circ}{\circ, #1}}
}

\makeatletter
\providecommand*{\twoheadrightarrowfill@}{%
  \arrowfill@\relbar\relbar\twoheadrightarrow
}
\providecommand*{\twoheadleftarrowfill@}{%
  \arrowfill@\twoheadleftarrow\relbar\relbar
}
\providecommand*{\xtwoheadrightarrow}[2][]{%
  \ext@arrow 0579\twoheadrightarrowfill@{#1}{#2}%
}
\providecommand*{\xtwoheadleftarrow}[2][]{%
  \ext@arrow 5097\twoheadleftarrowfill@{#1}{#2}%
}
\makeatother

\DeclareMathOperator{\Khred}{\overline{Kh}}
\DeclareMathOperator{\HFKhat}{\widehat{HFK}}
\newcommand{\C}{\text{C}_2^-}

\newcommand{\tensor}{\otimes}
\newcommand{\bigtensor}{\bigotimes}

\newcommand{\iso}{\cong}
\newcommand{\pmat}[1]{\begin{pmatrix}#1\end{pmatrix}}

\newcommand{\directsum}{\oplus}
\newcommand{\bigdirectsum}{\bigoplus}
\newcommand{\inj}[1][]{\xhookrightarrow{#1}}

\newcommand{\surj}[1][]{\xtwoheadrightarrow{#1}}

\newcommand{\homotopic}{\simeq}
\DeclareMathOperator{\cone}{cone}
\DeclareMathOperator{\smoothing}{sm}
\DeclareMathOperator{\cl}{cl}
\DeclareMathOperator{\gr}{gr}

\newcommand{\numset}[1]{\mathbb{#1}}

\newcommand{\Z}{\numset{Z}}

\newcommand{\R}{\numset{R}}
\newcommand{\Q}{\numset{Q}}
  \newcommand{\ctwohat}{\widehat{C}_2}

  \newcommand{\dr}{\mathcal{D}^\mathcal{R}}
  
  \newcommand{\LL}{\mathcal{L}}
  
  \newcommand{\ldplus}{\LL_{D}^+}

  \newcommand{\lsplus}{\LL_{S}^+}
  \newcommand{\lsplusprime}{\LL_{S'}^+}
  \newcommand{\lsplusdotdotdot}{\optilde{\lsplus}}
  
  \DeclareMathOperator{\CKh}{CKh}
  \DeclareMathOperator{\sm}{sm}

  \newcommand{\muthree}{\mu_{\rm{III}}}
  \newcommand{\mutwo}{\mu_{\rm{II}}}
  \newcommand{\muthreeprime}{\mu_{\rm{III}}'}
  \newcommand{\mutwoprime}{\mu_{\rm{II}}'}
  \newcommand{\muone}{\mu_{\rm{I}}}
  \newcommand{\muoneprime}{\mu_{\rm{I}}'}
  
  \newcommand{\dzw}{\mathcal{D}^\mathcal{B}}

  \newcommand{\coveredby}{\prec}
  
  \newcommand{\shortfact}[4]{\xymatrix{{#1}\ar@<1ex>[r]^{#2}& #3 \ar@<1ex>[l]^{ #4 }}}
  
  \DeclareMathOperator{\In}{In}
  \DeclareMathOperator{\Out}{Out}
  \DeclareMathOperator{\Fixed}{Fixed}
  \DeclareMathOperator{\Free}{Free}
  
  \newcommand{\filt}{\mathcal{F}}
  \newcommand{\spannedby}[1]{\langle #1 \rangle}

    \newcommand*{\defeq}{\mathrel{\vcenter{\baselineskip0.5ex \lineskiplimit0pt
    \hbox{\scriptsize.}\hbox{\scriptsize.}}}%
    =}
    
\newcommand{\intersect}{\cap}
\newcommand{\union}{\cup}

\newcolumntype{Y}{>{\centering\arraybackslash}X}

\newcommand{\nospaceperiod}{\makebox[0pt][l]{\,.}}
\graphicspath{{diagrams/}}

\begin{document}
\title{On The Invariance of the Dowlin Spectral Sequence}
\author{Samuel Tripp}
\address {Department of Mathematical Sciences, Worcester Polytechnic Institute, Worcester, MA 01609, 
  USA}
\email{\href{mailto:stripp@wpi.edu}{mailto:stripp@wpi.edu}}
\urladdr{\url{http://samueltripp.github.io}}
\author{Zachary Winkeler}
\address {Clark Science Center, Smith College, 44 College Lane, Northampton, MA 01063, USA}
\email{\href{mailto:zwinkeler@smith.edu}{zwinkeler@smith.edu}}


\begin{abstract}
Given a link $L$, Dowlin constructed a filtered complex inducing a spectral sequence with $E_2$-page isomorphic to the Khovanov homology $\Khred(L)$ and $E_\infty$-page isomorphic to the knot Floer homology $\HFKhat(m(L))$ of the mirror of the link. In this paper, we prove that the $E_k$-page of this spectral sequence is also a link invariant, for $k\ge 3$. 
\end{abstract}

\maketitle

\tableofcontents

\section{Introduction}
In \cite{nate}, Dowlin associates a filtered chain complex to a link $L$. The spectral sequence this filtered complex gives rise to has $E_2$-page isomorphic to the (reduced) Khovanov homology $\Khred(L)$ and converges to the knot Floer homology $\HFKhat(m(L))$ of the mirror of the link. The fact that the $E_2$- and $E_{\infty}$-pages of the spectral sequence are link invariants, not dependent on the diagram used to construct the filtered complex, suggests that the same may be true of all the higher pages of the spectral sequence. This is the main result of this paper. 

\begin{theorem}
For $k\ge 2$, the $E_k$-page of Dowlin's spectral sequence does not depend on the diagram used to construct the filtered complex, and is thus a link invariant. 
\end{theorem}

This provides a family of link invariants $\{E_k(L)\}_{k=2}^{\infty}$. The invariance of these higher pages of the Dowlin spectral sequence helps us further decipher the connection between Khovanov homology and knot Floer homology. It also suggests several future directions. 

The first is to find knots (or families of knots) which have the same Khovanov homology and knot Floer homology, but are distinguished by these higher page invariants. This is work in progress. A second direction is to consider implications in the study of transverse links. In \cite{olga}, Plamenevskaya identifies an invariant of transverse links $\psi(L)\in\Kh(L)$. Studying the effects of our invariance maps on this class in Khovanov homology, which is isomorphic to the $E_2$ page of the Dowlin spectral sequence, one could hope to define a countable family of transverse link invariants $\{\psi_k(L)\}_{k=2}^\infty$ by taking the image of $\psi$ on each page $E_k$ of Dowlin's spectral sequence, in the style of \cite{baldwinhigherinvariants}. It might prove interesting to compare these invariants, especially the image of $\psi$ on the $E_\infty$ page $\HFKhat(m(L))$ with known transverse link invariants \cite{braidgridloss}. A third direction for future work would be to perform a similar construction as that of the $s$ invariant in Khovanov homology \cite{sinvariant} and the $\tau$ invariant in knot Floer homology \cite{tauinvariant}.

\subsection*{Organization}
In \cref{sec:ctwominus} we review the construction of the filtered complex $\C$ inducing the spectral sequence from $\Khred(L)$ to $\HFKhat(m(L))$ for a link $L$, as originally defined by Dowlin in \cite{nate}. In \cref{sec:vertex-relabeling}, we prove that the homotopy type of this complex is invariant under a diagrammatic change called ``relabeling vertices''. In \cref{sec:moy-moves}, we define filtered chain maps between filtered complexes associated to diagrams separated by MOY moves, after recalling what these are. With these in hand, we prove invariance of the higher pages of the spectral sequence in \cref{sec:invariance}.

\subsection*{Conventions}
There are a few homological algebra conventions that we need to establish.

\begin{itemize}
    \item We will call our complexes chain complexes, despite the fact that our differentials will usually have degree $1$ with respect to the homological grading.
    \item Our filtrations will be \emph{descending}, which is to say that $\filt_i M \supseteq \filt_j M$ when $i \le j$.
    \item A \emph{filtered quasi-isomorphism} $f: A \to B$ is a filtered chain map which induces a quasi-isomorphism between the associated graded complexes $\gr(f): \gr(A) \to \gr(B)$. In other words, a filtered quasi-isomorphism induces a quasi-isomorphism between $E_0$-pages of spectral sequences, and equivalently induces isomorphisms between $E_1$-pages. If $A$ and $B$ are connected by a zig-zag of filtered quasi-isomorphisms, then they have the same weak filtered homotopy type, a relationship which we denote by $A \homotopic B$.
    \item Because the $E_1$-page of the filtered complex $\C$ is isomorphic to the Khovanov \emph{complex}, and not the Khovanov \emph{homology}, we will need to work with invariance maps which are not filtered quasi-isomorphisms. Instead, they only induce quasi-isomorphisms on the $E_1$-pages, or equivalently induce isomorphisms on the $E_2$-pages. We will call these maps \emph{$E_1$-quasi-isomorphisms} (terminology from \cite{cirici2019}). As above, we write $A \homotopic_1 B$ to denote that $A$ and $B$ are connected by a zig-zag of $E_1$-quasi-isomorphisms.
    \item Since we will be working with two different notions of weak equivalence, we also need two different mapping cones for a filtered map $f: A \to B$, denoted $\cone(f)$ and $\cone_1(f)$. Both of them will have the same underlying unfiltered complex, but will differ in the definition of the filtration. The former filtration is defined to be $\filt_i(\cone(f)) = \filt_i A \directsum \filt_i B$, whereas the latter filtration is given by $\filt_i(\cone_1(f)) = \filt_i A \directsum \filt_{i-1} B$.
\end{itemize}

\subsection*{Acknowledgements}
The authors thank Ina Petkova for suggesting this project, as well as providing helpful comments throughout, and thank John Baldwin and Nathan Dowlin for enlightening conversations. 
\section{The Spectral Sequence}\label{sec:ctwominus}

In this section, we review the construction of the spectral sequence from $\Khred(L)$ to $\HFKhat(L)$ for a link $L$, as originally defined by Dowlin in \cite{nate}. The spectral sequence arises from a filtered chain complex $\C(D)$ constructed from a \emph{partially singular braid diagram} $D$ associated to an unoriented link $L$. In \cref{sec:psbd}, we define these diagrams, and in \cref{sec:complex-definition} we associate a filtered chain complex to each such diagram. Finally, in \cref{sec:reidemeister-theorem}, we discuss how to associate a partially singular braid diagram to an unoriented link $L$, and we characterize the set of moves connecting any two such partially singular braid diagrams.

\subsection{Partially singular braid diagrams}\label{sec:psbd}

In this section we will define the types of diagrams we will need to construct the spectral sequence. We start by establishing some conventions regarding braid diagrams. If $D$ is a closed braid diagram, we can consider it as a $4$-valent graph embedded in $\R^2$ with vertices $V(D)$ the set of crossings, and edges $E(D)$ the set of arcs connecting these crossings. This agrees with the usual way of representing link diagrams as graphs. Given a graph $G$, recall that a \emph{subdivision} $H$ of $G$ is a graph obtained by adding $2$-valent vertices along edges of $G$.

\begin{definition}
A \emph{(closed) partially singular braid diagram} is an oriented graph embedded in $\R^2$ which can be obtained as a subdivision of a closed braid diagram, equipped with the following extra information:
\begin{itemize}
    \item a labeling of every $4$-valent vertex as ``positive'', ``negative'', or ``singular'',
    \item a further labeling of every singular vertex as either ``fixed'' or ``free'', and
    \item exactly one distinguished edge, which is called the ``decorated'' edge.
\end{itemize}
\end{definition}

An \emph{open} partially singular braid diagram is defined identically to a closed one, except that it also has $2n$ $1$-valent vertices (assuming $n$ strands) corresponding to the endpoints of the strands. When drawing partially singular braid diagrams, we indicate fixed singular vertices by drawing a circle around them, as in \cref{fig:vertex-types}; $2$-valent vertices are drawn simply as dots on the strands, and the decorated edge is denoted by two small lines, as in \cref{fig:other-diagram-features}.

\begin{figure}[ht]
    \centering
    \begin{tabularx}{0.8\textwidth}{YYYY}
        \includegraphics{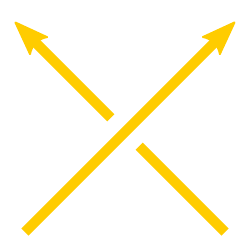} &
        \includegraphics{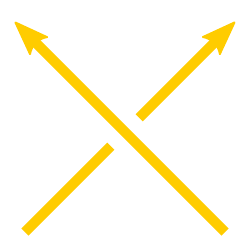} &
        \includegraphics{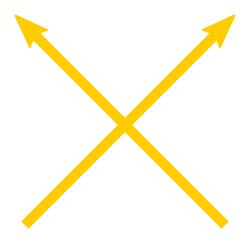} &
        \includegraphics{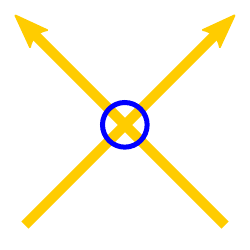} \\
        Positive & Negative & Free & Fixed
    \end{tabularx}
    \caption{The different types of vertices in a partially singular braid diagram.}
    \label{fig:vertex-types}
\end{figure}

\begin{figure}[ht]
    \centering
    \begin{tabularx}{0.8\textwidth}{YYYY}
        & \includegraphics{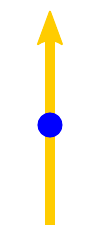} &
        \includegraphics{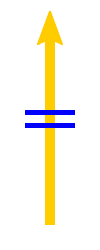} & \\
        & Bivalent vertex & Decorated edge &
    \end{tabularx}
    \caption{Other features that can occur in a braid diagram.}
    \label{fig:other-diagram-features}
\end{figure}

Throughout, we assume the decorated edge is leftmost in the diagram. We also assume a fixed ordering of the vertices whenever we consider a partially singular braid diagram $D$. We let $\Fixed(D)$ denote the set of fixed singular vertices of $D$ and $\Free(D)$ denote the set of free singular vertices of $D$. 

\begin{figure}[ht]
    \centering
    \begin{tikzcd}[row sep=0]
    & \includegraphics{diagrams/positive-crossing.pdf} \ar[dl, start anchor=real west, end anchor=north, "0"'] \ar[dr, start anchor=real east, end anchor=north, "1"] & \\
    \includegraphics{diagrams/free-vertex.pdf} & & \includegraphics{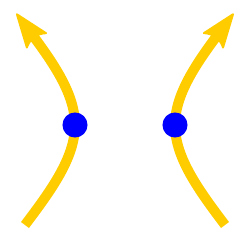} \\
    & \includegraphics{diagrams/negative-crossing.pdf} \ar[ul,start anchor=real west, end anchor=south,"1"] \ar[ur,start anchor=real east, end anchor=south,"0"'] &
    \end{tikzcd}
    \caption{The $0$- and $1$-resolutions of positive and negative crossings.}
    \label{fig:crossing-resolutions}
\end{figure}

A \emph{(fully) singular braid diagram} is a partially singular braid diagram with no crossings. It may arise from resolving a partially singular braid diagram $D$, in the following sense. Let $D$ be a partially singular braid diagram, with $c(D)$ the set of crossings of $D$. Then a \emph{resolution}, a function $I:c(D)\to \{0,1\}$, gives a fully singular braid diagram $D_I$ by resolving each crossing according to \cref{fig:crossing-resolutions}. In words, the $0$-resolution of a positive crossing is a singular vertex, and the $1$-resolution is the oriented smoothing with two subdivided edges. The $0$- and $1$-resolutions of a negative crossing are the $1$- and $0$-resolutions of a positive crossing, respectively. If a fully singular braid diagram $S$ arises as a complete resolution of a partially singular braid diagram $D$, then $\Fixed(S) = \Fixed(D)$, and $\Free(S)$ contains all crossings in $\Free(D)$ as well as those which were singularized in the resolution.


\subsection{The filtered complex \texorpdfstring{$\C(D)$}{C2-(D)}}\label{sec:complex-definition}

In this section, we recall Dowlin's construction of the filtered chain complex $\C(D)$ which gives rise to the spectral sequence connecting Khovanov homology to knot Floer homology.
Throughout, let $D$ be a partially singular braid diagram and $I$ a resolution of $D$ giving rise to the fully singular braid diagram $D_I$. We will first construct $\C(D_I)$ for each resolution $I$, then combine these into a cube complex $\C(D)$ by adding ``edge maps''. 

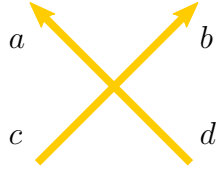
\begin{figure}[ht]
    \centering
\begingroup%
  \makeatletter%
  \providecommand\color[2][]{%
    \errmessage{(Inkscape) Color is used for the text in Inkscape, but the package 'color.sty' is not loaded}%
    \renewcommand\color[2][]{}%
  }%
  \providecommand\transparent[1]{%
    \errmessage{(Inkscape) Transparency is used (non-zero) for the text in Inkscape, but the package 'transparent.sty' is not loaded}%
    \renewcommand\transparent[1]{}%
  }%
  \providecommand\rotatebox[2]{#2}%
  \newcommand*\fsize{\dimexpr\f@size pt\relax}%
  \newcommand*\lineheight[1]{\fontsize{\fsize}{#1\fsize}\selectfont}%
  \ifx\svgwidth\undefined%
    \setlength{\unitlength}{86.40000343bp}%
    \ifx\svgscale\undefined%
      \relax%
    \else%
      \setlength{\unitlength}{\unitlength * \real{\svgscale}}%
    \fi%
  \else%
    \setlength{\unitlength}{\svgwidth}%
  \fi%
  \global\let\svgwidth\undefined%
  \global\let\svgscale\undefined%
  \makeatother%
  \begin{picture}(1,1)%
    \lineheight{1}%
    \setlength\tabcolsep{0pt}%
    \put(0,0){\includegraphics[width=\unitlength,page=1]{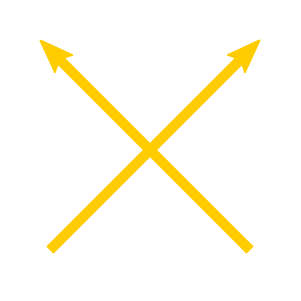}}%
    \put(0.04166667,0.66666666){\makebox(0,0)[lt]{\lineheight{1.25}\smash{\begin{tabular}[t]{l}$a$\end{tabular}}}}%
    \put(0.87500002,0.66666666){\makebox(0,0)[lt]{\lineheight{1.25}\smash{\begin{tabular}[t]{l}$b$\end{tabular}}}}%
    \put(0.04166667,0.25000002){\makebox(0,0)[lt]{\lineheight{1.25}\smash{\begin{tabular}[t]{l}$c$\end{tabular}}}}%
    \put(0.87500002,0.25000002){\makebox(0,0)[lt]{\lineheight{1.25}\smash{\begin{tabular}[t]{l}$d$\end{tabular}}}}%
  \end{picture}%
\endgroup%

    \caption{The local edge labels around a vertex.}
    \label{fig:edge-labels}
\end{figure}

To begin, label each edge of $D$ by a unique integer from $1$ to $k=|E(D)|$, and let $R(D) = \Q[U_1,\dots,U_k]$ be the polynomial ring over $\Q$ generated by one variable for each edge. Note that, whenever crossings in $D$ are resolved to get a diagram $D'$, there is a natural bijection between edges in $D$ and edges in $D'$, so we can extend our edge labels to any resolution of $D$. To each vertex $v\in V(D)$ we associate two polynomials, $L(v)$ and $L^+(v)$. Label the adjacent edges to each vertex $v\in V(D)$ as in \cref{fig:edge-labels}; if we draw the vertex such that all edges are oriented upwards, then we label the edge in the top left by $a$, the remaining edges by $b$, $c$, and $d$ as we traverse clockwise from the edge labeled $a$. Define $L(v) = U_a + U_b - U_c - U_d$ and $L^+(v) = U_a + U_b + U_c + U_d$.

One factor of $\C(D_I)$ will depend not on the specific resolution but only on $D$; we denote this factor $\ldplus$. Let
\begin{equation*}
    \ldplus \defeq \bigtensor_{v \in \Fixed(D)} \shortfact{R(D)}{L(v)}{R(D)}{L^+(v)} \,.
\end{equation*}

It should be noted that $\ldplus$ is not a chain complex, but rather a \emph{matrix factorization} (or \emph{curved complex}). A \emph{matrix factorization} is a module $M$ equipped with an endomorphism $\partial:M\to M$ such that $\partial^2=\omega\id_M$ for some potentially-nonzero scalar $\omega$, which is called the \emph{potential} of the matrix factorization. Despite the fact that $\partial$ does not square to zero, we still find occasion to refer to it as a \emph{differential} on $M$; this will be clear from context. In the case of $\ldplus$, $\omega = \sum_{v \in \Fixed(D)} L(v) L^+(v)$, which is often nonzero in $R(D)$. Matrix factorizations are well-studied algebraic objects, but for our purposes we will only need a few facts about them which will be found in \cref{sec:vertex-relabeling}.

The factor of $\C(D_I)$ which is dependent on the specific resolution is the $R(D)$-module $Q(D_I)=R(D)/(L(D_I)+N(D_I))$, where $L(D_I)$ and $N(D_I$) are two ideals of $R(D)$. The first of these is the \textit{linear ideal} $L(D_I)$, defined as
\begin{equation*}
    L(D_I) \defeq \sum_{v \in \Free(D_I)} (L(v)) \,.
\end{equation*}

The second is the \textit{nonlocal ideal} $N(D_I)$. Let $\Omega$ be a smoothly-embedded disk in $\R^2$ that does not contain the decorated edge, and such that the boundary only intersects $D$ transversely at edges. Let $\In(\Omega)$ (resp.\ $\Out(\Omega)$) denote the set of edges that intersect the boundary of $\Omega$ and are oriented inward (resp.\ outward). We define $N(\Omega)$ to be the polynomial
\begin{equation*}
    N(\Omega) \defeq \prod_{i \in \Out(\Omega)} U_i - \prod_{j \in \In(\Omega)} U_j \,.
\end{equation*}

The nonlocal ideal $N(D_I)$ is then generated by $N(\Omega)$ for all such embedded disks $\Omega$:
\begin{equation*}
    N(D_I) \defeq \sum_{\Omega} (N(\Omega)) \,.
\end{equation*}

With the above definitions in hand, the complex $\C(D_I)$ is then defined as
\begin{align*}
    \C(D_I) &\defeq Q(D_I) \tensor \ldplus \\
    &\defeq R(D)/(L(D_I)+N(D_I))\otimes \left( \bigtensor_{v \in \Fixed(D)} \shortfact{R(D)}{L(v)}{R(D)}{L^+(v)} \right)\,.
\end{align*}

It is shown in \cite[Lemma 2.4]{nate} that the potential $\omega$ of $\ldplus$ is contained in $L(D_I)+N(D_I)$, and thus is zero in $Q(D_I)$. Thus, the endomorphism of $\C(D_I)$ induced by $\ldplus$ squares to 0, so it is truly a differential; we denote it $d_0$.

As a module, define
\begin{equation*}
    \C(D) \defeq \bigdirectsum_{I \in \{0,1\}^{c(D)}} \C(D_I) \,.
\end{equation*}
The differential on $\C(D)$ will be defined as a sum $d_0 + d_1$, where $d_0$ is induced by the differential $d_0$ on the summands $\C(D_I)$, and $d_1$ is induced by edge maps that we have yet to define. In order to do so, we must first restrict the set of partially-singular braid diagrams we are working with.

\begin{definition}[{\cite[Definition 2.2]{nate}}]
The set $\dr$ contains all partially-singular braid diagrams $D$ satisfying the following conditions for all $I \in \{0,1\}^{c(D)}$:
\begin{itemize}
    \item $D_I$ is connected, and
    \item the linear terms $L(v)$ for $v \in \Free(D_I)$ form a regular sequence\footnote{The $\mathcal{R}$ in $\dr$ likely stands for ``regular''.} over $R(D)/N(D_I)$.
\end{itemize}
\end{definition}
The latter condition is an algebraic restriction which will be used in the proof of \cref{thm:vertex-relabeling}. It is equivalent to the existence of an ordering $v_1, \dots, v_k$ of the vertices in $\Free(D_I)$ such that $L(v_j)$ is not a zero divisor in $R(D)/(N(D_I) + (L(v_1),\dots,L(v_{j-1})))$ for each $1 \le j \le k$. Since $R(D)$ is a graded ring and the linear terms $L(v)$ are homogeneous of positive degree, if this condition is true for one ordering of $\Free(D_I)$, it is true for \textit{every} ordering.

For the rest of the definition of $\C(D)$, we will assume $D\in\dr$. Let $I$ and $J$ be two resolutions with $I \coveredby J$, i.e.\ $I$ and $J$ agree on all crossings except a single $c \in c(D)$, where $I(c) = 0$ and $J(c) = 1$. Let $v$ be the vertex corresponding to $c$, and label the edges adjacent to $v$ according to \cref{fig:edge-labels}.

The edge map $d_{I,J}:\C(D_I)\to \C(D_J)$ depends on whether $c$ is a positive or negative crossing. If $I$ and $J$ differ at a positive crossing, let $\oldphi_+: Q(D_I) \to Q(D_J)$ be the unique $R(D)$-module map such that $\oldphi_+(1)=1$, and define the edge map $d_{I,J}: \C(D_I) \to \C(D_J)$ to be $d_{I,J} = \oldphi_+ \tensor \id_{\ldplus}$. Else, $I$ and $J$ differ at a negative crossing $v$. In this case, let $\oldphi_-: Q(D_I) \to Q(D_J)$ be the unique $R(D)$-module map such that $\oldphi_-(1) = U_b - U_c$, and define the edge map $d_{I,J}: \C(D_I) \to \C(D_J)$ to be $d_{I,J} = \oldphi_- \tensor \id_{\ldplus}$. We may occasionally overload notation by referring to the edge map $d_{I,J}$ as $\phi_{\pm}$ when there is no risk of confusion.

Combining all of these maps together into a single map induces $d_1: \C(D) \to \C(D)$, given by
\begin{equation*}
    d_1 \defeq \sum_{I \coveredby J} \epsilon(I,J) d_{I,J} \,.
\end{equation*}
Here, $\epsilon(I,J)$ is a \emph{sign assignment}, which is a labeling of the edges of the cube of resolutions by $\{\pm 1\}$ satisfying the property that every square face has an odd number of $-1$-labeled edges. Such a sign assignment ensures that $(d_1)^2=0$, and any two choices of $\epsilon$ result in isomorphic complexes. As one example, we may let $\epsilon(I,J) = (-1)^k$, where $k$ is the number of $1$'s that come before the place at which $I$ and $J$ differ, as in \cite{barnatan2002}. We will further abuse notation by referring to $d_1$, the signed sum of the edge maps $d_{I,J}$ for all $I\coveredby J$, as itself an edge map.  

Consider $\C(D)$ as a chain complex with total differential $d_0 + d_1$. We filter $\C(D)$ by weight in the cube of resolutions, i.e.\ the filtration on $\C(D)$ is given by
\begin{equation*}
    \filt_p \C(D) \defeq \bigdirectsum_{w(I) \ge p} \C(D_I) \,,
\end{equation*}
where $w(I)=\sum_{c\in c(D)} I(c)$ is the \textit{weight} of $I$, i.e.\ the number of $1$-resolved crossings of $D_I$. Note that $d_0$ preserves the weight, and $d_1$ increases it by $1$, so the differential on $\C(D)$ is indeed filtered with respect to this decomposition.

\begin{remark}
We could have alternately defined $\C(D)$ by first defining $\C(S)$ for fully singular braid diagrams $S$, then defining $\C(D)$ to be the mapping cone
\begin{equation*}
\C(D) \defeq \cone_1(\oldphi \tensor \ldplus) = \left(\C(D_0) \to \C(D_1)\right) \,,
\end{equation*}
where $D_0$ and $D_1$ above are the $0$- and $1$-resolutions of a particular crossing, and $\oldphi: Q(D_0) \to Q(D_1)$ is the associated map of quotient modules. Iterating this construction produces a filtered complex that is isomorphic to the one that we defined previously.
\end{remark}


\subsection{Diagrams associated to a link}\label{sec:reidemeister-theorem}

Each partially singular braid diagram gives rise to an unoriented link by taking the \emph{unoriented smoothing.} 

\begin{definition}
Let $D$ be a partially singular braid diagram. The \emph{unoriented smoothing $\smoothing(D)$} is the unoriented link obtained from $D$ by smoothing each singular vertex in the way that does not respect the orientation. 
\end{definition}

\begin{figure}[ht]
    \centering
    \begin{tikzcd}
    \includegraphics[align=c]{diagrams/free-vertex.pdf} \ar[r,"\sm"] & \includegraphics[align=c]{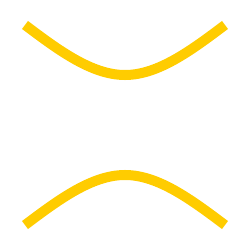}
    \end{tikzcd}
    \caption{Unoriented smoothing of a crossing.}
    \label{fig:unoriented-smoothing}
\end{figure}

\cref{fig:unoriented-smoothing} shows a local picture of smoothing a singular vertex, and \cref{fig:diagram-smoothing} gives an example of a partially singular braid diagram and the link obtained by taking the unoriented smoothing. 

\begin{figure}[ht]
    \centering
    \begin{tikzcd}
    \includegraphics[width=0.2\textwidth, align=c]{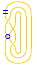} \ar[r,"\sm"] & \includegraphics[width=0.2\textwidth, align=c]{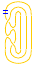}
    \end{tikzcd}
    \caption{A diagram $D$ and its unoriented smoothing $\smoothing(D)$.}
    \label{fig:diagram-smoothing}
\end{figure}

When $\sm(D)$ is an $\ell$-component link, we can construct a ``reduced'' version of $\C(D)$. First, choose a set of edges $e_1, \dots, e_\ell \in E(D)$ such that each $e_i$ is on a distinct component of $\sm(D)$. Then, let
\begin{equation*}
    \ctwohat(D) \defeq \C(D) \tensor \bigtensor_{e_i} \left(R(D) \xrightarrow{U_{e_i}} R(D) \right) \,.
\end{equation*}
We define the differentials given by multiplication by $U_{e_i}$ to have weight filtration degree $1$. Therefore, we get a weight filtration on $\ctwohat(D)$ induced by the above definition as a tensor product of filtered complexes. This is the filtered complex that is used to define the spectral sequence relating Khovanov homology and knot Floer homology.

\begin{theorem}[{\cite[Theorem 1.6]{nate}}]
Let $D \in \dr$ be a partially singular braid diagram with $\sm(D) = L$. The spectral sequence induced by the weight filtration on $\ctwohat(D)$ has $E_2$-page isomorphic to $\Khred(L)$ and converges to $\HFKhat(L)$.
\end{theorem}

In \cite{nate}, Dowlin proves that every link can be realized as the unoriented smoothing of a diagram in $\dr$ by first considering a braid whose plat closure is the desired link, then turning that braid into a partially-singular braid diagram. We will go about things similarly, but instead choose a different way of embedding braid closures into $\dr$ that better fits our particular invariance proofs.

\begin{proposition}\label{prop:psbd-existence}
Let $L$ be an unoriented link. There is a partially singular braid diagram $D \in \dr$ such that $\smoothing(D)=L$. 
\end{proposition}

To prove \cref{prop:psbd-existence}, we will make use of a special partially singular open braid diagram which we denote $I_n$. This open diagram $I_n$ consists of $2n$ upward oriented strands with $2n-1$ layers of singular vertices. The layers are symmetric, meaning layer $i$ has singular vertices between the same strands as layer $2n-i$ for $1\le i<n$. The first layer has a singular vertex between the strands $n$ and $n+1$. The second layer has two singular vertices; one between strands $n-1$ and $n$ and one between strands $n+1$ and $n+2$. In general, the $i^{\text{th}}$ layer has $i$ consecutive singular vertices, beginning with one between strands $n+1-i$ and $n+2-i$ and ending with one between strands $n-1+i$ and $n+i$. We let $\Fixed(I_n)$ be the singular vertices in layers $n$ and $n+1$, and let $\Free(I_n)$ be the rest of the singular vertices. See \cref{fig:i-3} for $I_n$ in the case $n=3$. 

\begin{figure}[ht]
    \centering
    \includegraphics{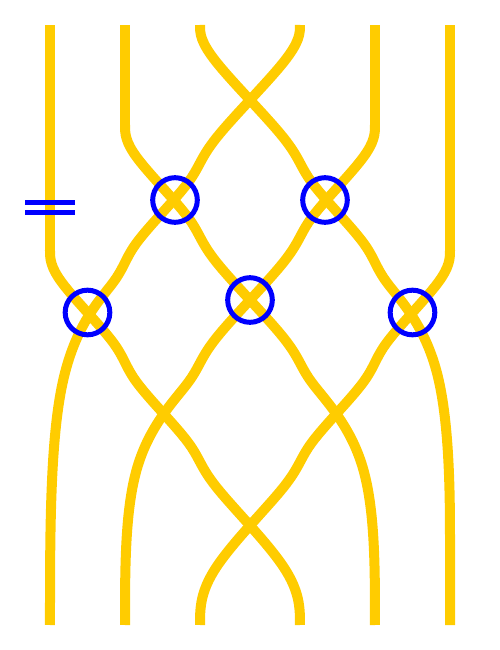}
    \caption{The partially singular braid diagram $I_n$ in the case $n=3$. }
    \label{fig:i-3}
\end{figure}

\begin{definition}\label{def:i-n-beta}
Given a braid $\beta \in B_n$, let $I_n(\beta)$ denote the partially-singular braid diagram $D$ built by putting $n$ downward-oriented strands to the right of $\beta$, and putting $I_n$ above and taking the braid closure.
\end{definition}

\begin{proof}[Proof of \cref{prop:psbd-existence}]
Given an unoriented link $L$, let $\beta$ be a braid with braid closure $\cl(\beta)$ isotopic to $L$, the existence of which is guaranteed by Alexander's Theorem \cite{alexander}. The unoriented smoothing $\sm(I_n(\beta))$ is isotopic to the braid closure $\cl(\beta)$ of $\beta$ itself, so $D=I_n(\beta)$ is a partially singular braid diagram with $\sm(D)$ isotopic to $L$. That $D\in\dr$ is an application of \cite[Lemma 7.1]{nate}. More specifically, $D$ contains a vertically-mirrored copy of the open braid diagram $S_{2n}$ defined in \cite{nate}, where it is proven that any such diagram is in $\dr$.
\end{proof}

See \cref{fig:partially singular-braid-representation} for an example of the process of constructing a partially singular braid diagram with specified unoriented smoothing. 

\begin{figure}[ht]
    \centering
    \begin{tikzcd}
    \includegraphics[align=c]{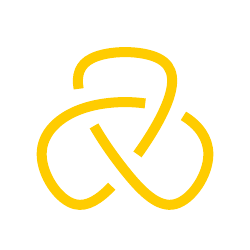} \ar[r] & \includegraphics[align=c]{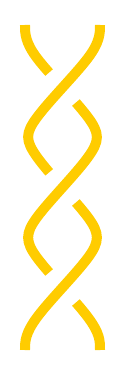} \ar[r] & \includegraphics[align=c]{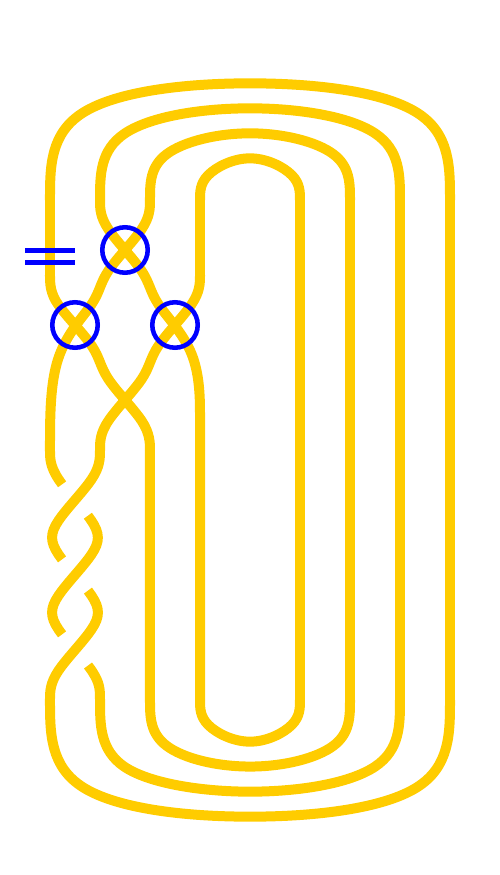}
    \end{tikzcd}
    \caption{The process of constructing a partially singular braid diagram with smoothing isotopic to a given knot.}
    \label{fig:partially singular-braid-representation}
\end{figure}

Let $\dzw = \{I_n(\beta) \mid \beta \in B_n, n \in \Z\}$ be the set of partially singular braid diagrams constructed as above\footnote{Here, the $\mathcal{B}$ in $\dzw$ stands for ``braid''.}. Then we have the following classification theorem.
\begin{theorem}\label{thm:markov}
Two diagrams in $\dzw$ have the same unoriented smoothing if and only if the underlying braids are connected by a finite sequence of Reidemeister II moves, Reidemeister III moves, (de)stabilizations, and conjugations. 
\end{theorem}
\begin{proof}
This is just Markov's theorem, repackaged \cite{markov}. 
\end{proof}

Since $\dzw\subset\dr$, we can construct the complex $\C(D)$ for any $D\in\dzw$. We overload notation by writing $\C(\beta)$ instead of $\C(I_n(\beta))$ for $\beta\in B_n$. We prove invariance of $\C(\beta)$ under the moves in \cref{thm:markov} in \cref{sec:invariance} using maps defined in \cref{sec:moy-moves}. 
\section{Vertex Relabeling}\label{sec:vertex-relabeling}

Before we continue towards a proof of invariance, we detour to comment on a quirk of the construction of $\C(D)$. One natural question to ask is why $\C(D)$ treats fixed and free singular vertices differently. It turns out that, in order for $H_*(\C(D))$ to be isomorphic to $\HFKhat(\sm(D))$, our diagram $D$ needs to be in $\dr$, which means satisfying the regular sequence condition. This condition cannot be satisfied unless $D$ contains sufficiently many fixed vertices in a sufficiently nice arrangement. On the other hand, we only know how to define the edge maps $d_{I,J}$ on free vertices, so we cannot make all of our vertices fixed either.

As a sort of compromise, we choose some of our vertices to be fixed and some to be free. We will not need to worry about which choice we have made when proving invariance under Reidemeister moves II and III in \cref{sec:invariance}, since they only involve local pictures of diagrams which contain some crossings but no singular vertices. While not a local move, we define stabilization to be compatible with our vertex labeling as well. Conjugation, however, will require us to change which vertices are fixed and which are free; this is what motivates the following theorem.

While it is not immediately obvious, it turns out that the homotopy type of $\C(D)$ does not depend on the particular labeling of vertices as fixed or free in the following sense:

\begin{theorem}\label{thm:vertex-relabeling}
If $D, D' \in \dr$ are identical partially singular braid diagrams up to relabeling of fixed and free vertices, then $\C(D) \homotopic \C(D')$.
\end{theorem}

To prove this, we need to introduce a slight variation on the technique of ``excluding a variable'' from \cite[Lemma 3.8]{rasmussen2006} or \cite[Proposition 9]{khovanov-rozansky2004}. Both sources are also good references for the relevant details on matrix factorizations, including the statement below on the effect of change of basis on matrix factorizations. 

We include the necessary details on matrix factorizations below. Let $R$ be a ring. For $a,b\in R$, let $\{a,b\}$ denote the matrix factorization
\begin{equation*}
    \{a,b\} \defeq \shortfact{R}{b}{R}{a} \,.
\end{equation*}
For $\vec{a},\vec{b} \in R^n$, let $\{\vec{a},\vec{b}\}=\pmat{a_1 & b_1 \\ a_2 & b_2 \\ \vdots & \vdots \\ a_n & b_n}$ denote the matrix factorization
\begin{equation*}
    \{\vec{a},\vec{b}\} \defeq \bigtensor_{i = 1}^{n} \{a_i,b_i\} =\bigtensor_{i = 1}^{n} \shortfact{R}{b_i}{R}{a_i} \,.
\end{equation*}
We have already seen a matrix factorization of this form; if we let $\vec{a} = (L^+(v_1),\dots,L^+(v_n))$ and $\vec{b} = (L(v_1),\dots,L(v_n))$ for a partially singular braid diagram $D$ with $\Fixed(D) = \{v_1,\dots,v_n\}$, then $\ldplus = \{\vec{a},\vec{b}\}$. By definition, the potential $\omega$ associated to the matrix factorization $\{\vec{a},\vec{b}\}$ is $\vec{a}\cdot\vec{b}=a_1b_1+\ldots+a_nb_n$. 

A change of basis for $R^n$ gives us an equivalent matrix factorization. One can check what effect various change of basis operations have on the representing matrix $\{\vec{a},\vec{b}\}$. Below, we will need just one change-of-basis operation: sending $\vec{e_i}$ to $\vec{e_i}+c\vec{e_j}$ for standard basis vectors $\vec{e_i}$ and $\vec{e_j}$ of $R^n$. This has the effect of replacing the matrix factorization by $\{\vec{a}',\vec{b}'\}$, where
\begin{equation*}
    \vec{a}'_k=\begin{cases}\vec{a}_k+c\vec{a}_j&\text{if $k=i$}\\a_k&\text{else}\end{cases}
\end{equation*}
and 
\begin{equation*}
\vec{b}'_k=\begin{cases}\vec{b}_k-c\vec{b}_i&\text{if $k=j$}\\b_k&\text{else}\end{cases} \nospaceperiod
\end{equation*}
For more details, see \cite{rasmussen2006,khovanov-rozansky2004}. 

Let $C=\{\vec{a},\vec{b}\}$ be any matrix factorization over $R$. We can decompose \[
    C = \shortfact{C'}{b_1}{C'}{a_1} \,,
\] where $C'=\{\vec{a}',\vec{b}'\}$ is the factorization obtained by omitting the first components of $\vec{a}$ and $\vec{b}$. Define $\pi: C \to C' \tensor R/(b_1)$ by $\pi((c_1,c_2))=c_2\otimes 1$. Before proving \cref{thm:vertex-relabeling}, we prove that if the potential of $C$ is 0 and $b_1$ is a non-zero-divisor in $R$, then $\pi$ is a quasi-isomorphism. 

\begin{lemma}\label{lemma:matrix-factorization-reduction}
If the potential of $C$ is 0 and $b_1$ is a non-zero-divisor in $R$, then $\pi$ is a quasi-isomorphism of chain complexes.
\end{lemma}

\begin{proof}
It is clear that $\pi$ is surjective; since $b_1$ is a non-zero-divisor, multiplication by $b_1$ is injective, so we have the following short exact sequence:

\begin{center}
\begin{tikzcd}[ampersand replacement=\&, row sep = huge, column sep = huge]
    0 \ar[r] \ar[d,shift left] \& C' \ar[r,"1"] \ar[d,shift left,"1"] \& C' \ar[r] \ar[d,shift left,"b_1"] \& 0 \ar[r] \ar[d,shift left] \& 0 \ar[d,shift left] \\
    0 \ar[r] \ar[u,shift left] \& C' \ar[r,"b_1"] \ar[u,shift left,"a_1b_1"] \& C' \ar[r] \ar[u,shift left,"a_1"] \& C'\otimes R/(b_1) \ar[r] \ar[u,shift left] \& 0 \ar[u,shift left] \nospaceperiod
\end{tikzcd}
\end{center}

Let $C''$ denote the first nonzero column in this sequence, the matrix factorization  $\shortfact{C'}{1}{C'}{a_1b_1}.$ By the corresponding long exact sequence in homology, it suffices to show that $C''$ is acyclic in order to prove that $\pi$ is a quasi-isomorphism. We write $C''$ in matrix form, then apply our above remarks about change of basis:
\begin{align*}
    \pmat{a_1b_1 & 1 \\ a_2 & b_2 \\ \vdots & \vdots \\ a_n & b_n} \sim \pmat{a_1b_1 + a_2b_2 & 1 \\ a_2 & 0 \\ \vdots & \vdots \\ a_n & b_n} \sim \pmat{\omega & 1 \\ a_2 & 0 \\ \vdots & \vdots \\ a_n & 0} \,.
\end{align*}
Since we know that the potential $\omega = 0$, we see that $C'' = \shortfact{\{\vec{a}',\vec{0}\}}{1}{\{\vec{a}',\vec{0}\}}{0}$, and therefore is acyclic.
\end{proof}

With this lemma, we can now prove that $\C(D)$ is independent of vertex labeling for $D \in \dr$:

\begin{proof}[Proof of \cref{thm:vertex-relabeling}]
Let $S \in \dr$ be a fully singular braid diagram, and let $w \in \Fixed(S)$ be some fixed vertex such that if $w$ were instead free, the new diagram $S'$ would still be in $\dr$. Note that $R(S') = R(S)$, and $Q(S') = Q(S)/(L(w))$. Since $\C(S) = Q(S) \tensor \lsplus$, we may consider $\C(S)$ as the matrix factorization $\{\vec{a},\vec{b}\}$ over $R = Q(S)$ with $\vec{a} = (L^+(v))_{v\in\Fixed(S)}$ and $\vec{b} = (L(v))_{v \in \Fixed(S)}$. Assume without loss of generality that $b_1 = L(w)$. Since $S' \in \dr$, we know that the linear terms $L(v)$ for $v \in \Free(S')$ form a regular sequence over $R(S')/N(S') = R(S)/N(S)$, and in particular, $L(w)$ is a non-zero-divisor in $Q(S)$, since $w\in\Free(S')$. We then get that
\begin{align*}
    \C(S) &= Q(S) \tensor_{R(S)} \lsplus \\
    &\iso \{\vec{a},\vec{b}\} \\
    &\homotopic Q(S)/(L(w)) \tensor_{Q(S)} \left( Q(S) \tensor_{R(S)} \{\vec{a}',\vec{b}'\} \right) && \text{(by \cref{lemma:matrix-factorization-reduction})} \\
    &\homotopic \left( Q(S)/(L(w)) \tensor_{Q(S)}  Q(S) \right) \tensor_{R(S)} \{\vec{a}',\vec{b}'\}  && \text{(by associativity of $\tensor$)} \\
    &\homotopic Q(S)/(L(w)) \tensor_{R(S)} \{\vec{a}',\vec{b}'\} \\
    &\iso Q(S') \tensor_{R(S)} \lsplusprime \\
    &\iso Q(S') \tensor_{R(S')} \lsplusprime && \text{(since $R(S) = R(S')$)} \\
    &= \C(S') \,.
\end{align*}
Therefore, we see that changing a fixed vertex to a free one in a fully singular diagram does not change the homotopy type of $\C(-)$ as long as both diagrams are in $\dr$.

Now, we need to extend this result. Let $D,D' \in \dr$ be partially singular braid diagrams that differ only on the labeling of a single vertex $w \in \Fixed(D) \intersect \Free(D')$. We know that $\C(D_I) \homotopic \C(D'_I)$ for all $I \in \{0,1\}^{c(D)}$. In particular, we have a map in one direction: $\pi: \C(D_I) \to \C(D'_I)$ is a filtered quasi-isomorphism inducing the above equivalence. Therefore, it suffices to show that $\pi$ commutes with the edge map $d_1$, which is the sum of $d_{I,J}$. Since $\pi$ is linear over $Q(S)$, we get that it is additionally $R(S)$-linear via the natural quotient map, and therefore commutes with scalar multiplication by elements of $R(S)$. Since the edge maps $d_{I,J}$ are defined via scalar multiplication by $1$ or $U_b - U_c$, we see that $\pi$ does in fact commute with the edge maps, and therefore extends to a filtered quasi-isomorphism $\pi: \C(D) \to \C(D')$ by \cref{lemma:cone-components}.

Given any two diagrams $D', D'' \in \dr$ that differ only by some number of vertex labels, we can construct a diagram $D \in \dr$ with $\Fixed(D) = \Fixed(D') \union \Fixed(D'')$, and therefore get that $\C(D') \homotopic \C(D) \homotopic \C(D'')$, thus proving the general case.
\end{proof}

\section{MOY Moves}\label{sec:moy-moves}

In \cite{moy1998}, Murakami, Ohtsuki, and Yamada study local operations on singular diagrams (``MOY moves''). While originally formulated for oriented planar trivalent graphs, they are relevant to us because one can think of singular vertices in our braids and braid resolutions as pairs of trivalent vertices instead. Two of these moves, MOY I and MOY III, represent planar isotopy when applied to the unoriented smoothing of a diagram, and thus are useful to make up for the fact that we cannot isotope singularized crossings in the same ways that we can smoothed ones. The MOY II move corresponds to a cup/cap cobordism, but is rather more limited in its application. Nevertheless, these three moves will suffice to define Reidemeister moves (and others) in \cref{sec:invariance}. The maps that we choose to realize these moves are inspired by those used in \cite{khovanov-rozansky2004} and \cite{khovanov-rozansky2008}.

In this section, we construct filtered chain maps relating $\C(D)$ and $\C(D')$, where $D$ and $D'$ are partially singular braid diagrams connected by an MOY I, II, or III move.

\subsection{MOY I}\label{sec:moy-i}

Suppose $D$ and $D'$ are partially singular braid diagrams that differ by an MOY I move, as illustrated in \cref{fig:moy-i}. In words, there is a fixed vertex $v$ in $D$ that meets the same edge $e$ twice; without loss of generality, $e$ is to the right of $v$. The diagram $D'$ is then obtained from $D$ by removing the edge $e$ and relabeling $v$ as a bivalent vertex.

\begin{theorem}\label{thm:moy-i}
There exist $R(D')$-linear filtered quasi-isomorphisms $\muone: \C(D) \to \C(D')$ and $\muoneprime: \C(D') \to \C(D)$. Under the identification $E_1(\C(-)) \iso \CKh^-(\sm(-))$, these maps induce the expected isomorphisms corresponding to planar isotopy.
\end{theorem}

\begin{figure}[ht]
    \centering
    \begin{tikzcd}
\begingroup%
  \makeatletter%
  \providecommand\color[2][]{%
    \errmessage{(Inkscape) Color is used for the text in Inkscape, but the package 'color.sty' is not loaded}%
    \renewcommand\color[2][]{}%
  }%
  \providecommand\transparent[1]{%
    \errmessage{(Inkscape) Transparency is used (non-zero) for the text in Inkscape, but the package 'transparent.sty' is not loaded}%
    \renewcommand\transparent[1]{}%
  }%
  \providecommand\rotatebox[2]{#2}%
  \newcommand*\fsize{\dimexpr\f@size pt\relax}%
  \newcommand*\lineheight[1]{\fontsize{\fsize}{#1\fsize}\selectfont}%
  \ifx\svgwidth\undefined%
    \setlength{\unitlength}{72bp}%
    \ifx\svgscale\undefined%
      \relax%
    \else%
      \setlength{\unitlength}{\unitlength * \real{\svgscale}}%
    \fi%
  \else%
    \setlength{\unitlength}{\svgwidth}%
  \fi%
  \global\let\svgwidth\undefined%
  \global\let\svgscale\undefined%
  \makeatother%
  \begin{picture}(1,1.39999998)%
    \lineheight{1}%
    \setlength\tabcolsep{0pt}%
    \put(0,0){\includegraphics[width=\unitlength,page=1]{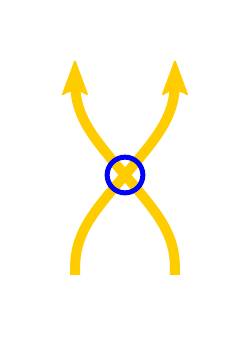}}%
    \put(0.27533335,1.19999999){\makebox(0,0)[lt]{\lineheight{1.25}\smash{\begin{tabular}[t]{l}1\end{tabular}}}}%
    \put(0.66844441,1.19999999){\makebox(0,0)[lt]{\lineheight{1.25}\smash{\begin{tabular}[t]{l}2\end{tabular}}}}%
    \put(0.2688888,0.1195553){\makebox(0,0)[lt]{\lineheight{1.25}\smash{\begin{tabular}[t]{l}3\end{tabular}}}}%
    \put(0.66844434,0.1195553){\makebox(0,0)[lt]{\lineheight{1.25}\smash{\begin{tabular}[t]{l}2\end{tabular}}}}%
  \end{picture}%
\endgroup%
 \ar[r,"\mu_I", start anchor=real east, end anchor=real west, shift left] & 
\begingroup%
  \makeatletter%
  \providecommand\color[2][]{%
    \errmessage{(Inkscape) Color is used for the text in Inkscape, but the package 'color.sty' is not loaded}%
    \renewcommand\color[2][]{}%
  }%
  \providecommand\transparent[1]{%
    \errmessage{(Inkscape) Transparency is used (non-zero) for the text in Inkscape, but the package 'transparent.sty' is not loaded}%
    \renewcommand\transparent[1]{}%
  }%
  \providecommand\rotatebox[2]{#2}%
  \newcommand*\fsize{\dimexpr\f@size pt\relax}%
  \newcommand*\lineheight[1]{\fontsize{\fsize}{#1\fsize}\selectfont}%
  \ifx\svgwidth\undefined%
    \setlength{\unitlength}{28.80000043bp}%
    \ifx\svgscale\undefined%
      \relax%
    \else%
      \setlength{\unitlength}{\unitlength * \real{\svgscale}}%
    \fi%
  \else%
    \setlength{\unitlength}{\svgwidth}%
  \fi%
  \global\let\svgwidth\undefined%
  \global\let\svgscale\undefined%
  \makeatother%
  \begin{picture}(1,3.49999989)%
    \lineheight{1}%
    \setlength\tabcolsep{0pt}%
    \put(0,0){\includegraphics[width=\unitlength,page=1]{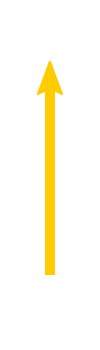}}%
    \put(0.43833314,0.30166626){\makebox(0,0)[lt]{\lineheight{1.25}\smash{\begin{tabular}[t]{l}1\end{tabular}}}}%
    \put(0.43833314,2.99999992){\makebox(0,0)[lt]{\lineheight{1.25}\smash{\begin{tabular}[t]{l}1\end{tabular}}}}%
  \end{picture}%
\endgroup%
 \ar[l, "\mu_I'", start anchor=real west, end anchor=real east, shift left]
    \end{tikzcd}
    \caption{An MOY I move.}
    \label{fig:moy-i}
\end{figure}
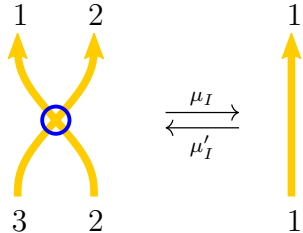

First, suppose $S$ and $S'$ are fully singular braid diagrams that differ by an MOY I move, as illustrated in \cref{fig:moy-i}. Specifically, $S$ contains a fixed singular vertex $v$ that meets the same edge twice. We would like to construct filtered chain maps $\muone:\C(S)\to \C(S')$ and $\muoneprime:\C(S')\to\C(S)$. To start, let us characterize $\C(S)$ and $\C(S')$. 

Without loss of generality, assume that the edge which is deleted by the MOY I move is to the right of the vertex. Label this edge with the variable $U_2$, label the top left edge $U_1$, and label the bottom left edge with $U_3$, again as in \cref{fig:moy-i}.

Let $R$ be the polynomial ring over all edges not shown in the local diagram; thus, $R(S') = R[U_1]$ and $R(S) = R(S')[U_2, U_3]$. We relate the associated quotient rings by the following proposition:

\begin{proposition}\label{prop:moy-i-prop}
As $R(S')$ modules, $Q(S') \iso Q(S)/(U_1 + U_2)$.
\end{proposition}

\begin{proof}
We expand the right-hand side as a quotient of a free $R(S')$-module:
\begin{align*}
    Q(S)/(U_1 + U_2)
    &\iso Q(S) \tensor_{R(S)} R(S)/(U_1 + U_2) \\
    &\iso R(S)/(L(S)+N(S))\otimes R(S)/(U_1+U_2)\\
    &\iso R(S)/(L(S)+N(S)+(U_1+U_2))\\
    &\iso R(S')[U_2,U_3]/(L(S)+N(S)+(U_1+U_2))\\
    &\iso R(S')[U_2,U_3]/(L(S)+\tilde{N}(S)+(U_2-U_3)+(U_1+U_2))\\
    &\iso R(S')/(L(S)+\tilde{N}(S)) \,.
\end{align*}

In the above, $\tilde{N}(S)$ is the sum of the non-local relations other than $U_1-U_3$; this is exactly equal to $N(S')$, as any region intersecting these local diagrams can be made to avoid $U_2$ and any intersections with $U_3$ can be isotoped to intersect $U_1$ instead. Further, $L(S)=L(S')$. Thus we have $Q(S)/(U_1+U_2)\iso R(S')/(L(S)+\tilde{N}(S))=R(S')/(L(S')+N(S'))=Q(S')$, as desired. 
\end{proof}

\begin{proposition}
The chain complexes $\C(S)$ and $\C(S')$ are quasi-isomorphic as complexes over $R(S')$. 
\end{proposition}
\begin{proof}
We can use \cref{prop:moy-i-prop} to expand $\C(S)$: 
\begin{align*}
    \C(S) &= Q(S) \tensor_{R(S)}\lsplus \\
    &= Q(S) \tensor \left( \shortfact{R(S)}{U_1-U_3}{R(S)}{U_1+2U_2+U_3}\tensor\lsplusdotdotdot\right) \\
    &\iso Q(S) \tensor  \shortfact{R(S)}{0}{R(S)}{2U_1+2U_2}\tensor\lsplusdotdotdot
    && \text{(using relation $U_1-U_3$ in $N(S)$)} \\
    &\homotopic Q(S) \tensor R(S)/(U_1 + U_2)\tensor\lsplusdotdotdot
    && \text{(replacing $2U_1 + 2U_2$ by the cokernel)} \\
    &\iso \left(Q(S) \tensor R(S)/(U_1 + U_2)\right) \tensor \lsplusdotdotdot \\
    &\iso Q(S') \tensor_{R(S')} \lsplusprime && \text{(by \cref{prop:moy-i-prop})} \\
    &= \C(S') \,.
\end{align*}
In the above, let $\lsplusdotdotdot=\bigtensor_{w \in \Fixed(D)\setminus\{v\}} \shortfact{R(D)}{L(w)}{R(D)}{L^+(w)}$, and note $\lsplusdotdotdot=\lsplusprime$. Note that we may replace the mapping cone of $2U_1 + 2U_2$ by its cokernel in the fourth line only after checking that $2U_1 + 2U_2$ is not a zero divisor in $Q(S)$; by the logic in the proof of \cref{prop:moy-i-prop}, we may choose a generating set of relations for $N(S) + L(S)$, none of which contain a term with a nonzero power of $U_2$. Therefore, $Q(S)$ is isomorphic to a free polynomial ring over $U_2$; since $2U_1 + 2U_2$ is a unit multiple (over $\Q$) of a monic polynomial in $U_2$, we therefore get that it is not a zero divisor in $Q(S)$.
\end{proof}

Let $\muone: \C(S) \to \C(S')$ be the quotient map implied by the above calculations. Explicitly on a simple tensor, $\muone([r]\tensor (a,b)\tensor \tilde{s})=[rb]\tensor \tilde{s}$. Let $\muoneprime: \C(S') \to \C(S)$ be the splitting of $\muone$ given by inclusion into the first $R(S)$ summand in the equivalence of $R(S)/(U_1+U_2)$ and $\shortfact{R(S)}{0}{R(S)}{2U_1+2U_2}$ in the above proof. Explicitly on a simple tensor, $\muoneprime([r]\tensor \tilde{s})=[r]\tensor(0,1)\tensor\tilde{s}$. 

For partially singular braid diagrams $D$ and $D'$ related by an MOY I move, extend both maps to the cube of resolutions by defining $\muone: \C(D_I) \to \C(D'_I)$ and $\muoneprime: \C(D'_I) \to \C(D_I)$ as above for each $I \in \{0,1\}^{c(D)}$.

\begin{proof}[Proof of \cref{thm:moy-i}]
It is clear that $\muone$ and $\muoneprime$ are filtered maps, since they are defined component-wise on the cube of resolutions.

We need to check that $\muone$ and $\muoneprime$ are chain maps, i.e.\ that they commute with the edge map $d_1$. Let $I,J \in \{0,1\}^{c(D)}$ with $I \coveredby J$. If $I$ and $J$ differ at a positive crossing, then $d_{I,J}$ is given by $\oldphi_+ \tensor \ldplus = 1 \tensor \ldplus$. Otherwise, $d_{I,J}$ is given by $\oldphi_- \tensor \ldplus = (U_b - U_c) \tensor \ldplus$. Either way, the edge maps are given by multiplication by an element of $R(D')$. Since $\mutwo$ and $\mutwoprime$ were defined to be $R(D')$-linear, we get that they commute with $d_1$.
\end{proof}

\subsection{MOY II}\label{sec:moy-ii}

Suppose $D$ and $D'$ are partially singular braid diagrams with $D'$ the result of applying an MOY II move to $D$ and reducing the number of crossings, as shown in \cref{fig:moy-ii}. In words, $D$ contains a free vertex $v_1$, a fixed vertex $v_2$, and two edges $e_5$ and $e_6$ from $v_2$ to $v_1$. The diagram $D'$ is obtained from $D$ by removing $e_5$ and $e_6$ and merging $v_1$ and $v_2$ into a single fixed vertex.

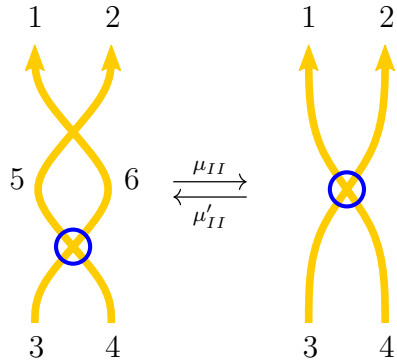
\begin{figure}[ht]
    \centering
    \begin{tikzcd}
\begingroup%
  \makeatletter%
  \providecommand\color[2][]{%
    \errmessage{(Inkscape) Color is used for the text in Inkscape, but the package 'color.sty' is not loaded}%
    \renewcommand\color[2][]{}%
  }%
  \providecommand\transparent[1]{%
    \errmessage{(Inkscape) Transparency is used (non-zero) for the text in Inkscape, but the package 'transparent.sty' is not loaded}%
    \renewcommand\transparent[1]{}%
  }%
  \providecommand\rotatebox[2]{#2}%
  \newcommand*\fsize{\dimexpr\f@size pt\relax}%
  \newcommand*\lineheight[1]{\fontsize{\fsize}{#1\fsize}\selectfont}%
  \ifx\svgwidth\undefined%
    \setlength{\unitlength}{64.79999828bp}%
    \ifx\svgscale\undefined%
      \relax%
    \else%
      \setlength{\unitlength}{\unitlength * \real{\svgscale}}%
    \fi%
  \else%
    \setlength{\unitlength}{\svgwidth}%
  \fi%
  \global\let\svgwidth\undefined%
  \global\let\svgscale\undefined%
  \makeatother%
  \begin{picture}(1,2.22222228)%
    \lineheight{1}%
    \setlength\tabcolsep{0pt}%
    \put(0,0){\includegraphics[width=\unitlength,page=1]{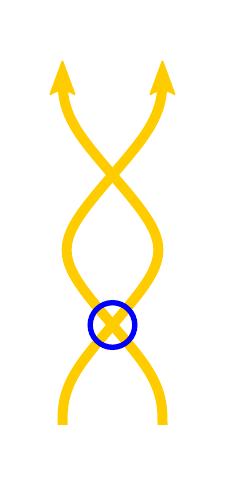}}%
    \put(0.23395047,2.05555564){\makebox(0,0)[lt]{\lineheight{1.25}\smash{\begin{tabular}[t]{l}1\end{tabular}}}}%
    \put(0.68716029,2.05555564){\makebox(0,0)[lt]{\lineheight{1.25}\smash{\begin{tabular}[t]{l}2\end{tabular}}}}%
    \put(0.13098754,1.11234582){\makebox(0,0)[lt]{\lineheight{1.25}\smash{\begin{tabular}[t]{l}5\end{tabular}}}}%
    \put(0.79753069,1.11234582){\makebox(0,0)[lt]{\lineheight{1.25}\smash{\begin{tabular}[t]{l}6\end{tabular}}}}%
    \put(0.24320972,0.1328394){\makebox(0,0)[lt]{\lineheight{1.25}\smash{\begin{tabular}[t]{l}3\end{tabular}}}}%
    \put(0.68685166,0.13358014){\makebox(0,0)[lt]{\lineheight{1.25}\smash{\begin{tabular}[t]{l}4\end{tabular}}}}%
  \end{picture}%
\endgroup%
 \ar[r,"\mu_{II}", start anchor=real east, end anchor=real west, shift left] & 
\begingroup%
  \makeatletter%
  \providecommand\color[2][]{%
    \errmessage{(Inkscape) Color is used for the text in Inkscape, but the package 'color.sty' is not loaded}%
    \renewcommand\color[2][]{}%
  }%
  \providecommand\transparent[1]{%
    \errmessage{(Inkscape) Transparency is used (non-zero) for the text in Inkscape, but the package 'transparent.sty' is not loaded}%
    \renewcommand\transparent[1]{}%
  }%
  \providecommand\rotatebox[2]{#2}%
  \newcommand*\fsize{\dimexpr\f@size pt\relax}%
  \newcommand*\lineheight[1]{\fontsize{\fsize}{#1\fsize}\selectfont}%
  \ifx\svgwidth\undefined%
    \setlength{\unitlength}{64.79999828bp}%
    \ifx\svgscale\undefined%
      \relax%
    \else%
      \setlength{\unitlength}{\unitlength * \real{\svgscale}}%
    \fi%
  \else%
    \setlength{\unitlength}{\svgwidth}%
  \fi%
  \global\let\svgwidth\undefined%
  \global\let\svgscale\undefined%
  \makeatother%
  \begin{picture}(1,2.22222228)%
    \lineheight{1}%
    \setlength\tabcolsep{0pt}%
    \put(0.23395047,2.05555564){\makebox(0,0)[lt]{\lineheight{1.25}\smash{\begin{tabular}[t]{l}1\end{tabular}}}}%
    \put(0.68716029,2.05555564){\makebox(0,0)[lt]{\lineheight{1.25}\smash{\begin{tabular}[t]{l}2\end{tabular}}}}%
    \put(0.24320972,0.1328394){\makebox(0,0)[lt]{\lineheight{1.25}\smash{\begin{tabular}[t]{l}3\end{tabular}}}}%
    \put(0.68685166,0.13358014){\makebox(0,0)[lt]{\lineheight{1.25}\smash{\begin{tabular}[t]{l}4\end{tabular}}}}%
    \put(0,0){\includegraphics[width=\unitlength,page=1]{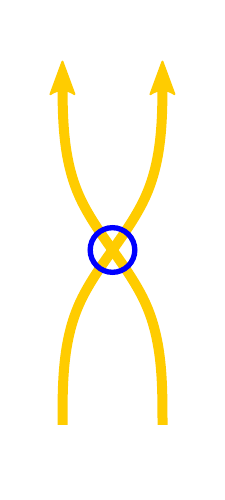}}%
  \end{picture}%
\endgroup%
 \ar[l, "\mu_{II}'", start anchor=real west, end anchor=real east, shift left]
    \end{tikzcd}
    \caption{An MOY II move.}
    \label{fig:moy-ii}
\end{figure}

\begin{theorem}\label{thm:moy-ii}
There exists a direct sum decomposition $\C(D) \iso \C(D') \directsum \C(D')$ as filtered chain complexes over $R(D')$. Define $\mu_{II}: \C(D) \to \C(D')$ to be projection onto the second summand, and define $\mu_{II}': \C(D') \to \C(D)$ to be inclusion into the first summand. Under the identification $E_1(\C(-)) \iso \CKh^-(\sm(-))$, the maps $\mutwo$ and $\mutwoprime$ induce the maps on Khovanov homology corresponding to the cobordisms which delete and introduce a circle, respectively.
\end{theorem}

To start, let $S$ and $S'$ be fully singular braid diagrams again with $S'$ the result of applying an MOY II move to $S$ reducing the number of crossings. Let $R$ be the polynomial ring over all edges not shown in the local diagrams, so that $R(S') = R[U_1, U_2, U_3, U_4]$, and $R(S) = R(S')[U_5, U_6]$. 

\begin{proposition}
As an $R(S')$-module, $Q(S) \iso Q(S')\spannedby{1} \directsum Q(S')\spannedby{U_6}$. 
\end{proposition}
\begin{proof}
First, note that $Q(S) = Q(S')[U_5, U_6]/(U_5 + U_6 - U_1 - U_2, U_5U_6 - U_1U_2)$. We do not need to consider any other non-local relations, as any region $\Omega$ intersecting these diagrams can be isotoped away from $U_5$ and $U_6$ to give an equivalent or stronger relation. We want to prove that $\{1, U_6\}$ is a basis for $Q(S)$ over $Q(S')$. To see that $\{1,U_6\}$ is a generating set, it is enough to note that in $Q(S)$, $U_5=U_1+U_2-U_6$, and that $(U_1+U_2-U_6)U_6-U_1U_2=0$, so $U_6^2=(U_1+U_2)U_6-U_3U_4$. Linear independence follows from the fact that $U_6^2-(U_1 + U_2)U_6 + U_1U_2$ is a monic polynomial of degree 2 in $U_6$.
\end{proof}

Using this proposition, we can decompose 
\begin{align*}
    \C(S)
    &\iso Q(S) \tensor_{R(S)} \lsplus \\
    &\iso Q(S) \tensor_{R(S)} \left( \bigtensor_{v \in \Fixed(S)} \shortfact{R(S)}{L(v)}{R(S)}{L^+(v)} \right) \\
    &\iso Q(S) \tensor_{R(S)} \left( \bigtensor_{v \in \Fixed(S')} \shortfact{R(S)}{L(v)}{R(S)}{L^+(v)} \right) \\
    &\iso \bigtensor_{v \in \Fixed(S')} \shortfact{Q(S)}{L(v)}{Q(S)}{L^+(v)} \\
    &\iso \bigtensor_{v \in \Fixed(S')} \shortfact{Q(S')\spannedby{1} \directsum Q(S')\spannedby{U_6}}{L(v)}{Q(S')\spannedby{1} \directsum Q(S')\spannedby{U_6}}{L^+(v)} \\
    &\iso \left( \bigtensor_{v \in \Fixed(S')} \shortfact{Q(S')\spannedby{1}}{L(v)}{Q(S')\spannedby{1}}{L^+(v)} \right) \directsum \left( \bigtensor_{v \in \Fixed(S')} \shortfact{Q(S')\spannedby{U_6}}{L(v)}{Q(S')\spannedby{U_6}}{L^+(v)} \right) \\
    &\iso \left(Q(S')\spannedby{1} \tensor_{R(S')} \lsplusprime\right) \directsum \left(Q(S')\spannedby{U_6} \tensor_{R(S')} \lsplusprime\right) \\
    &\iso \C(S')\spannedby{1} \directsum \C(S')\spannedby{U_6} \,.
\end{align*}

Define $\mutwo: \C(S) \to \C(S')$ to be projection onto the second summand in the above decomposition, and define $\mutwoprime: \C(S') \to \C(S)$ to be inclusion into the first summand. For partially singular braid diagrams $D$ and $D'$ related by an MOY II move, extend both maps to the cube of resolutions by defining $\mutwo: \C(D_I) \to \C(D'_I)$ and $\mutwoprime: \C(D'_I) \to \C(D_I)$ as above for each $I \in \{0,1\}^{c(D)}$.

\begin{proof}[Proof of \cref{thm:moy-ii}]
It is clear that $\mutwo$ and $\mutwoprime$ are filtered maps, since they are defined component-wise on the cube of resolutions. Next, we need to check that $\mutwo$ and $\mutwoprime$ are chain maps, i.e.\ that they commute with the edge map $d_1$. Let $I,J \in \{0,1\}^{c(D)}$ with $I \coveredby J$. If $I$ and $J$ differ at a positive crossing, then $d_{I,J}$ is given by $\oldphi_+ \tensor \ldplus = 1 \tensor \ldplus$. Otherwise, $d_{I,J}$ is given by $\oldphi_- \tensor \ldplus = (U_b - U_c) \tensor \ldplus$. Either way, the edge maps are given by multiplication by an element of $R(D')$. Since $\mutwo$ and $\mutwoprime$ were defined to be $R(D')$-linear, we get that they commute with $d_1$.

We used a direct sum decomposition of $\C(S)$ to define these maps on complete resolutions. We can see this direct sum decomposition on the cube of resolutions as well. Specifically, we have a split exact sequence:
\begin{equation*}
    \xymatrix{0\ar[r]&
    \C(D')\ar@<1ex>[r]^{\mutwoprime}&
    \C(D)\ar@<1ex>[r]^{\mutwo}\ar@<1ex>[l]^{1 \tensor \ldplus}&
    \C(D')\ar[r]\ar@<1ex>[l]^{(U_6 - U_1) \tensor \ldplus}&
    0} \nospaceperiod
\end{equation*}

Finally, we want to show that these maps induce the correct morphisms on the Khovanov complex. The cobordism corresponding to the introduction of a circle is induced by multiplication by $1$ \cite{barnatan2005}. This should correspond to $\mutwoprime$, which we can see also induces multiplication by $1$ on homology. The cobordism corresponding to the deletion of a circle should send $1 \mapsto 0$ and $X \mapsto 1$, where $X$ is a variable associated to the shrinking circle. In our case, $\mutwo$ maps $1 \mapsto 0$ and $U_6 \mapsto 1$, inducing this same map on homology.
\end{proof}

One can repeat the same argument to show that we also have similar MOY II decompositions for the cases in \cref{fig:moy-ii-variations}.

\begin{figure}[ht]
    \centering
    \begin{tikzcd}
\begingroup%
  \makeatletter%
  \providecommand\color[2][]{%
    \errmessage{(Inkscape) Color is used for the text in Inkscape, but the package 'color.sty' is not loaded}%
    \renewcommand\color[2][]{}%
  }%
  \providecommand\transparent[1]{%
    \errmessage{(Inkscape) Transparency is used (non-zero) for the text in Inkscape, but the package 'transparent.sty' is not loaded}%
    \renewcommand\transparent[1]{}%
  }%
  \providecommand\rotatebox[2]{#2}%
  \newcommand*\fsize{\dimexpr\f@size pt\relax}%
  \newcommand*\lineheight[1]{\fontsize{\fsize}{#1\fsize}\selectfont}%
  \ifx\svgwidth\undefined%
    \setlength{\unitlength}{64.79999828bp}%
    \ifx\svgscale\undefined%
      \relax%
    \else%
      \setlength{\unitlength}{\unitlength * \real{\svgscale}}%
    \fi%
  \else%
    \setlength{\unitlength}{\svgwidth}%
  \fi%
  \global\let\svgwidth\undefined%
  \global\let\svgscale\undefined%
  \makeatother%
  \begin{picture}(1,2.22222228)%
    \lineheight{1}%
    \setlength\tabcolsep{0pt}%
    \put(0,0){\includegraphics[width=\unitlength,page=1]{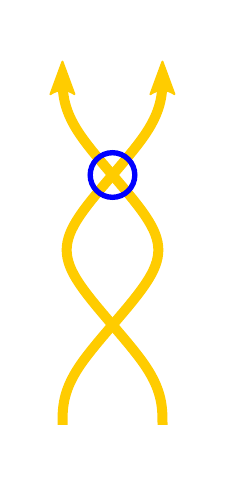}}%
    \put(0.23395047,2.05555564){\makebox(0,0)[lt]{\lineheight{1.25}\smash{\begin{tabular}[t]{l}1\end{tabular}}}}%
    \put(0.68716029,2.05555564){\makebox(0,0)[lt]{\lineheight{1.25}\smash{\begin{tabular}[t]{l}2\end{tabular}}}}%
    \put(0.13098754,1.11234582){\makebox(0,0)[lt]{\lineheight{1.25}\smash{\begin{tabular}[t]{l}5\end{tabular}}}}%
    \put(0.79753069,1.11234582){\makebox(0,0)[lt]{\lineheight{1.25}\smash{\begin{tabular}[t]{l}6\end{tabular}}}}%
    \put(0.24320972,0.1328394){\makebox(0,0)[lt]{\lineheight{1.25}\smash{\begin{tabular}[t]{l}3\end{tabular}}}}%
    \put(0.68685166,0.13358014){\makebox(0,0)[lt]{\lineheight{1.25}\smash{\begin{tabular}[t]{l}4\end{tabular}}}}%
  \end{picture}%
\endgroup%
 \ar[r,"\mu_{II}", start anchor=real east, end anchor=real west, shift left] &  \ar[l, "\mu_{II}'", start anchor=real west, end anchor=real east, shift left] & 
\begingroup%
  \makeatletter%
  \providecommand\color[2][]{%
    \errmessage{(Inkscape) Color is used for the text in Inkscape, but the package 'color.sty' is not loaded}%
    \renewcommand\color[2][]{}%
  }%
  \providecommand\transparent[1]{%
    \errmessage{(Inkscape) Transparency is used (non-zero) for the text in Inkscape, but the package 'transparent.sty' is not loaded}%
    \renewcommand\transparent[1]{}%
  }%
  \providecommand\rotatebox[2]{#2}%
  \newcommand*\fsize{\dimexpr\f@size pt\relax}%
  \newcommand*\lineheight[1]{\fontsize{\fsize}{#1\fsize}\selectfont}%
  \ifx\svgwidth\undefined%
    \setlength{\unitlength}{64.79999828bp}%
    \ifx\svgscale\undefined%
      \relax%
    \else%
      \setlength{\unitlength}{\unitlength * \real{\svgscale}}%
    \fi%
  \else%
    \setlength{\unitlength}{\svgwidth}%
  \fi%
  \global\let\svgwidth\undefined%
  \global\let\svgscale\undefined%
  \makeatother%
  \begin{picture}(1,2.22222228)%
    \lineheight{1}%
    \setlength\tabcolsep{0pt}%
    \put(0,0){\includegraphics[width=\unitlength,page=1]{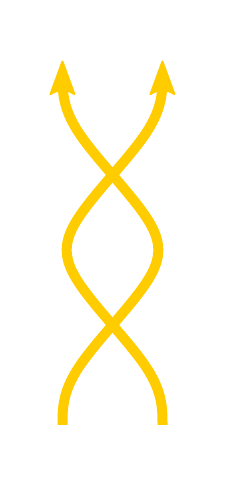}}%
    \put(0.23395047,2.05555564){\makebox(0,0)[lt]{\lineheight{1.25}\smash{\begin{tabular}[t]{l}1\end{tabular}}}}%
    \put(0.68716029,2.05555564){\makebox(0,0)[lt]{\lineheight{1.25}\smash{\begin{tabular}[t]{l}2\end{tabular}}}}%
    \put(0.13098754,1.11234582){\makebox(0,0)[lt]{\lineheight{1.25}\smash{\begin{tabular}[t]{l}5\end{tabular}}}}%
    \put(0.79753069,1.11234582){\makebox(0,0)[lt]{\lineheight{1.25}\smash{\begin{tabular}[t]{l}6\end{tabular}}}}%
    \put(0.24320972,0.1328394){\makebox(0,0)[lt]{\lineheight{1.25}\smash{\begin{tabular}[t]{l}3\end{tabular}}}}%
    \put(0.68685166,0.13358014){\makebox(0,0)[lt]{\lineheight{1.25}\smash{\begin{tabular}[t]{l}4\end{tabular}}}}%
  \end{picture}%
\endgroup%
 \ar[r,"\mu_{II}", start anchor=real east, end anchor=real west, shift left] & 
\begingroup%
  \makeatletter%
  \providecommand\color[2][]{%
    \errmessage{(Inkscape) Color is used for the text in Inkscape, but the package 'color.sty' is not loaded}%
    \renewcommand\color[2][]{}%
  }%
  \providecommand\transparent[1]{%
    \errmessage{(Inkscape) Transparency is used (non-zero) for the text in Inkscape, but the package 'transparent.sty' is not loaded}%
    \renewcommand\transparent[1]{}%
  }%
  \providecommand\rotatebox[2]{#2}%
  \newcommand*\fsize{\dimexpr\f@size pt\relax}%
  \newcommand*\lineheight[1]{\fontsize{\fsize}{#1\fsize}\selectfont}%
  \ifx\svgwidth\undefined%
    \setlength{\unitlength}{64.79999828bp}%
    \ifx\svgscale\undefined%
      \relax%
    \else%
      \setlength{\unitlength}{\unitlength * \real{\svgscale}}%
    \fi%
  \else%
    \setlength{\unitlength}{\svgwidth}%
  \fi%
  \global\let\svgwidth\undefined%
  \global\let\svgscale\undefined%
  \makeatother%
  \begin{picture}(1,2.22222228)%
    \lineheight{1}%
    \setlength\tabcolsep{0pt}%
    \put(0.23395047,2.05555564){\makebox(0,0)[lt]{\lineheight{1.25}\smash{\begin{tabular}[t]{l}1\end{tabular}}}}%
    \put(0.68716029,2.05555564){\makebox(0,0)[lt]{\lineheight{1.25}\smash{\begin{tabular}[t]{l}2\end{tabular}}}}%
    \put(0.24320972,0.1328394){\makebox(0,0)[lt]{\lineheight{1.25}\smash{\begin{tabular}[t]{l}3\end{tabular}}}}%
    \put(0.68685166,0.13358014){\makebox(0,0)[lt]{\lineheight{1.25}\smash{\begin{tabular}[t]{l}4\end{tabular}}}}%
    \put(0,0){\includegraphics[width=\unitlength,page=1]{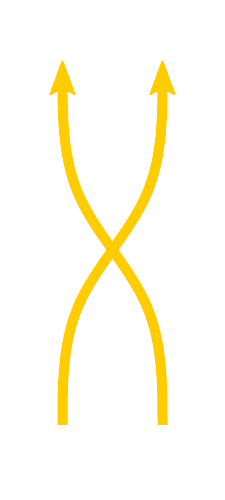}}%
  \end{picture}%
\endgroup%
 \ar[l, "\mu_{II}'", start anchor=real west, end anchor=real east, shift left]
    \end{tikzcd}
    \caption{Variations of the MOY II move.}
    \label{fig:moy-ii-variations}
\end{figure}
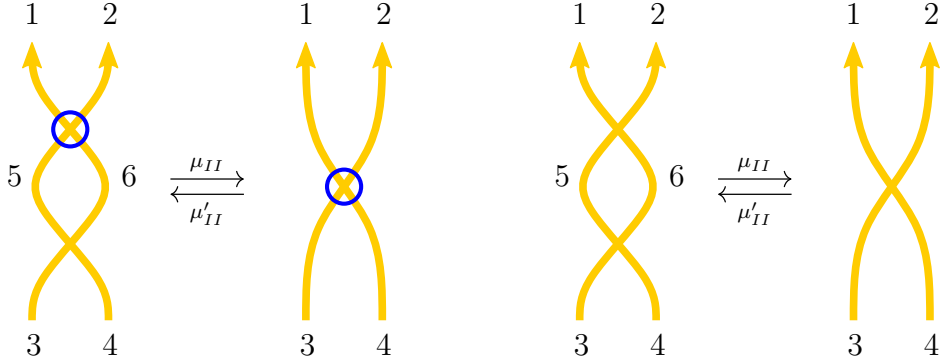

\subsection{MOY III}\label{sec:moy-iii}

Suppose $D$ and $D'$ are fully singular braid diagrams with $D'$ the result of applying an MOY III move to $D$ and reducing the number of crossings, as shown in \cref{fig:moy-iii}. In words, $D$ contains a fixed vertex $v_1$, free vertices $v_2$ and $v_3$, and edges $e_7: v_2 \to v_1$, $e_8: v_3 \to v_1$, and $e_9: v_3 \to v_2$. The diagram $D'$ is obtained from $D$ by removing the edges $e_7$, $e_8$, and $e_9$, merging $v_1$ and $v_3$ into a single fixed vertex, removing $v_2$, and merging $e_6$ into $e_3$.

\begin{figure}[ht]
    \centering
    \begin{tikzcd}
    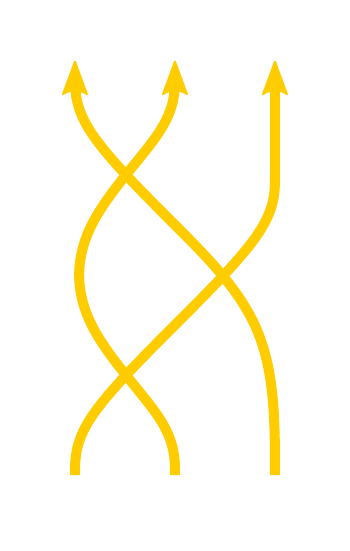 \ar[r,"\mu_{III}", start anchor=real east, end anchor=real west, shift left] & 
\begingroup%
  \makeatletter%
  \providecommand\color[2][]{%
    \errmessage{(Inkscape) Color is used for the text in Inkscape, but the package 'color.sty' is not loaded}%
    \renewcommand\color[2][]{}%
  }%
  \providecommand\transparent[1]{%
    \errmessage{(Inkscape) Transparency is used (non-zero) for the text in Inkscape, but the package 'transparent.sty' is not loaded}%
    \renewcommand\transparent[1]{}%
  }%
  \providecommand\rotatebox[2]{#2}%
  \newcommand*\fsize{\dimexpr\f@size pt\relax}%
  \newcommand*\lineheight[1]{\fontsize{\fsize}{#1\fsize}\selectfont}%
  \ifx\svgwidth\undefined%
    \setlength{\unitlength}{100.79999828bp}%
    \ifx\svgscale\undefined%
      \relax%
    \else%
      \setlength{\unitlength}{\unitlength * \real{\svgscale}}%
    \fi%
  \else%
    \setlength{\unitlength}{\svgwidth}%
  \fi%
  \global\let\svgwidth\undefined%
  \global\let\svgscale\undefined%
  \makeatother%
  \begin{picture}(1,1.57142863)%
    \lineheight{1}%
    \setlength\tabcolsep{0pt}%
    \put(0,0){\includegraphics[width=\unitlength,page=1]{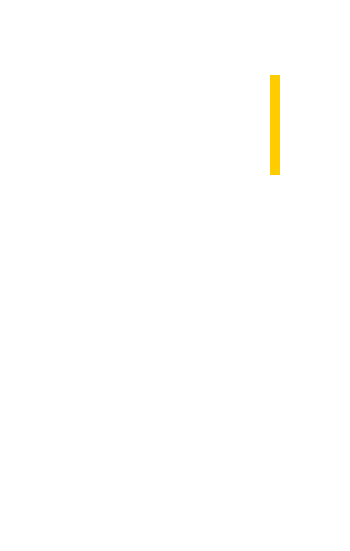}}%
    \put(0.19666662,1.42857149){\makebox(0,0)[lt]{\lineheight{1.25}\smash{\begin{tabular}[t]{l}1\end{tabular}}}}%
    \put(0.47746026,1.42857149){\makebox(0,0)[lt]{\lineheight{1.25}\smash{\begin{tabular}[t]{l}2\end{tabular}}}}%
    \put(0.76349202,1.42936514){\makebox(0,0)[lt]{\lineheight{1.25}\smash{\begin{tabular}[t]{l}3\end{tabular}}}}%
    \put(0.76349202,0.08539671){\makebox(0,0)[lt]{\lineheight{1.25}\smash{\begin{tabular}[t]{l}3\end{tabular}}}}%
    \put(0.19154755,0.0858729){\makebox(0,0)[lt]{\lineheight{1.25}\smash{\begin{tabular}[t]{l}4\end{tabular}}}}%
    \put(0.47706345,0.08619033){\makebox(0,0)[lt]{\lineheight{1.25}\smash{\begin{tabular}[t]{l}5\end{tabular}}}}%
    \put(0,0){\includegraphics[width=\unitlength,page=2]{moy-iii-prime.pdf}}%
  \end{picture}%
\endgroup%
 \ar[l, "\mu_{III}'", start anchor=real west, end anchor=real east, shift left]
    \end{tikzcd}
    \caption{An MOY III move.}
    \label{fig:moy-iii}
\end{figure}

\begin{theorem}\label{thm:moy-iii}
There exist $R(D')$-linear filtered quasi-isomorphisms $\muthree: \C(D) \to \C(D')$ and $\muthreeprime: \C(D') \to \C(D)$. Furthermore, $\C(D')$ is isomorphic to a direct summand of $\C(D)$. Under the identification $E_1(\C(-)) \iso \CKh^-(\sm(-))$, these maps induce the expected isomorphisms corresponding to planar isotopy.
\end{theorem}

As for MOY I and II, we will again start by defining these maps on fully singular braid diagrams then extending them to the cube of resolutions. Let $S$ and $S'$ be fully singular braid diagrams with $S'$ the result of applying an MOY III move to $S$ reducing the number of crossings as in \cref{fig:moy-iii}. We will prove that $\C(S)\cong \C(S')\oplus\Upsilon_L$, where $\Upsilon_L$ is some acyclic complex. Furthermore, the MOY III move has a nontrivial horizontal mirroring. We will prove that in the case where $S$ and $S'$ are connected by an MOY III move which is the mirror of \cref{fig:moy-iii}, we have $\C(S)\cong\C(S')\oplus\Upsilon_R$. While it is true that $\Upsilon_R=\Upsilon_L$, we will neither need this fact nor prove it in this paper. Nevertheless, we may refer to the complex as $\Upsilon$ regardless.

\begin{figure}[ht]
    \centering
    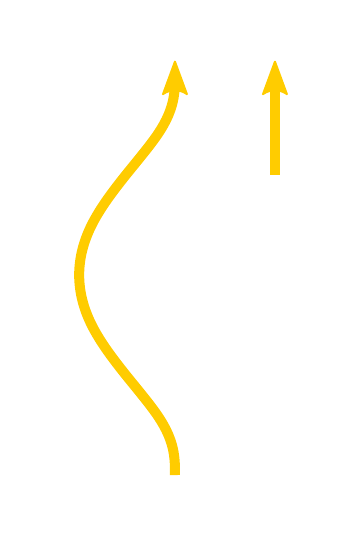
    \caption{The fully singular diagram $S''$ used in the definitions of $\muthree$ and $\muthreeprime$. }
    \label{fig:moy-iii-prime-prime}
\end{figure}

We will construct a map $\muthree:\C(S')\to\C(S)$ and another map $\muthreeprime:\C(S)\to\C(S')$ which splits $\muthree$, thus proving that $\C(S')$ is a direct summand of $\C(S)$. Let $S''$ be the fully singular braid diagram where the middle singular vertex $v_2$ is replaced by the oriented smoothing as in \cref{fig:moy-iii-prime-prime}, so that we may define the map $1 \tensor \lsplus: \C(S) \to \C(S'')$. We may then apply an MOY II move $\mutwo$ to the left two strands in $S''$ to get a map $\mu_{II}: \C(S'') \to \C(S')$. Therefore, we define $\muthree = \mutwo \comp (1 \tensor \lsplus)$. We can also reverse the order of these operations to define $\muthreeprime = ((U_9 - U_3) \tensor \ldplus) \comp \mutwoprime$. Note that the maps $1 \tensor \lsplus$ and $(U_9-U_3)\tensor \lsplus$ are well-defined since if $v_2$ were replaced by a positive or negative crossing, these would simply be multiples of the edge maps corresponding to resolutions of that crossing.

\begin{proposition}\label{prop:moy-iii-splits}
$\muthree$ splits $\muthreeprime$, i.e.\ $\muthree \comp \muthreeprime = \id_{\C(S')}$
\end{proposition}

\begin{proof}
We expand out the definitions of $\muthree$ and $\muthreeprime$ to get
\begin{align*}
    \muthree \comp \muthreeprime &= \mutwo \comp (1 \tensor \lsplus) \comp ((U_9 - U_3) \tensor \lsplus) \comp \mutwoprime \\
    &= \mutwo \comp ((U_9 - U_3) \tensor \lsplus) \comp \mutwoprime \\
    &= \pmat{0 & 1} \pmat{-U_3 & -U_4U_5 \\ 1 & U_4 + U_5 - U_3} \pmat{1 \\ 0} \tensor \lsplus \\
    &= \pmat{1} \tensor \lsplus \\
    &= \id_{\C(S')} \,. \qedhere
\end{align*}
\end{proof}

Therefore, we get a direct sum decomposition $\C(S) \iso \Upsilon\spannedby{1} \directsum \C(S')\spannedby{U_9 - U_3}$. For partially singular braid diagrams $D$ and $D'$ related by an MOY III move, extend both maps to the cube of resolutions by defining $\muthree: \C(D_I) \to \C(D'_I)$ and $\muthreeprime: \C(D'_I) \to \C(D_I)$ as above for each $I \in \{0,1\}^{c(D)}$.

\begin{proof}[Proof of \cref{thm:moy-iii}]
It is clear that $\muthree$ and $\muthreeprime$ are filtered maps, since they are defined component-wise on the cube of resolutions. Furthermore, we can extend our proof of \cref{prop:moy-iii-splits} to partially-singular braid diagrams since both the edge maps and MOY II maps are defined on such diagrams, so $\C(D')$ really is a summand of $\C(D)$.

We also need to check that $\muthree$ and $\muthreeprime$ are chain maps, i.e.\ that they commute with the edge maps $d_1$. Let $I,J \in \{0,1\}^{c(D)}$ with $I \coveredby J$. If $I$ and $J$ differ at a positive crossing, then $d_{I,J}$ is given by $\oldphi_+ \tensor \ldplus = 1 \tensor \ldplus$. Otherwise, $d_{I,J}$ is given by $\oldphi_- \tensor \ldplus = (U_b - U_c) \tensor \ldplus$. Either way, the edge maps are given by multiplication by an element of $R(D')$. Since $\muthree$ and $\muthreeprime$ were defined to be $R(D')$-linear, we get that they commute with $d_1$.
\end{proof}

As before, a similar argument shows that we also have MOY III decompositions for the cases in \cref{fig:moy-iii-variations}.

\begin{figure}[ht]
    \centering
    \begin{tikzcd}[column sep=small]
    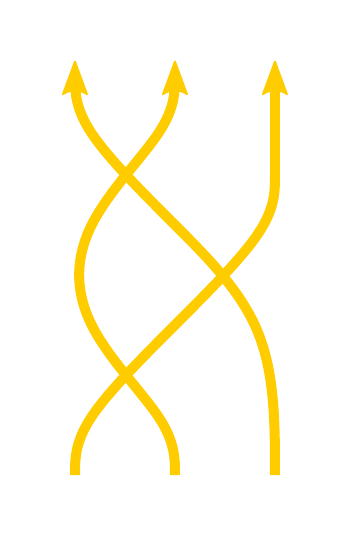 \ar[r,"\mu_{III}", start anchor=real east, end anchor=real west, shift left] &  \ar[l, "\mu_{III}'", start anchor=real west, end anchor=real east, shift left] &[-25pt] 
\begingroup%
  \makeatletter%
  \providecommand\color[2][]{%
    \errmessage{(Inkscape) Color is used for the text in Inkscape, but the package 'color.sty' is not loaded}%
    \renewcommand\color[2][]{}%
  }%
  \providecommand\transparent[1]{%
    \errmessage{(Inkscape) Transparency is used (non-zero) for the text in Inkscape, but the package 'transparent.sty' is not loaded}%
    \renewcommand\transparent[1]{}%
  }%
  \providecommand\rotatebox[2]{#2}%
  \newcommand*\fsize{\dimexpr\f@size pt\relax}%
  \newcommand*\lineheight[1]{\fontsize{\fsize}{#1\fsize}\selectfont}%
  \ifx\svgwidth\undefined%
    \setlength{\unitlength}{100.79999828bp}%
    \ifx\svgscale\undefined%
      \relax%
    \else%
      \setlength{\unitlength}{\unitlength * \real{\svgscale}}%
    \fi%
  \else%
    \setlength{\unitlength}{\svgwidth}%
  \fi%
  \global\let\svgwidth\undefined%
  \global\let\svgscale\undefined%
  \makeatother%
  \begin{picture}(1,1.57142863)%
    \lineheight{1}%
    \setlength\tabcolsep{0pt}%
    \put(0,0){\includegraphics[width=\unitlength,page=1]{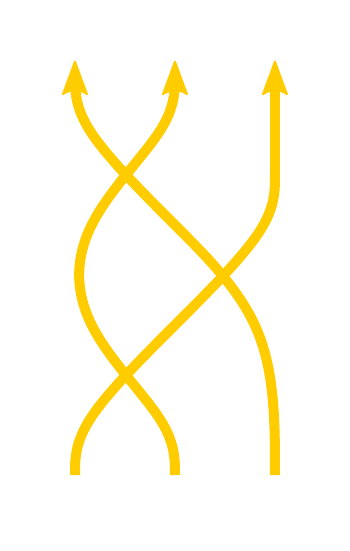}}%
    \put(0.19666662,1.42857149){\makebox(0,0)[lt]{\lineheight{1.25}\smash{\begin{tabular}[t]{l}1\end{tabular}}}}%
    \put(0.47746026,1.42857149){\makebox(0,0)[lt]{\lineheight{1.25}\smash{\begin{tabular}[t]{l}2\end{tabular}}}}%
    \put(0.76349202,1.42936514){\makebox(0,0)[lt]{\lineheight{1.25}\smash{\begin{tabular}[t]{l}3\end{tabular}}}}%
    \put(0.19154755,0.0858729){\makebox(0,0)[lt]{\lineheight{1.25}\smash{\begin{tabular}[t]{l}4\end{tabular}}}}%
    \put(0.47706345,0.08619033){\makebox(0,0)[lt]{\lineheight{1.25}\smash{\begin{tabular}[t]{l}5\end{tabular}}}}%
    \put(0.76269834,0.08539671){\makebox(0,0)[lt]{\lineheight{1.25}\smash{\begin{tabular}[t]{l}6\end{tabular}}}}%
    \put(0.54892849,0.97166662){\makebox(0,0)[lt]{\lineheight{1.25}\smash{\begin{tabular}[t]{l}7\end{tabular}}}}%
    \put(0.10142845,0.75738094){\makebox(0,0)[lt]{\lineheight{1.25}\smash{\begin{tabular}[t]{l}8\end{tabular}}}}%
    \put(0.54880947,0.54309526){\makebox(0,0)[lt]{\lineheight{1.25}\smash{\begin{tabular}[t]{l}9\end{tabular}}}}%
  \end{picture}%
\endgroup%
 \ar[r,"\mu_{III}", start anchor=real east, end anchor=real west, shift left] & 
\begingroup%
  \makeatletter%
  \providecommand\color[2][]{%
    \errmessage{(Inkscape) Color is used for the text in Inkscape, but the package 'color.sty' is not loaded}%
    \renewcommand\color[2][]{}%
  }%
  \providecommand\transparent[1]{%
    \errmessage{(Inkscape) Transparency is used (non-zero) for the text in Inkscape, but the package 'transparent.sty' is not loaded}%
    \renewcommand\transparent[1]{}%
  }%
  \providecommand\rotatebox[2]{#2}%
  \newcommand*\fsize{\dimexpr\f@size pt\relax}%
  \newcommand*\lineheight[1]{\fontsize{\fsize}{#1\fsize}\selectfont}%
  \ifx\svgwidth\undefined%
    \setlength{\unitlength}{100.79999828bp}%
    \ifx\svgscale\undefined%
      \relax%
    \else%
      \setlength{\unitlength}{\unitlength * \real{\svgscale}}%
    \fi%
  \else%
    \setlength{\unitlength}{\svgwidth}%
  \fi%
  \global\let\svgwidth\undefined%
  \global\let\svgscale\undefined%
  \makeatother%
  \begin{picture}(1,1.57142863)%
    \lineheight{1}%
    \setlength\tabcolsep{0pt}%
    \put(0,0){\includegraphics[width=\unitlength,page=1]{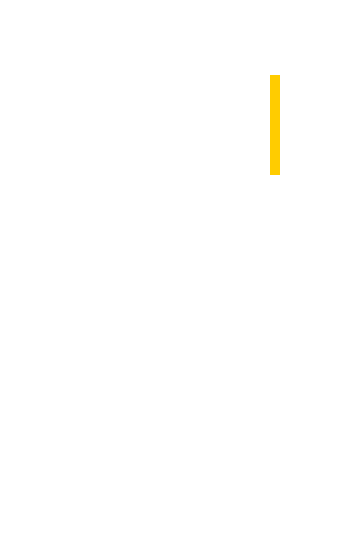}}%
    \put(0.19666662,1.42857149){\makebox(0,0)[lt]{\lineheight{1.25}\smash{\begin{tabular}[t]{l}1\end{tabular}}}}%
    \put(0.47746026,1.42857149){\makebox(0,0)[lt]{\lineheight{1.25}\smash{\begin{tabular}[t]{l}2\end{tabular}}}}%
    \put(0.76349202,1.42936514){\makebox(0,0)[lt]{\lineheight{1.25}\smash{\begin{tabular}[t]{l}3\end{tabular}}}}%
    \put(0.76349202,0.08539671){\makebox(0,0)[lt]{\lineheight{1.25}\smash{\begin{tabular}[t]{l}3\end{tabular}}}}%
    \put(0.19154755,0.0858729){\makebox(0,0)[lt]{\lineheight{1.25}\smash{\begin{tabular}[t]{l}4\end{tabular}}}}%
    \put(0.47706345,0.08619033){\makebox(0,0)[lt]{\lineheight{1.25}\smash{\begin{tabular}[t]{l}5\end{tabular}}}}%
    \put(0,0){\includegraphics[width=\unitlength,page=2]{moy-iii-c-prime.pdf}}%
  \end{picture}%
\endgroup%
 \ar[l, "\mu_{III}'", start anchor=real west, end anchor=real east, shift left]
    \end{tikzcd}
    \caption{Variations of the MOY III move.}
    \label{fig:moy-iii-variations}
\end{figure}
\section{Invariance}\label{sec:invariance}

In this section, we prove that $\C(\beta)$ is an invariant of the braid closure $\cl(\beta)$ by showing that it is invariant under each of the four moves of \cref{thm:markov}. The first two moves, Reidmeister II and III, apply to any partially singular braid diagram $D$, whereas the second two moves, stabilization and conjugation, are specific to diagrams of the form $D = I_n(\beta)$. 

\subsection{Reidemeister II}
We begin by proving invariance under Reidemeister II moves. There are two distinct such moves, but they are mirror images of each other, and their proofs are almost identical. We prove one of the cases in detail below. 

\begin{figure}[ht]
    \centering
    \begin{tikzcd}
\begingroup%
  \makeatletter%
  \providecommand\color[2][]{%
    \errmessage{(Inkscape) Color is used for the text in Inkscape, but the package 'color.sty' is not loaded}%
    \renewcommand\color[2][]{}%
  }%
  \providecommand\transparent[1]{%
    \errmessage{(Inkscape) Transparency is used (non-zero) for the text in Inkscape, but the package 'transparent.sty' is not loaded}%
    \renewcommand\transparent[1]{}%
  }%
  \providecommand\rotatebox[2]{#2}%
  \newcommand*\fsize{\dimexpr\f@size pt\relax}%
  \newcommand*\lineheight[1]{\fontsize{\fsize}{#1\fsize}\selectfont}%
  \ifx\svgwidth\undefined%
    \setlength{\unitlength}{72bp}%
    \ifx\svgscale\undefined%
      \relax%
    \else%
      \setlength{\unitlength}{\unitlength * \real{\svgscale}}%
    \fi%
  \else%
    \setlength{\unitlength}{\svgwidth}%
  \fi%
  \global\let\svgwidth\undefined%
  \global\let\svgscale\undefined%
  \makeatother%
  \begin{picture}(1,2.0999999)%
    \lineheight{1}%
    \setlength\tabcolsep{0pt}%
    \put(0,0){\includegraphics[width=\unitlength,page=1]{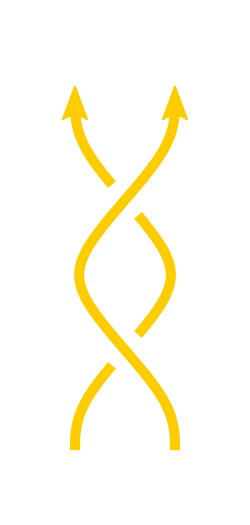}}%
    \put(0.26888881,0.11955524){\makebox(0,0)[lt]{\lineheight{1.25}\smash{\begin{tabular}[t]{l}3\end{tabular}}}}%
    \put(0.66816659,0.12022191){\makebox(0,0)[lt]{\lineheight{1.25}\smash{\begin{tabular}[t]{l}4\end{tabular}}}}%
    \put(0.27533326,1.89999991){\makebox(0,0)[lt]{\lineheight{1.25}\smash{\begin{tabular}[t]{l}1\end{tabular}}}}%
    \put(0.66844436,1.89999991){\makebox(0,0)[lt]{\lineheight{1.25}\smash{\begin{tabular}[t]{l}2\end{tabular}}}}%
    \put(0.1178888,1.00111104){\makebox(0,0)[lt]{\lineheight{1.25}\smash{\begin{tabular}[t]{l}5\end{tabular}}}}%
    \put(0.81777767,1.00111104){\makebox(0,0)[lt]{\lineheight{1.25}\smash{\begin{tabular}[t]{l}6\end{tabular}}}}%
  \end{picture}%
\endgroup%
 \ar[<->,r, start anchor=real east, end anchor=real west, shift left] & 
\begingroup%
  \makeatletter%
  \providecommand\color[2][]{%
    \errmessage{(Inkscape) Color is used for the text in Inkscape, but the package 'color.sty' is not loaded}%
    \renewcommand\color[2][]{}%
  }%
  \providecommand\transparent[1]{%
    \errmessage{(Inkscape) Transparency is used (non-zero) for the text in Inkscape, but the package 'transparent.sty' is not loaded}%
    \renewcommand\transparent[1]{}%
  }%
  \providecommand\rotatebox[2]{#2}%
  \newcommand*\fsize{\dimexpr\f@size pt\relax}%
  \newcommand*\lineheight[1]{\fontsize{\fsize}{#1\fsize}\selectfont}%
  \ifx\svgwidth\undefined%
    \setlength{\unitlength}{72bp}%
    \ifx\svgscale\undefined%
      \relax%
    \else%
      \setlength{\unitlength}{\unitlength * \real{\svgscale}}%
    \fi%
  \else%
    \setlength{\unitlength}{\svgwidth}%
  \fi%
  \global\let\svgwidth\undefined%
  \global\let\svgscale\undefined%
  \makeatother%
  \begin{picture}(1,2.0999999)%
    \lineheight{1}%
    \setlength\tabcolsep{0pt}%
    \put(0,0){\includegraphics[width=\unitlength,page=1]{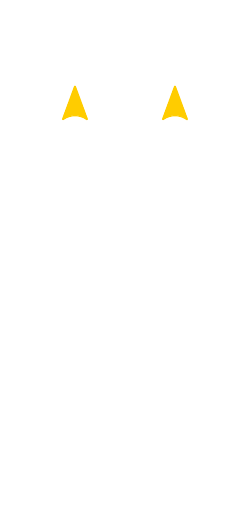}}%
    \put(0.27533324,0.12066646){\makebox(0,0)[lt]{\lineheight{1.25}\smash{\begin{tabular}[t]{l}1\end{tabular}}}}%
    \put(0.27533324,1.89999991){\makebox(0,0)[lt]{\lineheight{1.25}\smash{\begin{tabular}[t]{l}1\end{tabular}}}}%
    \put(0.66844429,0.11955539){\makebox(0,0)[lt]{\lineheight{1.25}\smash{\begin{tabular}[t]{l}2\end{tabular}}}}%
    \put(0.66844429,1.89999991){\makebox(0,0)[lt]{\lineheight{1.25}\smash{\begin{tabular}[t]{l}2\end{tabular}}}}%
    \put(0,0){\includegraphics[width=\unitlength,page=2]{r-ii-prime.pdf}}%
  \end{picture}%
\endgroup%

    \end{tikzcd}
    \caption{A Reidemeister II move.}
    \label{fig:r-ii}
\end{figure}

\begin{theorem}\label{thm:r-ii}
If $D$ and $D'$ are two partially singular braid diagrams that differ by a Reidemeister II move, then $\C(D) \homotopic_1 \C(D')$ over $R(D')$.
\end{theorem}

\begin{proof}

Let $D$ and $D'$ be the diagrams in \cref{fig:r-ii}, with $D'$ the result of eliminating two crossings from $D'$ by means of a Reidemeister II move. We will use \cref{lemma:gaussian-elimination} to simplify $\C(D)$ and $\C(D')$ to see they have the same homotopy type. Label the edges of $D$ with variables $U_1,\dots,U_6$ as in \cref{fig:r-ii}, and order the crossings from top to bottom. Let $\phi_1 = \phi_+ = 1$ be the edge map corresponding to the top (positive) vertex, and let $\phi_2 = \phi_- = U_6 - U_3$ be the edge map corresponding to the bottom (negative) vertex. Then, fixing a sign assignment without loss of generality, we expand the cube of resolutions for $\C(D)$ as:

\begin{center}
\begin{tikzcd}[ampersand replacement=\&, row sep = huge, column sep = huge]
    \C(D_{00}) \ar[r,"\phi_1"] \ar[d,"\phi_2"] \& \C(D_{10}) \ar[d,"-\phi_2"] \\
    \C(D_{01}) \ar[r,"\phi_1"] \& \C(D_{11}) \nospaceperiod
\end{tikzcd}
\end{center} 

Note that $\C(D_{10})$ is isomorphic to $\C(D')$ via the removal of bivalent vertices, so our goal is to show that $\C(D) \homotopic_1 \C(D_{10})$ as a filtered chain complex over $R(D')$. We will work over the larger ring $R(D')[U_3,U_4]$, but will not enforce linearity with respect to $U_5$ or $U_6$. First, note that $\C(D_{00}) \iso \C(D_{11})$. We see that we can apply the MOY II decomposition from \cref{sec:moy-ii} to write $\C(D_{01}) = \C(D_{11})\spannedby{1} \directsum \C(D_{00})\spannedby{U_6}$. We compute the maps induced by $\phi_1$ and $\phi_2$ on these decompositions to get an isomorphic cube of resolutions:

\begin{center}
\begin{tikzcd}[ampersand replacement=\&, row sep = huge, column sep = huge]
    \C(D_{00}) \ar[r,"1"] \ar[d,"\pmat{-U_3 \\ 1}"] \& \C(D_{10}) \ar[d,"U_3-U_2"] \\
    \C(D_{11})\spannedby{1} \directsum \C(D_{00})\spannedby{U_6} \ar[r,"\pmat{1 & U_2}"] \& \C(D_{11}) \nospaceperiod
\end{tikzcd}
\end{center}

This is the first of several times we will use \cref{lemma:gaussian-elimination} to simplify a cube of resolutions in this paper. This key lemma allows us to effectively cancel out isomorphisms of direct summands in cubes. In this case, it yields the $E_1$-quasi-isomorphic complex: 

\begin{center}
\begin{tikzcd}[ampersand replacement=\&, row sep = huge, column sep = huge]
    0 \ar[r] \ar[d] \& \C(D_{10}) \ar[d] \\
    0 \ar[r] \& 0 \nospaceperiod
\end{tikzcd}
\end{center}

We conclude by noting that $\C(D_{10}) \iso \C(D')$ as chain complexes over $R(D')[U_3,U_4]$. This proves invariance under one type of Reidemeister II move; the proof of the mirror-image move is analogous.

\end{proof}

\subsection{Reidemeister III}

We will prove invariance under the Reidemeister III move shown in \cref{fig:r-iii}, which corresponds to sliding a strand over a positive crossing. In terms of the braid group, it represents the relation $\sigma_i \sigma_{i+1} \sigma_i = \sigma_{i+1} \sigma_i \sigma_{i+1}$. All other variations of the Reidemeister III move follow from this one plus the Reidemeister II invariance result from \cref{thm:r-ii}.

\begin{figure}[ht]
    \centering
    \begin{tikzcd}
    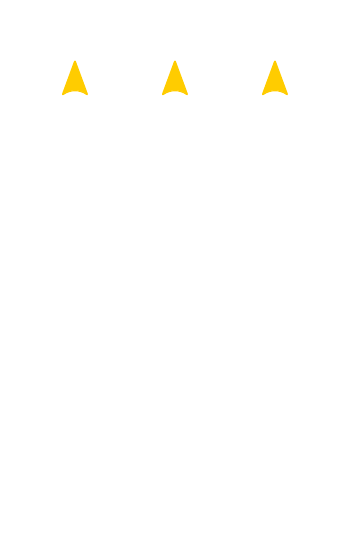 \ar[<->,r, start anchor=real east, end anchor=real west, shift left] & 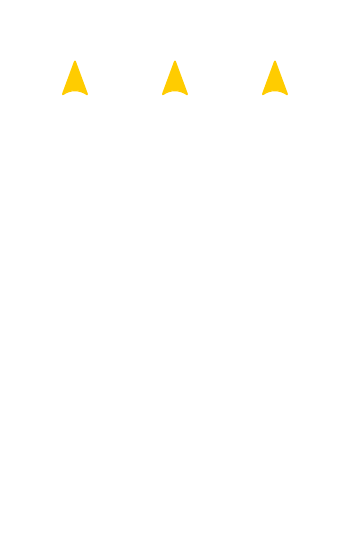
    \end{tikzcd}
    \caption{A Reidemeister III move.}
    \label{fig:r-iii}
\end{figure}

\begin{theorem}\label{thm:r-iii}
If $D$ and $D'$ are two partially singular braid diagrams that differ by a Reidemeister III move, then $\C(D) \homotopic_1 \C(D')$.
\end{theorem}

\begin{proof}

Let $D$ be the diagram on the left and $D'$ the diagram on the right in \cref{fig:r-iii}. We will use \cref{lemma:gaussian-elimination} to simplify $\C(D)$ and $\C(D')$ to see they have the same homotopy type. Label the edges of $D$ with variables $U_1,\dots,U_9$ as in \cref{fig:r-iii}, and order the crossings from top to bottom. We expand the cube of resolutions for $\C(D)$ as:

\begin{center}
\begin{tikzcd}[ampersand replacement=\&, row sep = huge, column sep = huge]
    \& \C(D_{100}) \ar[r] \ar[dr] \& \C(D_{110}) \ar[dr,"-"] \\
    \C(D_{000}) \ar[ur,"-"] \ar[r] \ar[dr] \& \C(D_{010}) \ar[ur] \ar[dr,"-" near start] \& \C(D_{101}) \ar[r] \& \C(D_{111}) \\
    \& \C(D_{001}) \ar[ur] \ar[r] \& \C(D_{011}) \ar[ur,"-"] \nospaceperiod
\end{tikzcd}
\end{center}

Since our local picture of $D$ consists of only positive crossings, all edge maps in this cube are given by $\phi_+ = 1$ up to a sign assignment, which we take to be the one in the above cube of resolutions without loss of generality.

By \cref{thm:moy-iii}, we note $\C(D_{000})\iso\C(D_{110})\directsum\Upsilon$, where $\Upsilon$ is acyclic. By a slight generalization of \cref{lemma:homotopic-cubes}, we get an $E_1$-quasi-isomorphic cube after replacing $\C(D_{000})$ by $\C(D_{110})$. Furthermore, \cref{thm:moy-ii} gives us that $\C(D_{010}) \iso \C(D_{110})\spannedby{1} \directsum \C(D_{110})\spannedby{U_9}$. Therefore, the above cube is $E_1$-quasi-isomorphic to:

\begin{center}
\begin{tikzcd}[ampersand replacement=\&, row sep = huge, column sep = normal]
    \& \C(D_{100}) \ar[r] \ar[dr] \& \C(D_{110}) \ar[dr,"-"] \\
    \C(D_{110}) \ar[ur,"-"] \ar[r,"\alpha"] \ar[dr] \& \C(D_{110})\spannedby{1} \directsum \C(D_{110})\spannedby{U_9} \ar[ur, "\beta" near start] \ar[dr,"-" near start] \& \C(D_{101}) \ar[r] \& \C(D_{111}) \\
    \& \C(D_{001}) \ar[ur] \ar[r] \& \C(D_{011}) \ar[ur,"-"] \nospaceperiod
\end{tikzcd}
\end{center}

We compute the induced maps in the above cube to be $\alpha = \pmat{-U_3 \\ 1}$ and $\beta = \pmat {1 & U_2}$. By \cref{lemma:gaussian-elimination}, we can cancel the isomorphisms of direct summands in the above cube to obtain the $E_1$-quasi-isomorphic complex:

\begin{center}
\begin{tikzcd}[ampersand replacement=\&, row sep = huge, column sep = huge]
    \& \C(D_{100}) \ar[r] \ar[dr] \ar[ddr,"\gamma"] \& 0 \ar[dr] \\
    0 \ar[ur] \ar[r] \ar[dr] \& 0 \ar[ur] \ar[dr] \& \C(D_{101}) \ar[r] \& \C(D_{111}) \\
    \& \C(D_{001}) \ar[ur] \ar[r] \& \C(D_{011}) \ar[ur,"-"] \nospaceperiod
\end{tikzcd}
\end{center}

Removing the trivial complexes in the above cube, and noting that the map $\gamma$ induced by cancellation is given by multiplication by $1$, we get the complex: 

\begin{center}
\begin{tikzcd}[ampersand replacement=\&, row sep = huge, column sep = huge]
    \C(D_{100}) \ar[r] \ar[ddr] \& \C(D_{101}) \ar[dr] \\
    \& \& \C(D_{111}) \\
    \C(D_{001}) \ar[uur] \ar[r] \& \C(D_{011}) \ar[ur,"-"] \nospaceperiod
\end{tikzcd}
\end{center}

Now, recalling that we have a second diagram $D'$ to work with, we may go through the same steps to simplify $\C(D')$ to get the complex:

\begin{center}
\begin{tikzcd}[ampersand replacement=\&, row sep = huge, column sep = huge]
    \C(D'_{100}) \ar[r] \ar[ddr] \& \C(D'_{101}) \ar[dr,"-"] \\
    \& \& \C(D'_{111}) \\
    \C(D'_{001}) \ar[uur] \ar[r] \& \C(D'_{011}) \ar[ur] \nospaceperiod
\end{tikzcd}
\end{center}

We conclude by noting that the reduced complexes for $\C(D)$ and $\C(D')$ are isomorphic via the map that reflects the complexes about a horizontal axis, i.e.\ swaps the $100$- and $001$-resolutions, swaps the $101$- and $011$-resolutions, and fixes the $111$-resolution. This map is a chain map since all the edge maps are $\pm 1$, and therefore $\C(D) \homotopic_1 \C(D')$.

\end{proof}

\subsection{Stabilization}

\begin{figure}[ht]
    \centering
    \begin{tabularx}{0.8\textwidth}{YY}
        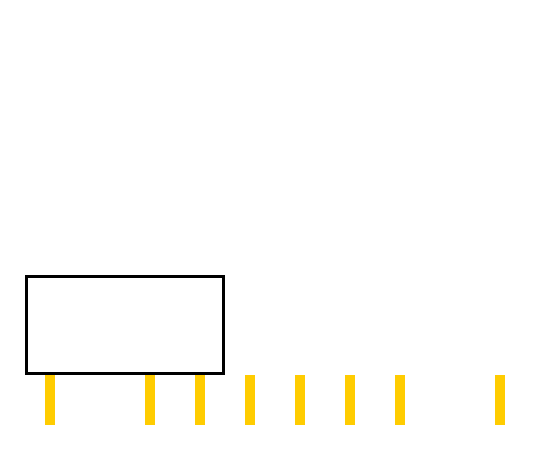 &
        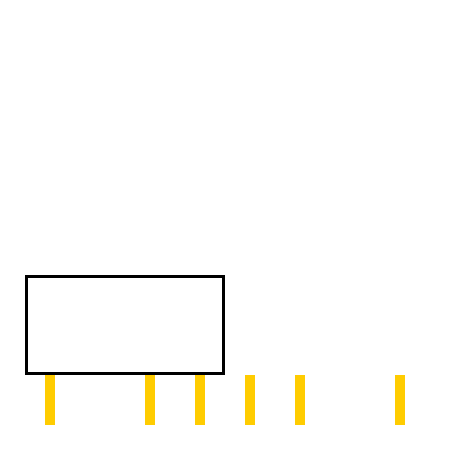 \\
        $D$: Stabilized/Original & $D'$: Original/Destabilized
    \end{tabularx}
    \caption{Diagrams related by a positive stabilization. Depending on context, we will either consider the diagram $D$ and its \emph{destabilization} $D'$, or we will consider the diagram $D'$ and its \emph{stabilization} $D$. }
    \label{fig:stabilization}
\end{figure}

Let $\beta \in B_{n-1}$ be an element of the braid group for $n \ge 2$, and let $\sigma_{n-1} \in B_n$ be the generator which introduces a positive crossing between strands $n-1$ and $n$. The \textit{positive stabilization} of $\beta$ is the braid $\sigma_{n-1} \beta \in B_n$, where here we are considering $\beta$ as an element of $B_n$ via the natural inclusion $B_{n-1} \inj B_n$. Analogously, the \textit{negative stabilization} of $\beta$ is the braid $\sigma_{n-1}^{-1} \beta$. For a braid $\beta = \sigma_{n-1}^{\pm 1} \beta' \in B_n$ in the image of one of these operations, we say that $\beta' \in B_{n-1}$ is the \textit{destabilization} of $\beta$. 

\begin{theorem}\label{thm:stabilization}
The $E_1$-homotopy type of the filtered complex $\C(\beta)$ is invariant under positive and negative (de)stabilization, i.e. 
$\C(\sigma_{n-1} \beta) \homotopic_1 \C(\sigma_{n-1}^{-1} \beta) \homotopic_1 \C(\beta)$
\end{theorem}

We note that $\beta$, $\sigma_{n-1} \beta$, and $\sigma_{n-1}^{-1} \beta$ all have isotopic braid closures. Before we prove stabilization invariance, we will need to relate diagrams containing the open braid diagrams $I_n$ and $I_{n-1}$, as the (de)stabilization operations alter the number of strands of our partially singular braid diagrams. Therefore, we first note that we can see $I_{n-1}$ as a sub-diagram of $I_n$ by ignoring the rightmost vertices in every row, as in \cref{fig:i-n-recursive-def}.

\begin{figure}[ht]
    \centering
    \begin{tikzcd}
\begingroup%
  \makeatletter%
  \providecommand\color[2][]{%
    \errmessage{(Inkscape) Color is used for the text in Inkscape, but the package 'color.sty' is not loaded}%
    \renewcommand\color[2][]{}%
  }%
  \providecommand\transparent[1]{%
    \errmessage{(Inkscape) Transparency is used (non-zero) for the text in Inkscape, but the package 'transparent.sty' is not loaded}%
    \renewcommand\transparent[1]{}%
  }%
  \providecommand\rotatebox[2]{#2}%
  \newcommand*\fsize{\dimexpr\f@size pt\relax}%
  \newcommand*\lineheight[1]{\fontsize{\fsize}{#1\fsize}\selectfont}%
  \ifx\svgwidth\undefined%
    \setlength{\unitlength}{57.60000086bp}%
    \ifx\svgscale\undefined%
      \relax%
    \else%
      \setlength{\unitlength}{\unitlength * \real{\svgscale}}%
    \fi%
  \else%
    \setlength{\unitlength}{\svgwidth}%
  \fi%
  \global\let\svgwidth\undefined%
  \global\let\svgscale\undefined%
  \makeatother%
  \begin{picture}(1,2.75000002)%
    \lineheight{1}%
    \setlength\tabcolsep{0pt}%
    \put(0,0){\includegraphics[width=\unitlength,page=1]{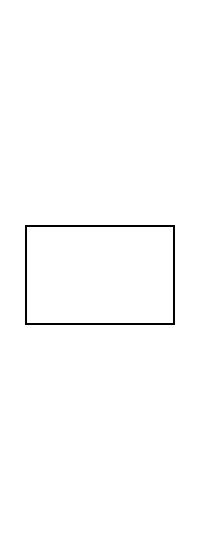}}%
    \put(0.42698841,1.33485189){\color[rgb]{0,0,0}\makebox(0,0)[lt]{\lineheight{1.25}\smash{\begin{tabular}[t]{l}$I_{n}$\end{tabular}}}}%
    \put(0.36448839,0.89630046){\color[rgb]{0,0,0}\makebox(0,0)[lt]{\lineheight{1.25}\smash{\begin{tabular}[t]{l}$\dots$\end{tabular}}}}%
    \put(0.36448839,1.77130148){\color[rgb]{0,0,0}\makebox(0,0)[lt]{\lineheight{1.25}\smash{\begin{tabular}[t]{l}$\cdots$\end{tabular}}}}%
    \put(0.23948838,0.77302998){\color[rgb]{0,0,0}\makebox(0,0)[lt]{\lineheight{1.25}\smash{\begin{tabular}[t]{l}$\underbrace{\hspace{0.4in}}_{2n}$\end{tabular}}}}%
    \put(0.23948835,2.02303037){\color[rgb]{0,0,0}\makebox(0,0)[lt]{\lineheight{1.25}\smash{\begin{tabular}[t]{l}$\overbrace{\hspace{0.4in}}^{2n}$\end{tabular}}}}%
    \put(0,0){\includegraphics[width=\unitlength,page=2]{i-n-recursive-def-prime.pdf}}%
  \end{picture}%
\endgroup%
 \ar[r,equal,start anchor=real east, end anchor=real west] & 
        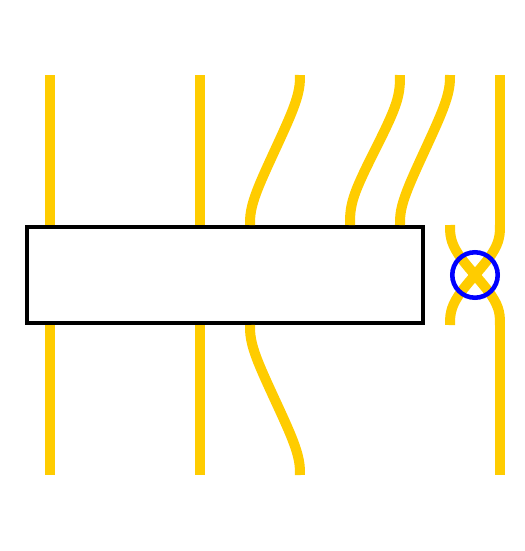
    \end{tikzcd}
    \caption{A recursive definition of $I_n$.}
    \label{fig:i-n-recursive-def}
\end{figure}

Consider $I_n(\sigma_{n-1}\beta)$, as shown in \cref{fig:stabilization-labels}. Let $R$ be the polynomial ring over all edges not labeled in \cref{fig:stabilization-labels}, and label the rest of the edges accordingly, so that $R(I_n(\sigma_{n-1}\beta)) = R[U_1,U_2,U_3,U_4,U_5]$. 

\begin{figure}[ht]
    \centering
    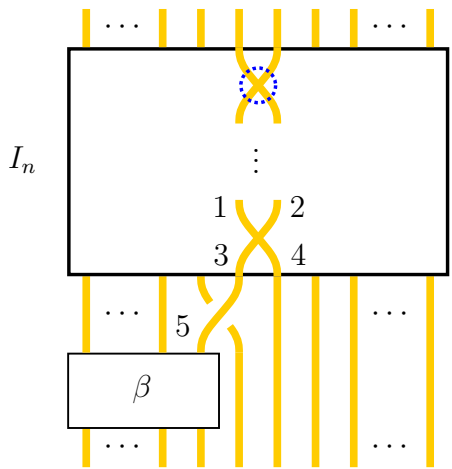
    \caption{Relevant edge labels near $\sigma_{n-1}$. Note that the top vertex is fixed only when $n=2$, and is otherwise free for $n \ge 3$.}
    \label{fig:stabilization-labels}
\end{figure}

Let $\phi = \phi_+ = 1$ be the edge map corresponding to the positive crossing of $\sigma_{n-1}$. We write the one-dimensional cube of resolutions corresponding to resolving the crossing:

\begin{center}
\begin{tikzcd}[ampersand replacement=\&, row sep = huge, column sep = huge]
\C(I_{n}(\sigma_{n-1}\beta)_0) \ar[r,"\phi"] \& \C(I_{n}(\sigma_{n-1}\beta)_1) \nospaceperiod
\end{tikzcd}
\end{center}

The diagrams corresponding to these resolutions are illustrated in \cref{fig:stabilization-cone}.

\begin{figure}[ht]
    \centering
    \begin{tikzcd} 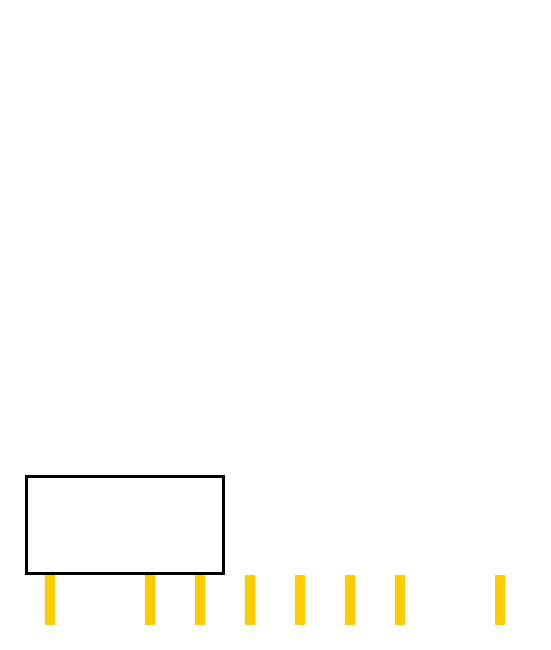 \ar[r, start anchor=real east, end anchor=real west, "\phi"] & 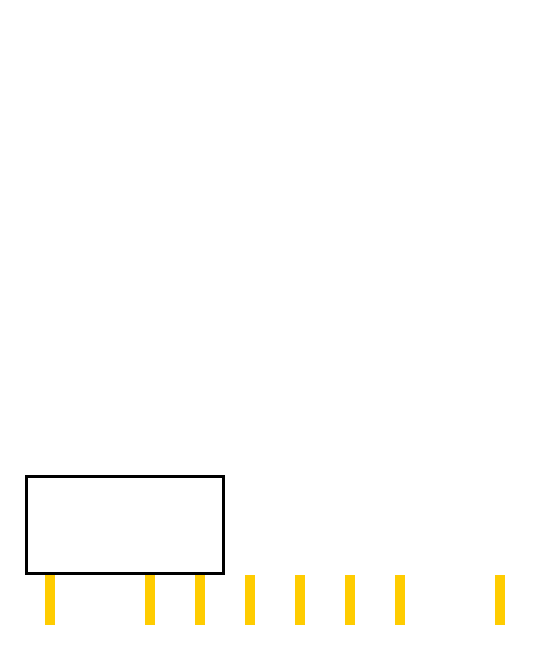
    \end{tikzcd}
    \caption{The mapping cone decomposition induced by $\sigma_{n-1}$.}
    \label{fig:stabilization-cone}
\end{figure}

Our goal now is to use MOY moves to modify both resolutions so that they can be represented using a common diagram, tracking the effect on the complexes. By an abuse of notation, we will denote this common diagram $I'_{n}(\beta)$, which is gotten analogously to $I_n(\beta)$: we place a straight strand to the right of $\beta$, place $n$ straight strands to the right of that, top this diagram with $I'_{n}$, and take the braid closure. It remains to define $I'_n$. We let $I'_n$ be $I_n$, without the singular vertex between strands $n$ and $n+1$ in the first layer. As for $I_{n}$, we can also define $I'_{n}$ by building on $I_{n-1}$, as in \cref{fig:i-n-prime-recursive-def}. 

\begin{figure}[ht]
    \centering
    \begin{tikzcd}
\begingroup%
  \makeatletter%
  \providecommand\color[2][]{%
    \errmessage{(Inkscape) Color is used for the text in Inkscape, but the package 'color.sty' is not loaded}%
    \renewcommand\color[2][]{}%
  }%
  \providecommand\transparent[1]{%
    \errmessage{(Inkscape) Transparency is used (non-zero) for the text in Inkscape, but the package 'transparent.sty' is not loaded}%
    \renewcommand\transparent[1]{}%
  }%
  \providecommand\rotatebox[2]{#2}%
  \newcommand*\fsize{\dimexpr\f@size pt\relax}%
  \newcommand*\lineheight[1]{\fontsize{\fsize}{#1\fsize}\selectfont}%
  \ifx\svgwidth\undefined%
    \setlength{\unitlength}{57.60000086bp}%
    \ifx\svgscale\undefined%
      \relax%
    \else%
      \setlength{\unitlength}{\unitlength * \real{\svgscale}}%
    \fi%
  \else%
    \setlength{\unitlength}{\svgwidth}%
  \fi%
  \global\let\svgwidth\undefined%
  \global\let\svgscale\undefined%
  \makeatother%
  \begin{picture}(1,2.75000002)%
    \lineheight{1}%
    \setlength\tabcolsep{0pt}%
    \put(0,0){\includegraphics[width=\unitlength,page=1]{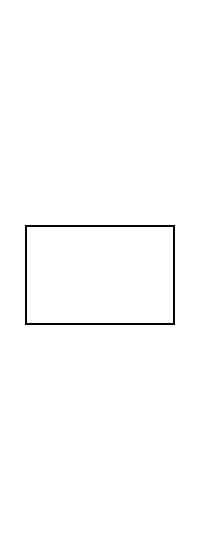}}%
    \put(0.42698841,1.33485189){\color[rgb]{0,0,0}\makebox(0,0)[lt]{\lineheight{1.25}\smash{\begin{tabular}[t]{l}$I'_{n}$\end{tabular}}}}%
    \put(0.36448839,0.89630046){\color[rgb]{0,0,0}\makebox(0,0)[lt]{\lineheight{1.25}\smash{\begin{tabular}[t]{l}$\dots$\end{tabular}}}}%
    \put(0.36448839,1.77130148){\color[rgb]{0,0,0}\makebox(0,0)[lt]{\lineheight{1.25}\smash{\begin{tabular}[t]{l}$\cdots$\end{tabular}}}}%
    \put(0.23948838,0.77302998){\color[rgb]{0,0,0}\makebox(0,0)[lt]{\lineheight{1.25}\smash{\begin{tabular}[t]{l}$\underbrace{\hspace{0.4in}}_{2n}$\end{tabular}}}}%
    \put(0.23948835,2.02303037){\color[rgb]{0,0,0}\makebox(0,0)[lt]{\lineheight{1.25}\smash{\begin{tabular}[t]{l}$\overbrace{\hspace{0.4in}}^{2n}$\end{tabular}}}}%
    \put(0,0){\includegraphics[width=\unitlength,page=2]{i-n-prime-recursive-def-prime.pdf}}%
  \end{picture}%
\endgroup%
 \ar[r,equal,start anchor=real east, end anchor=real west] & 
        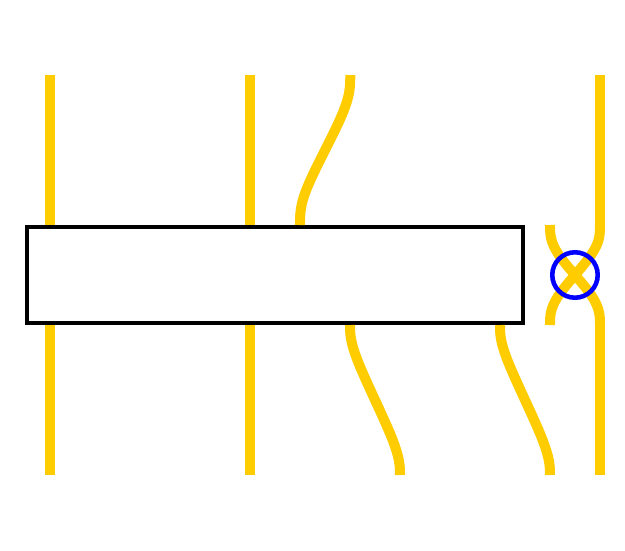
    \end{tikzcd}
    \caption{Building $I'_n$ from $I_{n-1}$.}
    \label{fig:i-n-prime-recursive-def}
\end{figure}

With this definition in mind, we now see that $I_{n}(\sigma_{n-1}\beta)_0$ is one MOY III move away from $I'_{n}(\beta)$, and $I_{n}(\sigma_{n-1}\beta)_1$ is one MOY II move away from $I'_{n}(\beta)$. On the one-dimensional cube of resolutions, then, we get 
\begin{center}
\begin{tikzcd}[ampersand replacement=\&, row sep = huge, column sep = huge]
\C(I'_{n}(\beta))\spannedby{U_3 - U_5} \directsum \Upsilon \ar[r,"\phi"] \& \C(I'_{n}(\beta))\spannedby{1} \directsum \C(I'_{n}(\beta))\spannedby{U_4} \nospaceperiod
\end{tikzcd}
\end{center}
Since $\Upsilon$ is acyclic, we can ignore it by \cref{lemma:homotopic-cones}. We compute the map induced by $\phi$ on the summands as:
\begin{center}
\begin{tikzcd}[ampersand replacement=\&, row sep = huge, column sep = huge]
\C(I'_{n}(\beta))\spannedby{U_3 - U_5} \ar[r,"\pmat{U_1 + U_2 - U_5 \\ -1}"] \& \C(I'_{n}(\beta))\spannedby{1} \directsum \C(I'_{n}(\beta))\spannedby{U_4} \nospaceperiod
\end{tikzcd}
\end{center}
Since the $-1$ entry represents an isomorphism of $\C(I'_n(\beta))$ summands, we can cancel it by \cref{lemma:gaussian-elimination}. This proves the following lemma. 

\begin{lemma}
The complexes $\C(\sigma_{n-1}\beta)$ and $\C(I'_{n}(\beta))$ have the same $E_1$-homotopy type. 
\end{lemma}

Therefore, to prove conjugation invariance, it remains to prove the following proposition. 

\begin{proposition}\label{prop:i-n-prime}
The complexes $\C(I'_{n}(\beta))$ and $\C(\beta)$ have the same $E_1$-homotopy type. 
\end{proposition}

\begin{proof}
To begin, we note that $I'_{n}(\beta)$ and $I_{n-1}(\beta)$ are really braid closures, so for ease of understanding the upcoming MOY moves, we will replace our usual depiction of $I'_{n}(\beta)$ with a shifted version, as in \cref{fig:shifting-vertices-in-i-n-prime}. 

\begin{figure}[ht]
    \centering
    \begin{tabularx}{\textwidth}{YY}
        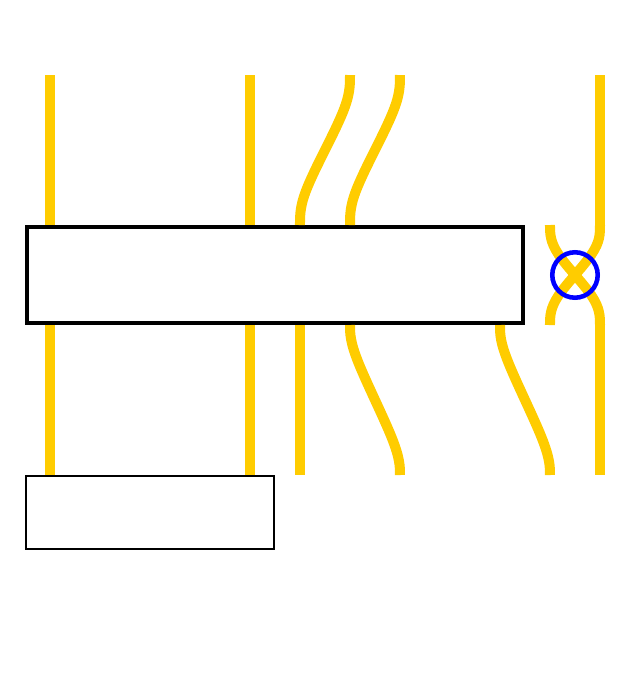 &
        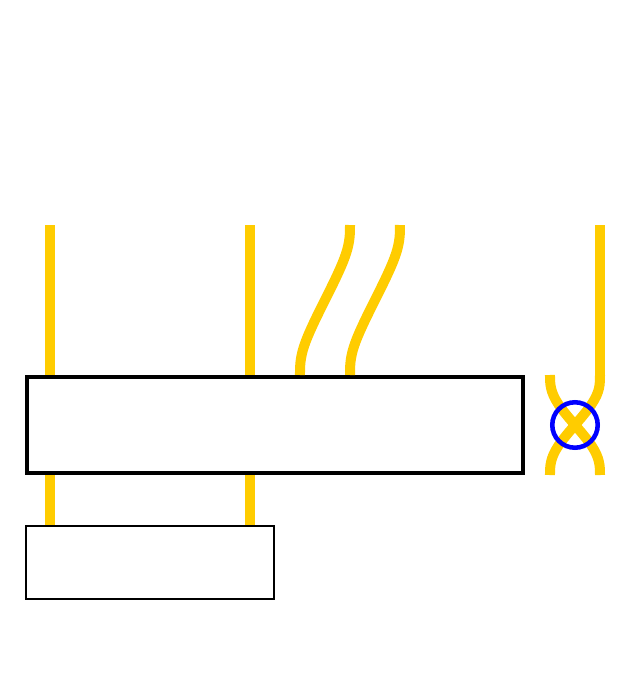
    \end{tabularx}
    \caption{Shifting vertices in $I'_n(\beta)$.}
    \label{fig:shifting-vertices-in-i-n-prime}
\end{figure}

In this shifted version, we identify the local picture on the left in \cref{fig:stabilization-lemma}, consisting of a pair of intersecting strands and $n-2$ other strands which intersect both. We can apply $n-2$ MOY III moves to simplify this part of the diagram to the local picture on the right in \cref{fig:stabilization-lemma}, consisting of $n-2$ straight strands and one pair of intersecting strands. By \cref{thm:moy-iii}, each of these preserves the $E_1$-homotopy type of the complex. The global picture at this stage can be seen on the left in \cref{fig:stabilization-last-step}. To arrive at the diagram for $I_{n-1}(\beta)$, we apply two MOY I moves to the two remaining fixed vertices outside of $I_{n-1}$. By \cref{thm:moy-i}, each of these preserves the $E_1$-homotopy type of the complex. This leads to the diagram on the right in \cref{fig:stabilization-last-step}, which is exactly the diagram for $I_{n-1}(\beta)$. 

\end{proof}

\begin{figure}[ht]
    \centering
    \begin{tabularx}{0.8\textwidth}{YY}
\begingroup%
  \makeatletter%
  \providecommand\color[2][]{%
    \errmessage{(Inkscape) Color is used for the text in Inkscape, but the package 'color.sty' is not loaded}%
    \renewcommand\color[2][]{}%
  }%
  \providecommand\transparent[1]{%
    \errmessage{(Inkscape) Transparency is used (non-zero) for the text in Inkscape, but the package 'transparent.sty' is not loaded}%
    \renewcommand\transparent[1]{}%
  }%
  \providecommand\rotatebox[2]{#2}%
  \newcommand*\fsize{\dimexpr\f@size pt\relax}%
  \newcommand*\lineheight[1]{\fontsize{\fsize}{#1\fsize}\selectfont}%
  \ifx\svgwidth\undefined%
    \setlength{\unitlength}{86.40000343bp}%
    \ifx\svgscale\undefined%
      \relax%
    \else%
      \setlength{\unitlength}{\unitlength * \real{\svgscale}}%
    \fi%
  \else%
    \setlength{\unitlength}{\svgwidth}%
  \fi%
  \global\let\svgwidth\undefined%
  \global\let\svgscale\undefined%
  \makeatother%
  \begin{picture}(1,1.4999999)%
    \lineheight{1}%
    \setlength\tabcolsep{0pt}%
    \put(0,0){\includegraphics[width=\unitlength,page=1]{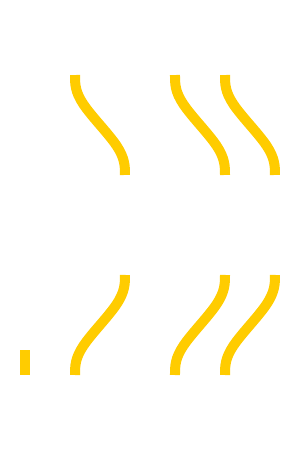}}%
    \put(0.32632557,1.18086729){\color[rgb]{0,0,0}\makebox(0,0)[lt]{\lineheight{1.25}\smash{\begin{tabular}[t]{l}$\cdots$\end{tabular}}}}%
    \put(0.32632557,0.26420038){\color[rgb]{0,0,0}\makebox(0,0)[lt]{\lineheight{1.25}\smash{\begin{tabular}[t]{l}$\dots$\end{tabular}}}}%
    \put(0,0){\includegraphics[width=\unitlength,page=2]{stabilization-lemma.pdf}}%
    \put(0.24299223,1.34868692){\color[rgb]{0,0,0}\makebox(0,0)[lt]{\lineheight{1.25}\smash{\begin{tabular}[t]{l}$\overbrace{\hspace{0.4in}}^{n-3}$\end{tabular}}}}%
    \put(0.24299231,0.18202044){\color[rgb]{0,0,0}\makebox(0,0)[lt]{\lineheight{1.25}\smash{\begin{tabular}[t]{l}$\underbrace{\hspace{0.4in}}_{n-3}$\end{tabular}}}}%
  \end{picture}%
\endgroup%
 &
\begingroup%
  \makeatletter%
  \providecommand\color[2][]{%
    \errmessage{(Inkscape) Color is used for the text in Inkscape, but the package 'color.sty' is not loaded}%
    \renewcommand\color[2][]{}%
  }%
  \providecommand\transparent[1]{%
    \errmessage{(Inkscape) Transparency is used (non-zero) for the text in Inkscape, but the package 'transparent.sty' is not loaded}%
    \renewcommand\transparent[1]{}%
  }%
  \providecommand\rotatebox[2]{#2}%
  \newcommand*\fsize{\dimexpr\f@size pt\relax}%
  \newcommand*\lineheight[1]{\fontsize{\fsize}{#1\fsize}\selectfont}%
  \ifx\svgwidth\undefined%
    \setlength{\unitlength}{86.40000343bp}%
    \ifx\svgscale\undefined%
      \relax%
    \else%
      \setlength{\unitlength}{\unitlength * \real{\svgscale}}%
    \fi%
  \else%
    \setlength{\unitlength}{\svgwidth}%
  \fi%
  \global\let\svgwidth\undefined%
  \global\let\svgscale\undefined%
  \makeatother%
  \begin{picture}(1,1.4999999)%
    \lineheight{1}%
    \setlength\tabcolsep{0pt}%
    \put(0.32632557,1.18086729){\color[rgb]{0,0,0}\makebox(0,0)[lt]{\lineheight{1.25}\smash{\begin{tabular}[t]{l}$\cdots$\end{tabular}}}}%
    \put(0.32632557,0.26420038){\color[rgb]{0,0,0}\makebox(0,0)[lt]{\lineheight{1.25}\smash{\begin{tabular}[t]{l}$\dots$\end{tabular}}}}%
    \put(0.24299223,1.34868692){\color[rgb]{0,0,0}\makebox(0,0)[lt]{\lineheight{1.25}\smash{\begin{tabular}[t]{l}$\overbrace{\hspace{0.4in}}^{n-3}$\end{tabular}}}}%
    \put(0.24299231,0.18202044){\color[rgb]{0,0,0}\makebox(0,0)[lt]{\lineheight{1.25}\smash{\begin{tabular}[t]{l}$\underbrace{\hspace{0.4in}}_{n-3}$\end{tabular}}}}%
    \put(0,0){\includegraphics[width=\unitlength,page=1]{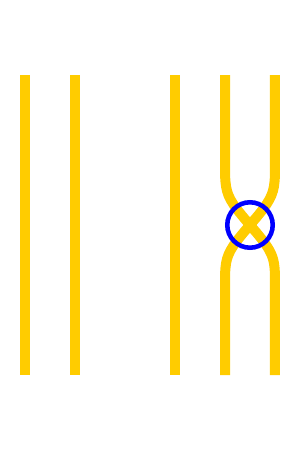}}%
  \end{picture}%
\endgroup%

    \end{tabularx}
    \caption{Local pictures of diagrams related by a sequence of $n-2$ MOY III moves.}
    \label{fig:stabilization-lemma}
\end{figure}

\begin{figure}[ht]
    \centering
    \begin{tabularx}{\textwidth}{YY}
        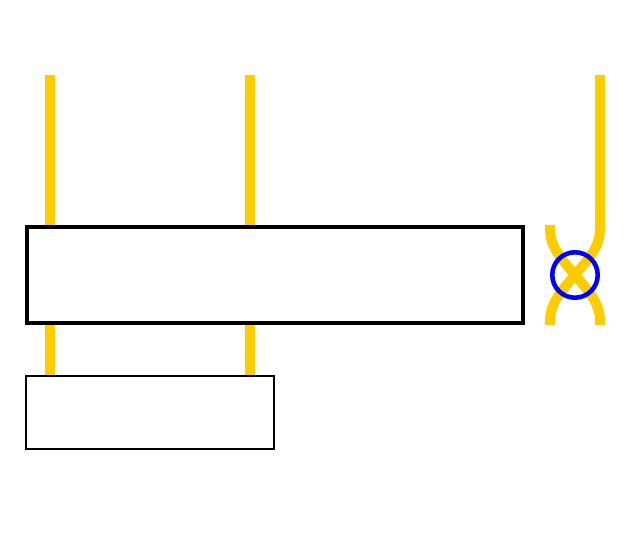 &
        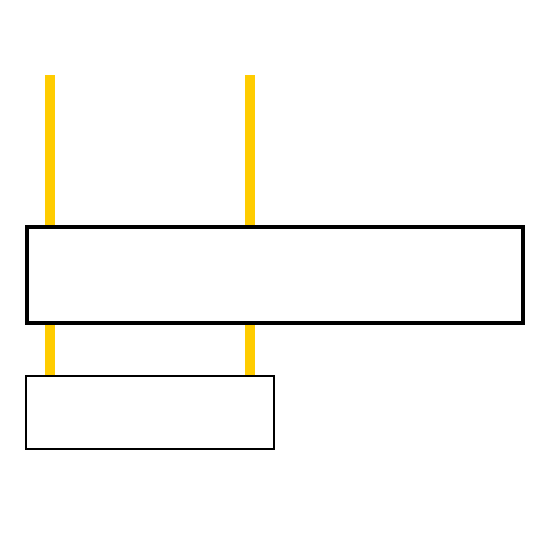
    \end{tabularx}
    \caption{The last step in simplifying $I'_n(\beta)$ in the proof of \cref{prop:i-n-prime}. }
    \label{fig:stabilization-last-step}
\end{figure}

\begin{proof}[Proof of \cref{thm:stabilization}]
We have shown that $\C(\sigma_{n-1} \beta) \homotopic \C(I'_n(\beta)) \homotopic \C(\beta)$. We can simplify $\C(\sigma_{n-1}^{-1} \beta)$ to $\C(I'_n(\beta))$ as well. Using the same edge labels and notation as before, we write the cube of resolutions for $\C(\beta)$ as
\begin{center}
\begin{tikzcd}[ampersand replacement=\&, row sep = huge, column sep = huge]
\C(I_n(\beta)_0) \ar[r,"\phi"] \& \C(I_n(\beta)_1)
\end{tikzcd}
\end{center}
where this time $\phi = \phi_- = U_3-U_5$. Applying an MOY II move to $I_n(\beta)_0$ and an MOY III move to $I_n(\beta)_1$ to write our complexes in terms of $\C(I'_n(\beta))$ gives us the complex
\begin{center}
\begin{tikzcd}[ampersand replacement=\&, row sep = huge, column sep = huge]
\C(I'_n(\beta))\spannedby{1} \directsum \C(I'_n(\beta))\spannedby{U_4} \ar[r,"\phi"] \& \C(I'_n(\beta))\spannedby{U_3 - U_5} \directsum \Upsilon \nospaceperiod
\end{tikzcd}
\end{center}
Again, excluding $\Upsilon$ and computing the map induced by $\phi$, we get:
\begin{center}
\begin{tikzcd}[ampersand replacement=\&, row sep = huge, column sep = huge]
\C(I'_n(\beta))\spannedby{1} \directsum \C(I'_n(\beta))\spannedby{U_4} \ar[r,"\pmat{1 & U_4}"] \& \C(I'_n(\beta))\spannedby{U_3 - U_5} \nospaceperiod
\end{tikzcd}
\end{center}
As before, we may cancel the $1$ in the above matrix to see that $\C(\sigma_{n-1}^{-1} \beta) \homotopic_1 \C(I'_n(\beta))$ as well. The rest of the proof follows from \cref{prop:i-n-prime}. Since we have covered both the positive and negative cases, this suffices to show invariance under stabilization.
\end{proof}

\subsection{Conjugation}

Conjugation invariance is the following statement:

\begin{theorem}\label{thm:conjugation}
For any $\alpha, \beta \in B_n$, we have that $\C(\alpha^{-1} \beta \alpha) \homotopic_1 \C(\beta)$.
\end{theorem}

To begin, we prove a lemma relating complexes associated to diagrams that locally look like the pictures in \cref{fig:conjugation-lemma}.

\begin{figure}[ht]
    \centering
    \begin{tabularx}{\textwidth}{YY}
        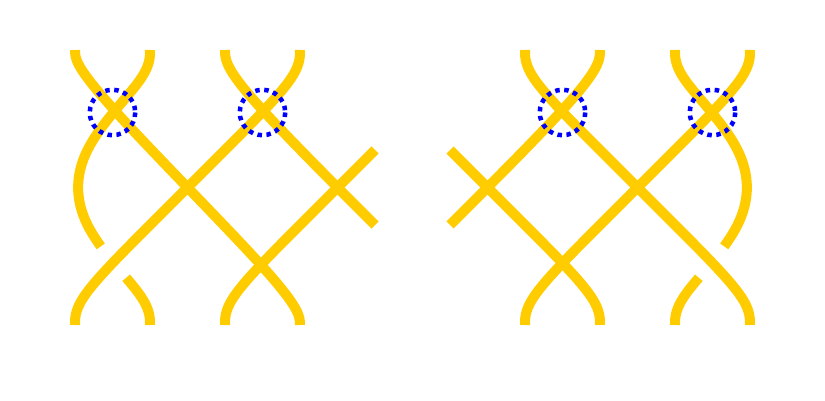 &
\begingroup%
  \makeatletter%
  \providecommand\color[2][]{%
    \errmessage{(Inkscape) Color is used for the text in Inkscape, but the package 'color.sty' is not loaded}%
    \renewcommand\color[2][]{}%
  }%
  \providecommand\transparent[1]{%
    \errmessage{(Inkscape) Transparency is used (non-zero) for the text in Inkscape, but the package 'transparent.sty' is not loaded}%
    \renewcommand\transparent[1]{}%
  }%
  \providecommand\rotatebox[2]{#2}%
  \newcommand*\fsize{\dimexpr\f@size pt\relax}%
  \newcommand*\lineheight[1]{\fontsize{\fsize}{#1\fsize}\selectfont}%
  \ifx\svgwidth\undefined%
    \setlength{\unitlength}{237.59999657bp}%
    \ifx\svgscale\undefined%
      \relax%
    \else%
      \setlength{\unitlength}{\unitlength * \real{\svgscale}}%
    \fi%
  \else%
    \setlength{\unitlength}{\svgwidth}%
  \fi%
  \global\let\svgwidth\undefined%
  \global\let\svgscale\undefined%
  \makeatother%
  \begin{picture}(1,0.4848485)%
    \lineheight{1}%
    \setlength\tabcolsep{0pt}%
    \put(0,0){\includegraphics[width=\unitlength,page=1]{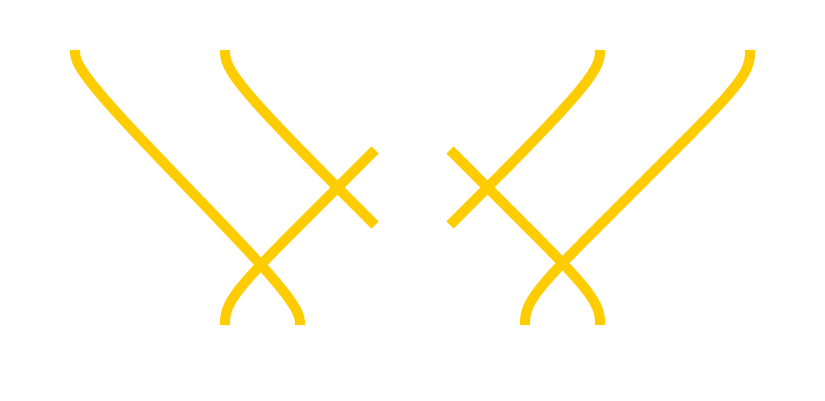}}%
    \put(0.46714871,0.24758818){\color[rgb]{0,0,0}\makebox(0,0)[lt]{\lineheight{1.25}\smash{\begin{tabular}[t]{l}$\cdots$\end{tabular}}}}%
    \put(0.46714871,0.11122452){\color[rgb]{0,0,0}\makebox(0,0)[lt]{\lineheight{1.25}\smash{\begin{tabular}[t]{l}$\cdots$\end{tabular}}}}%
    \put(0.46714871,0.38395186){\color[rgb]{0,0,0}\makebox(0,0)[lt]{\lineheight{1.25}\smash{\begin{tabular}[t]{l}$\cdots$\end{tabular}}}}%
    \put(0,0){\includegraphics[width=\unitlength,page=2]{conjugation-lemma-prime.pdf}}%
  \end{picture}%
\endgroup%
 \\
        $A$ & $A'$
    \end{tabularx}
    \caption{Local pictures of diagrams with equivalent $\C(-)$.}
    \label{fig:conjugation-lemma}
\end{figure}

\begin{lemma}\label{lemma:conjugation-lemma}
Let $A$ and $A'$ be partially singular braid diagrams that are identical outside of a specific region, where they look like the diagrams in \cref{fig:conjugation-lemma}, i.e.\ $A$ has two opposite crossings whereas $A'$ has oriented smoothings. Then $\C(A) \homotopic_1 \C(A')$. 
\end{lemma}

\begin{proof}
Let $\phi_1 = \phi_+ = 1$ be the edge map corresponding to the left (positive) crossing, and let $\phi_2 = \phi_- = a_n - e_{n-1}$ be the edge map corresponding to the right (negative) crossing. We expand the cube of resolutions for $\C(A)$ as follows:

\begin{center}
\begin{tikzcd}[ampersand replacement=\&, row sep = huge, column sep = huge]
    \C(A_{00}) \ar[r,"\oldphi_1"] \ar[d,"\oldphi_2"] \& \C(A_{10}) \ar[d,"\oldphi_2"] \\
    \C(A_{01}) \ar[r,"\oldphi_1"] \& \C(A_{11}) \nospaceperiod
\end{tikzcd}
\end{center}

The diagrams for these four partial resolutions look like \cref{fig:conjugation-cube}.

\begin{figure}[ht]
    \centering
    \begin{tikzcd}[ampersand replacement=\&, row sep = huge, column sep = huge]
\begingroup%
  \makeatletter%
  \providecommand\color[2][]{%
    \errmessage{(Inkscape) Color is used for the text in Inkscape, but the package 'color.sty' is not loaded}%
    \renewcommand\color[2][]{}%
  }%
  \providecommand\transparent[1]{%
    \errmessage{(Inkscape) Transparency is used (non-zero) for the text in Inkscape, but the package 'transparent.sty' is not loaded}%
    \renewcommand\transparent[1]{}%
  }%
  \providecommand\rotatebox[2]{#2}%
  \newcommand*\fsize{\dimexpr\f@size pt\relax}%
  \newcommand*\lineheight[1]{\fontsize{\fsize}{#1\fsize}\selectfont}%
  \ifx\svgwidth\undefined%
    \setlength{\unitlength}{151.19999313bp}%
    \ifx\svgscale\undefined%
      \relax%
    \else%
      \setlength{\unitlength}{\unitlength * \real{\svgscale}}%
    \fi%
  \else%
    \setlength{\unitlength}{\svgwidth}%
  \fi%
  \global\let\svgwidth\undefined%
  \global\let\svgscale\undefined%
  \makeatother%
  \begin{picture}(1,0.4761905)%
    \lineheight{1}%
    \setlength\tabcolsep{0pt}%
    \put(0.44837687,0.22240049){\color[rgb]{0,0,0}\makebox(0,0)[lt]{\lineheight{1.25}\smash{\begin{tabular}[t]{l}$\cdots$\end{tabular}}}}%
    \put(0.44837687,0.0557333){\color[rgb]{0,0,0}\makebox(0,0)[lt]{\lineheight{1.25}\smash{\begin{tabular}[t]{l}$\cdots$\end{tabular}}}}%
    \put(0.44837687,0.3890672){\color[rgb]{0,0,0}\makebox(0,0)[lt]{\lineheight{1.25}\smash{\begin{tabular}[t]{l}$\cdots$\end{tabular}}}}%
    \put(0,0){\includegraphics[width=\unitlength,page=1]{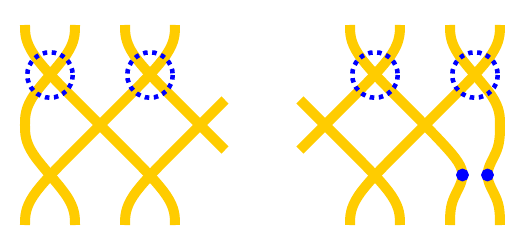}}%
  \end{picture}%
\endgroup%
 \ar[r,"\oldphi_1",start anchor=real east, end anchor=real west] \ar[d,"\oldphi_2"] \& 
\begingroup%
  \makeatletter%
  \providecommand\color[2][]{%
    \errmessage{(Inkscape) Color is used for the text in Inkscape, but the package 'color.sty' is not loaded}%
    \renewcommand\color[2][]{}%
  }%
  \providecommand\transparent[1]{%
    \errmessage{(Inkscape) Transparency is used (non-zero) for the text in Inkscape, but the package 'transparent.sty' is not loaded}%
    \renewcommand\transparent[1]{}%
  }%
  \providecommand\rotatebox[2]{#2}%
  \newcommand*\fsize{\dimexpr\f@size pt\relax}%
  \newcommand*\lineheight[1]{\fontsize{\fsize}{#1\fsize}\selectfont}%
  \ifx\svgwidth\undefined%
    \setlength{\unitlength}{151.19999313bp}%
    \ifx\svgscale\undefined%
      \relax%
    \else%
      \setlength{\unitlength}{\unitlength * \real{\svgscale}}%
    \fi%
  \else%
    \setlength{\unitlength}{\svgwidth}%
  \fi%
  \global\let\svgwidth\undefined%
  \global\let\svgscale\undefined%
  \makeatother%
  \begin{picture}(1,0.4761905)%
    \lineheight{1}%
    \setlength\tabcolsep{0pt}%
    \put(0,0){\includegraphics[width=\unitlength,page=1]{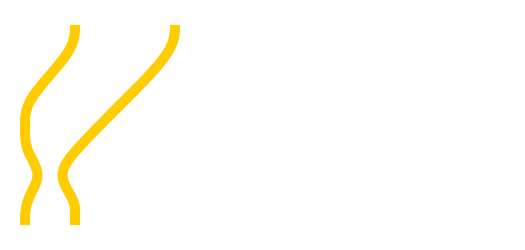}}%
    \put(0.44837687,0.22240049){\color[rgb]{0,0,0}\makebox(0,0)[lt]{\lineheight{1.25}\smash{\begin{tabular}[t]{l}$\cdots$\end{tabular}}}}%
    \put(0.44837687,0.0557333){\color[rgb]{0,0,0}\makebox(0,0)[lt]{\lineheight{1.25}\smash{\begin{tabular}[t]{l}$\cdots$\end{tabular}}}}%
    \put(0.44837687,0.3890672){\color[rgb]{0,0,0}\makebox(0,0)[lt]{\lineheight{1.25}\smash{\begin{tabular}[t]{l}$\cdots$\end{tabular}}}}%
    \put(0,0){\includegraphics[width=\unitlength,page=2]{conjugation-lemma-10.pdf}}%
  \end{picture}%
\endgroup%
 \ar[d,"\oldphi_2"] \\
\begingroup%
  \makeatletter%
  \providecommand\color[2][]{%
    \errmessage{(Inkscape) Color is used for the text in Inkscape, but the package 'color.sty' is not loaded}%
    \renewcommand\color[2][]{}%
  }%
  \providecommand\transparent[1]{%
    \errmessage{(Inkscape) Transparency is used (non-zero) for the text in Inkscape, but the package 'transparent.sty' is not loaded}%
    \renewcommand\transparent[1]{}%
  }%
  \providecommand\rotatebox[2]{#2}%
  \newcommand*\fsize{\dimexpr\f@size pt\relax}%
  \newcommand*\lineheight[1]{\fontsize{\fsize}{#1\fsize}\selectfont}%
  \ifx\svgwidth\undefined%
    \setlength{\unitlength}{151.19999313bp}%
    \ifx\svgscale\undefined%
      \relax%
    \else%
      \setlength{\unitlength}{\unitlength * \real{\svgscale}}%
    \fi%
  \else%
    \setlength{\unitlength}{\svgwidth}%
  \fi%
  \global\let\svgwidth\undefined%
  \global\let\svgscale\undefined%
  \makeatother%
  \begin{picture}(1,0.4761905)%
    \lineheight{1}%
    \setlength\tabcolsep{0pt}%
    \put(0.44837687,0.22240049){\color[rgb]{0,0,0}\makebox(0,0)[lt]{\lineheight{1.25}\smash{\begin{tabular}[t]{l}$\cdots$\end{tabular}}}}%
    \put(0.44837687,0.0557333){\color[rgb]{0,0,0}\makebox(0,0)[lt]{\lineheight{1.25}\smash{\begin{tabular}[t]{l}$\cdots$\end{tabular}}}}%
    \put(0.44837687,0.3890672){\color[rgb]{0,0,0}\makebox(0,0)[lt]{\lineheight{1.25}\smash{\begin{tabular}[t]{l}$\cdots$\end{tabular}}}}%
    \put(0,0){\includegraphics[width=\unitlength,page=1]{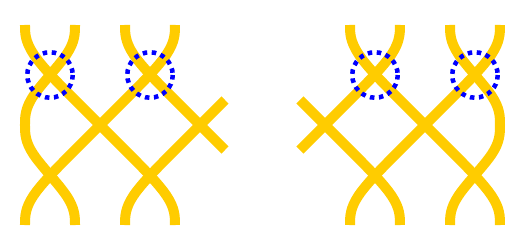}}%
  \end{picture}%
\endgroup%
 \ar[r,"\oldphi_1",start anchor=real east, end anchor=real west] \& 
\begingroup%
  \makeatletter%
  \providecommand\color[2][]{%
    \errmessage{(Inkscape) Color is used for the text in Inkscape, but the package 'color.sty' is not loaded}%
    \renewcommand\color[2][]{}%
  }%
  \providecommand\transparent[1]{%
    \errmessage{(Inkscape) Transparency is used (non-zero) for the text in Inkscape, but the package 'transparent.sty' is not loaded}%
    \renewcommand\transparent[1]{}%
  }%
  \providecommand\rotatebox[2]{#2}%
  \newcommand*\fsize{\dimexpr\f@size pt\relax}%
  \newcommand*\lineheight[1]{\fontsize{\fsize}{#1\fsize}\selectfont}%
  \ifx\svgwidth\undefined%
    \setlength{\unitlength}{151.19999313bp}%
    \ifx\svgscale\undefined%
      \relax%
    \else%
      \setlength{\unitlength}{\unitlength * \real{\svgscale}}%
    \fi%
  \else%
    \setlength{\unitlength}{\svgwidth}%
  \fi%
  \global\let\svgwidth\undefined%
  \global\let\svgscale\undefined%
  \makeatother%
  \begin{picture}(1,0.4761905)%
    \lineheight{1}%
    \setlength\tabcolsep{0pt}%
    \put(0,0){\includegraphics[width=\unitlength,page=1]{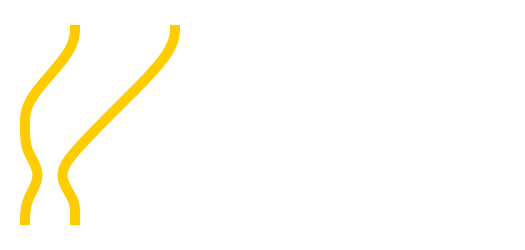}}%
    \put(0.44837687,0.22240049){\color[rgb]{0,0,0}\makebox(0,0)[lt]{\lineheight{1.25}\smash{\begin{tabular}[t]{l}$\cdots$\end{tabular}}}}%
    \put(0.44837687,0.0557333){\color[rgb]{0,0,0}\makebox(0,0)[lt]{\lineheight{1.25}\smash{\begin{tabular}[t]{l}$\cdots$\end{tabular}}}}%
    \put(0.44837687,0.3890672){\color[rgb]{0,0,0}\makebox(0,0)[lt]{\lineheight{1.25}\smash{\begin{tabular}[t]{l}$\cdots$\end{tabular}}}}%
    \put(0,0){\includegraphics[width=\unitlength,page=2]{conjugation-lemma-11.pdf}}%
  \end{picture}%
\endgroup%
 \nospaceperiod
    \end{tikzcd}
    \caption{The cube of resolutions for diagram $A$ in \cref{fig:conjugation-lemma}.}
    \label{fig:conjugation-cube}
\end{figure}

We can use MOY III moves to simplify three of the four corners of this cube. For $A_{00}$, we can start with a MOY III move on the left, simplifying the diagram. Each MOY III move we do allows us to do another, until we have done $n-1$ such moves moving left-to-right. We denote the resulting diagram $A''_{00}$; it is shown in \cref{fig:conjugation-cube-reduced}. By \cref{thm:moy-iii}, $A_{00}$ and $A'_{00}$ are $E_1$-quasi-isomorphic. 

Similarly, we can simplify $A_{11}$ to $A'_{11}$ by performing $n-1$ MOY III moves right-to-left, and we can simplify $A_{01}$ to $A'_{01}$ by performing $n-1$ MOY III moves left-to-right. In each case, \cref{thm:moy-iii} ensures we are preserving the $E_1$-quasi-isomorphism type. The resulting diagrammatic cube of resolutions is shown in \cref{fig:conjugation-cube-reduced}. Thus we obtain the complex

\begin{center}
\begin{tikzcd}[ampersand replacement=\&, row sep = huge, column sep = huge]
    \C(A''_{00})\spannedby{x} \ar[r,"\oldphi_1"] \ar[d,"\oldphi_2"] \& \C(A_{10}) \ar[d,"\oldphi_2"] \\
    \C(A''_{01})\spannedby{x} \directsum \Upsilon \ar[r,"\oldphi_1"] \& \C(A''_{11})\spannedby{x} \nospaceperiod
\end{tikzcd}
\end{center}

This cube ignores the $\Upsilon$ summands in $A_{00}$ and $A_{11}$ by \cref{lemma:homotopic-cubes}, but retains the $\Upsilon\defeq\Upsilon_1\directsum\dots\directsum\Upsilon_{n-1}$ summand in $A_{01}$. Further, as every vertex in the middle row is free, we may choose $x = (b_1 - c_1)\dots(b_{n-1} - c_{n-1})$ to be the generator for all three complexes modulo the linear ideal $L$. 

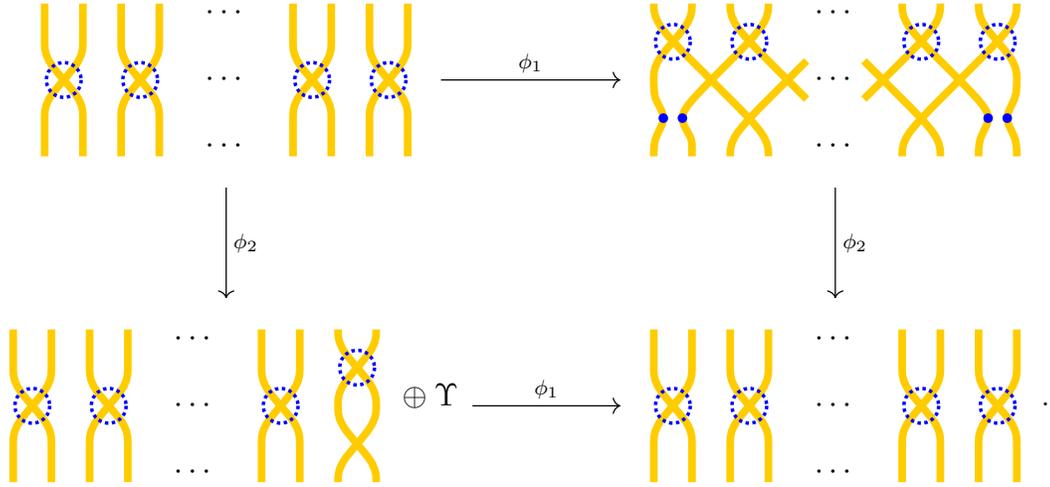
\begin{figure}[ht]
    \centering
    \begin{tikzcd}[ampersand replacement=\&, row sep = huge, column sep = huge]
        \raisebox{-0.5\height}{
\begingroup%
  \makeatletter%
  \providecommand\color[2][]{%
    \errmessage{(Inkscape) Color is used for the text in Inkscape, but the package 'color.sty' is not loaded}%
    \renewcommand\color[2][]{}%
  }%
  \providecommand\transparent[1]{%
    \errmessage{(Inkscape) Transparency is used (non-zero) for the text in Inkscape, but the package 'transparent.sty' is not loaded}%
    \renewcommand\transparent[1]{}%
  }%
  \providecommand\rotatebox[2]{#2}%
  \newcommand*\fsize{\dimexpr\f@size pt\relax}%
  \newcommand*\lineheight[1]{\fontsize{\fsize}{#1\fsize}\selectfont}%
  \ifx\svgwidth\undefined%
    \setlength{\unitlength}{151.19999313bp}%
    \ifx\svgscale\undefined%
      \relax%
    \else%
      \setlength{\unitlength}{\unitlength * \real{\svgscale}}%
    \fi%
  \else%
    \setlength{\unitlength}{\svgwidth}%
  \fi%
  \global\let\svgwidth\undefined%
  \global\let\svgscale\undefined%
  \makeatother%
  \begin{picture}(1,0.4761905)%
    \lineheight{1}%
    \setlength\tabcolsep{0pt}%
    \put(0.44837687,0.22240049){\color[rgb]{0,0,0}\makebox(0,0)[lt]{\lineheight{1.25}\smash{\begin{tabular}[t]{l}$\cdots$\end{tabular}}}}%
    \put(0.44837687,0.0557333){\color[rgb]{0,0,0}\makebox(0,0)[lt]{\lineheight{1.25}\smash{\begin{tabular}[t]{l}$\cdots$\end{tabular}}}}%
    \put(0.44837687,0.3890672){\color[rgb]{0,0,0}\makebox(0,0)[lt]{\lineheight{1.25}\smash{\begin{tabular}[t]{l}$\cdots$\end{tabular}}}}%
    \put(0,0){\includegraphics[width=\unitlength,page=1]{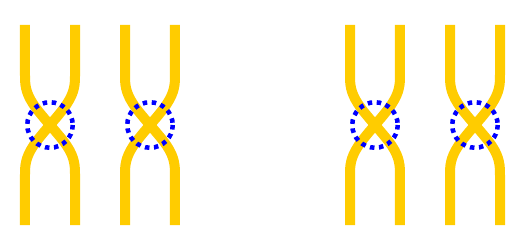}}%
  \end{picture}%
\endgroup%
} \ar[r,"\oldphi_1",start anchor=real east, end anchor=real west] \ar[d,"\oldphi_2"] \& \raisebox{-0.5\height}{} \ar[d,"\oldphi_2"] \\
        \raisebox{-0.5\height}{
\begingroup%
  \makeatletter%
  \providecommand\color[2][]{%
    \errmessage{(Inkscape) Color is used for the text in Inkscape, but the package 'color.sty' is not loaded}%
    \renewcommand\color[2][]{}%
  }%
  \providecommand\transparent[1]{%
    \errmessage{(Inkscape) Transparency is used (non-zero) for the text in Inkscape, but the package 'transparent.sty' is not loaded}%
    \renewcommand\transparent[1]{}%
  }%
  \providecommand\rotatebox[2]{#2}%
  \newcommand*\fsize{\dimexpr\f@size pt\relax}%
  \newcommand*\lineheight[1]{\fontsize{\fsize}{#1\fsize}\selectfont}%
  \ifx\svgwidth\undefined%
    \setlength{\unitlength}{151.19999313bp}%
    \ifx\svgscale\undefined%
      \relax%
    \else%
      \setlength{\unitlength}{\unitlength * \real{\svgscale}}%
    \fi%
  \else%
    \setlength{\unitlength}{\svgwidth}%
  \fi%
  \global\let\svgwidth\undefined%
  \global\let\svgscale\undefined%
  \makeatother%
  \begin{picture}(1,0.4761905)%
    \lineheight{1}%
    \setlength\tabcolsep{0pt}%
    \put(0.44837687,0.22240049){\color[rgb]{0,0,0}\makebox(0,0)[lt]{\lineheight{1.25}\smash{\begin{tabular}[t]{l}$\cdots$\end{tabular}}}}%
    \put(0.44837687,0.0557333){\color[rgb]{0,0,0}\makebox(0,0)[lt]{\lineheight{1.25}\smash{\begin{tabular}[t]{l}$\cdots$\end{tabular}}}}%
    \put(0.44837687,0.3890672){\color[rgb]{0,0,0}\makebox(0,0)[lt]{\lineheight{1.25}\smash{\begin{tabular}[t]{l}$\cdots$\end{tabular}}}}%
    \put(0,0){\includegraphics[width=\unitlength,page=1]{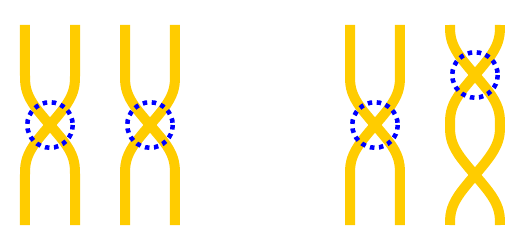}}%
  \end{picture}%
\endgroup%
} \directsum \Upsilon \ar[r,"\oldphi_1",start anchor=real east, end anchor=real west] \& \raisebox{-0.5\height}{} \nospaceperiod
    \end{tikzcd}
    \caption{The reduced cube of resolutions.}
    \label{fig:conjugation-cube-reduced}
\end{figure}

We further decompose $A''_{01}$ via an MOY II move on the right into two copies of $A''_{00}$, generated by $x$ and $a_nx$. Additionally, we see that $A''_{00}$ and $A''_{11}$ are isomorphic. We can compute the maps induced by $\phi_1$ and $\phi_2$ and write our complex as

\begin{center}
\begin{tikzcd}[ampersand replacement=\&, row sep = huge, column sep = huge]
    \C(A''_{00})\spannedby{x} \ar[r,"\oldphi_1"] \ar[d,"\pmat{-e_{n-1} \\ 1 \\ *}"] \& \C(A_{10}) \ar[d,"\oldphi_2"] \\
    \C(A''_{00})\spannedby{x} \directsum \C(A''_{00})\spannedby{a_n x} \directsum \Upsilon \ar[r,"\pmat{1 & f_{n-1} & *}"] \& \C(A''_{00})\spannedby{x} \nospaceperiod
\end{tikzcd}
\end{center}

We may cancel the $1$s in the above matrices to reduce the complex by \cref{lemma:gaussian-elimination} to obtain

\begin{center}
\begin{tikzcd}[ampersand replacement=\&, row sep = huge, column sep = huge]
    0 \ar[r] \ar[d] \& \C(A_{10}) \ar[d] \\
    \Upsilon \ar[r] \& 0 \nospaceperiod
\end{tikzcd}
\end{center}

Since $\Upsilon$ is a direct sum of $E_1$-acyclic complexes, we see that the $E_2$-page of the above complex is isomorphic to that of $\C(A_{10})$, which is isomorphic to $\C(A')$, thereby proving \cref{lemma:conjugation-lemma} in the case of a positive crossing on the left and a negative one on the right.

The opposite case is analogous; applying the same moves (mirrored horizontally) results in the complex
\begin{center}
\begin{tikzcd}[ampersand replacement=\&, row sep = huge, column sep = huge]
    \C(A''_{00})\spannedby{x} \ar[r,"\pmat{f_0 \\ -1 \\ *}"] \ar[d,"\oldphi_2"] \& \C(A''_{00})\spannedby{x} \directsum \C(A''_{00})\spannedby{b_0 x} \directsum \Upsilon \ar[d,"\pmat{1 & e_0 & *}"] \\
    \C(A_{01}) \ar[r,"\oldphi_1"] \& \C(A''_{00})\spannedby{x}
\end{tikzcd}
\end{center}
which we can simplify to get that the $E_2$-page is the same as that of $\C(A_{01})$ and therefore $\C(A')$.
\end{proof}

With this lemma in hand, we are now prepared to prove conjugation invariance.

\begin{figure}[ht]
    \centering
    \begin{tabularx}{\textwidth}{YY}
        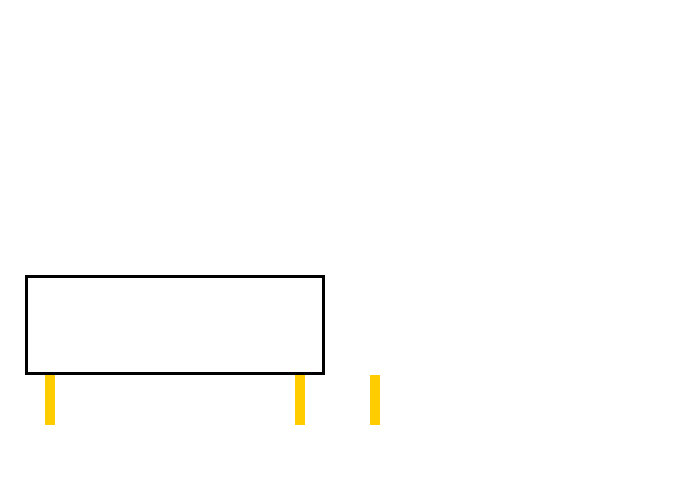 &
        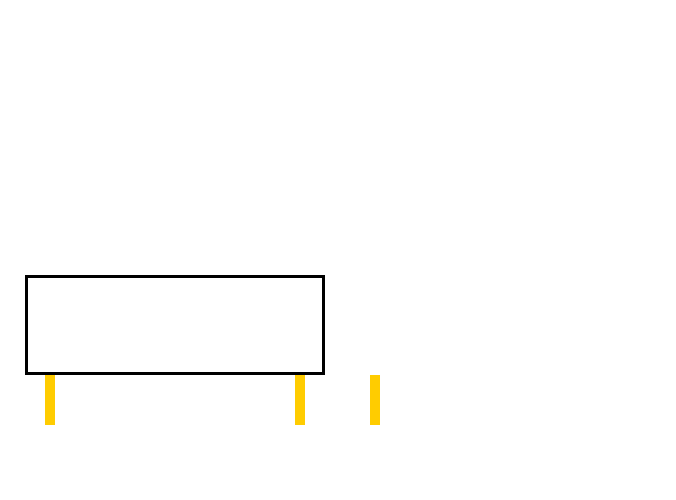 \\
        $D$ & $D'$
    \end{tabularx}
    \caption{When $\gamma=1$, the diagram $D$ is the result of conjugating the braid $\beta$ in $D'$ by a generator of the braid group.}
    \label{fig:conjugation}
\end{figure}

\begin{figure}[ht]
    \centering
    \begin{tabularx}{\textwidth}{YY}
        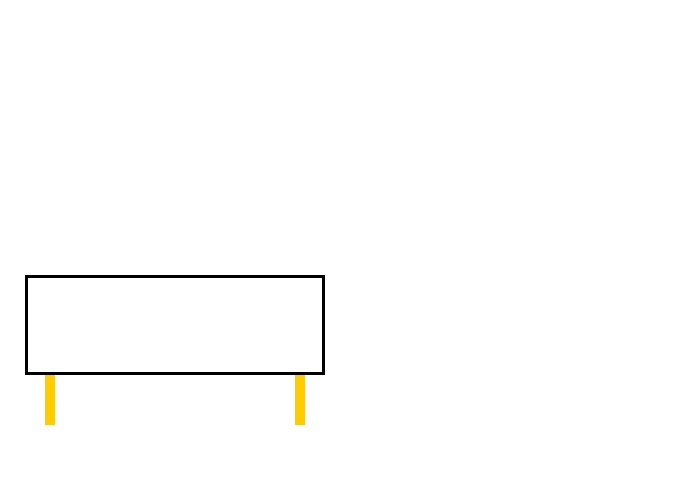 &
        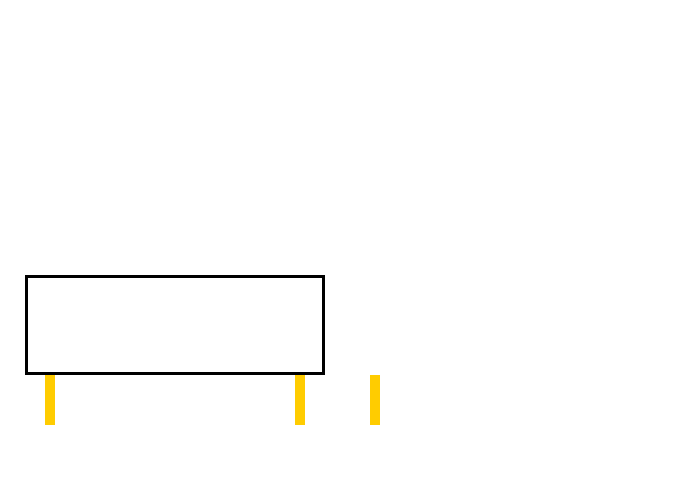 \\
        $D''$ & $D'''$
    \end{tabularx}
    \caption{Alternate diagrams for proving conjugation invariance.}
    \label{fig:conjugation-alt}
\end{figure}

\begin{proof}[Proof of \cref{thm:conjugation}]
It suffices to prove this in the case that $\alpha=\sigma_i^{\pm 1}$ is any generator of the braid group (or its inverse). Therefore, let $\sigma_i \in B_n$ be the generator which introduces a positive crossing between strands $i$ and $i+1$.

Graphically, we would like to show that $\C(D) \homotopic_1 \C(D')$, where $D$ and $D'$ are the partially singular braids depicted in \cref{fig:conjugation}, when $\gamma = 1$. In order to prove this, we will instead show that $\C(D'') \homotopic_1 \C(D') \homotopic_1 \C(D''')$ for a generic $\gamma \in B_n$, where $D''$ and $D'''$ are the diagrams in \cref{fig:conjugation-alt}.

Since we are considering the case $\alpha = \sigma_i$, note that in $D''$, the positive crossing occurs between strands $i$ and $i+1$. If we decompose $I_n$, we see that for any $i$, we locally get a picture like \cref{fig:conjugation-lemma}, where the top row of vertices is fixed if $i=1$, and free if $i > 1$. Therefore, we may apply \cref{lemma:conjugation-lemma} directly to see that $\C(D'') \homotopic_1 \C(D')$. Additionally, note that in $D'''$, the negative crossing occurs between strands $i$ and $i+1$. If we decompose $I_n$, we see that for any $i$, we locally get a picture like \cref{fig:conjugation-lemma}, except that the bottom row of vertices is fixed if $i=1$, and free if $i > 1$. In the latter case, this is not an issue and we may proceed as before to use \cref{lemma:conjugation-lemma} to prove that $\C(D''') \homotopic_1 \C(D')$. If $i=1$, then we first use \cref{thm:vertex-relabeling} to relabel the top row of vertices as free and the bottom row as fixed; this diagram is still in $\dr$ as it contains the open braid $S_{2n}$ from \cite{nate} as a sub-diagram, so we may proceed with the rest of the proof as usual.

Therefore, we can prove the desired equivalence $\C(D) \homotopic_1 \C(D')$ when $\gamma=1$ by first performing a Reidemeister II move to add two crossings to the right side of $D$, then using the equivalences $\C(D'') \homotopic_1 \C(D')$ and $\C(D''') \homotopic_1 \C(D')$ to simplify the diagram to $D'$. This proves \cref{thm:conjugation} in the case of $\alpha = \sigma_i$. We illustrate these steps for the case $n=2$ and $\alpha = \sigma_1$ in \cref{fig:conjugation-steps}. The proof for $\alpha = \sigma_i^{-1}$ is analogous, from which the proof for general $\alpha \in B_n$ follows.
\end{proof}

\begin{figure}[ht]
    \centering
    \begin{tabularx}{\textwidth}{YYYYYY}
\begingroup%
  \makeatletter%
  \providecommand\color[2][]{%
    \errmessage{(Inkscape) Color is used for the text in Inkscape, but the package 'color.sty' is not loaded}%
    \renewcommand\color[2][]{}%
  }%
  \providecommand\transparent[1]{%
    \errmessage{(Inkscape) Transparency is used (non-zero) for the text in Inkscape, but the package 'transparent.sty' is not loaded}%
    \renewcommand\transparent[1]{}%
  }%
  \providecommand\rotatebox[2]{#2}%
  \newcommand*\fsize{\dimexpr\f@size pt\relax}%
  \newcommand*\lineheight[1]{\fontsize{\fsize}{#1\fsize}\selectfont}%
  \ifx\svgwidth\undefined%
    \setlength{\unitlength}{72bp}%
    \ifx\svgscale\undefined%
      \relax%
    \else%
      \setlength{\unitlength}{\unitlength * \real{\svgscale}}%
    \fi%
  \else%
    \setlength{\unitlength}{\svgwidth}%
  \fi%
  \global\let\svgwidth\undefined%
  \global\let\svgscale\undefined%
  \makeatother%
  \begin{picture}(1,2)%
    \lineheight{1}%
    \setlength\tabcolsep{0pt}%
    \put(0,0){\includegraphics[width=\unitlength,page=1]{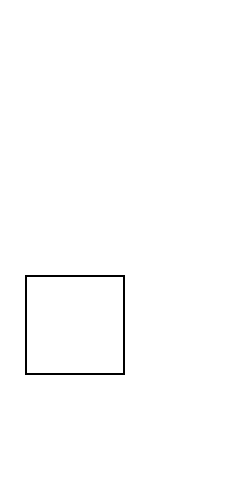}}%
    \put(0.26659063,0.66704113){\color[rgb]{0,0,0}\makebox(0,0)[lt]{\lineheight{1.25}\smash{\begin{tabular}[t]{l}$\beta$\end{tabular}}}}%
    \put(0,0){\includegraphics[width=\unitlength,page=2]{conjugation-step-1.pdf}}%
  \end{picture}%
\endgroup%
 &
\begingroup%
  \makeatletter%
  \providecommand\color[2][]{%
    \errmessage{(Inkscape) Color is used for the text in Inkscape, but the package 'color.sty' is not loaded}%
    \renewcommand\color[2][]{}%
  }%
  \providecommand\transparent[1]{%
    \errmessage{(Inkscape) Transparency is used (non-zero) for the text in Inkscape, but the package 'transparent.sty' is not loaded}%
    \renewcommand\transparent[1]{}%
  }%
  \providecommand\rotatebox[2]{#2}%
  \newcommand*\fsize{\dimexpr\f@size pt\relax}%
  \newcommand*\lineheight[1]{\fontsize{\fsize}{#1\fsize}\selectfont}%
  \ifx\svgwidth\undefined%
    \setlength{\unitlength}{72bp}%
    \ifx\svgscale\undefined%
      \relax%
    \else%
      \setlength{\unitlength}{\unitlength * \real{\svgscale}}%
    \fi%
  \else%
    \setlength{\unitlength}{\svgwidth}%
  \fi%
  \global\let\svgwidth\undefined%
  \global\let\svgscale\undefined%
  \makeatother%
  \begin{picture}(1,2)%
    \lineheight{1}%
    \setlength\tabcolsep{0pt}%
    \put(0,0){\includegraphics[width=\unitlength,page=1]{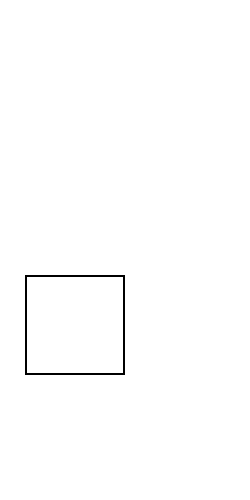}}%
    \put(0.26659063,0.66704113){\color[rgb]{0,0,0}\makebox(0,0)[lt]{\lineheight{1.25}\smash{\begin{tabular}[t]{l}$\beta$\end{tabular}}}}%
    \put(0,0){\includegraphics[width=\unitlength,page=2]{conjugation-step-2.pdf}}%
  \end{picture}%
\endgroup%
 &
\begingroup%
  \makeatletter%
  \providecommand\color[2][]{%
    \errmessage{(Inkscape) Color is used for the text in Inkscape, but the package 'color.sty' is not loaded}%
    \renewcommand\color[2][]{}%
  }%
  \providecommand\transparent[1]{%
    \errmessage{(Inkscape) Transparency is used (non-zero) for the text in Inkscape, but the package 'transparent.sty' is not loaded}%
    \renewcommand\transparent[1]{}%
  }%
  \providecommand\rotatebox[2]{#2}%
  \newcommand*\fsize{\dimexpr\f@size pt\relax}%
  \newcommand*\lineheight[1]{\fontsize{\fsize}{#1\fsize}\selectfont}%
  \ifx\svgwidth\undefined%
    \setlength{\unitlength}{72bp}%
    \ifx\svgscale\undefined%
      \relax%
    \else%
      \setlength{\unitlength}{\unitlength * \real{\svgscale}}%
    \fi%
  \else%
    \setlength{\unitlength}{\svgwidth}%
  \fi%
  \global\let\svgwidth\undefined%
  \global\let\svgscale\undefined%
  \makeatother%
  \begin{picture}(1,2)%
    \lineheight{1}%
    \setlength\tabcolsep{0pt}%
    \put(0,0){\includegraphics[width=\unitlength,page=1]{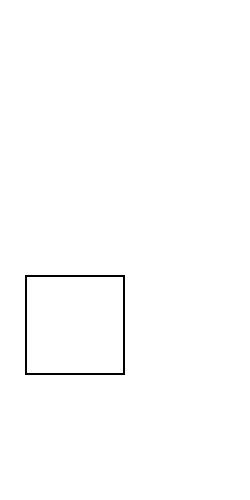}}%
    \put(0.26659063,0.66704113){\color[rgb]{0,0,0}\makebox(0,0)[lt]{\lineheight{1.25}\smash{\begin{tabular}[t]{l}$\beta$\end{tabular}}}}%
    \put(0,0){\includegraphics[width=\unitlength,page=2]{conjugation-step-3.pdf}}%
  \end{picture}%
\endgroup%
 &
\begingroup%
  \makeatletter%
  \providecommand\color[2][]{%
    \errmessage{(Inkscape) Color is used for the text in Inkscape, but the package 'color.sty' is not loaded}%
    \renewcommand\color[2][]{}%
  }%
  \providecommand\transparent[1]{%
    \errmessage{(Inkscape) Transparency is used (non-zero) for the text in Inkscape, but the package 'transparent.sty' is not loaded}%
    \renewcommand\transparent[1]{}%
  }%
  \providecommand\rotatebox[2]{#2}%
  \newcommand*\fsize{\dimexpr\f@size pt\relax}%
  \newcommand*\lineheight[1]{\fontsize{\fsize}{#1\fsize}\selectfont}%
  \ifx\svgwidth\undefined%
    \setlength{\unitlength}{72bp}%
    \ifx\svgscale\undefined%
      \relax%
    \else%
      \setlength{\unitlength}{\unitlength * \real{\svgscale}}%
    \fi%
  \else%
    \setlength{\unitlength}{\svgwidth}%
  \fi%
  \global\let\svgwidth\undefined%
  \global\let\svgscale\undefined%
  \makeatother%
  \begin{picture}(1,2)%
    \lineheight{1}%
    \setlength\tabcolsep{0pt}%
    \put(0,0){\includegraphics[width=\unitlength,page=1]{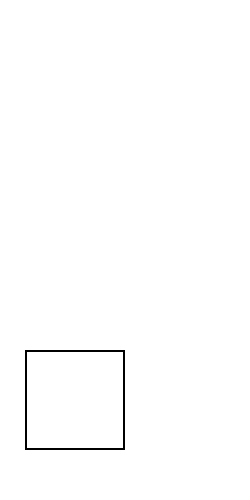}}%
    \put(0.26659063,0.36704104){\color[rgb]{0,0,0}\makebox(0,0)[lt]{\lineheight{1.25}\smash{\begin{tabular}[t]{l}$\beta$\end{tabular}}}}%
    \put(0,0){\includegraphics[width=\unitlength,page=2]{conjugation-step-4.pdf}}%
  \end{picture}%
\endgroup%
 &
\begingroup%
  \makeatletter%
  \providecommand\color[2][]{%
    \errmessage{(Inkscape) Color is used for the text in Inkscape, but the package 'color.sty' is not loaded}%
    \renewcommand\color[2][]{}%
  }%
  \providecommand\transparent[1]{%
    \errmessage{(Inkscape) Transparency is used (non-zero) for the text in Inkscape, but the package 'transparent.sty' is not loaded}%
    \renewcommand\transparent[1]{}%
  }%
  \providecommand\rotatebox[2]{#2}%
  \newcommand*\fsize{\dimexpr\f@size pt\relax}%
  \newcommand*\lineheight[1]{\fontsize{\fsize}{#1\fsize}\selectfont}%
  \ifx\svgwidth\undefined%
    \setlength{\unitlength}{72bp}%
    \ifx\svgscale\undefined%
      \relax%
    \else%
      \setlength{\unitlength}{\unitlength * \real{\svgscale}}%
    \fi%
  \else%
    \setlength{\unitlength}{\svgwidth}%
  \fi%
  \global\let\svgwidth\undefined%
  \global\let\svgscale\undefined%
  \makeatother%
  \begin{picture}(1,2)%
    \lineheight{1}%
    \setlength\tabcolsep{0pt}%
    \put(0,0){\includegraphics[width=\unitlength,page=1]{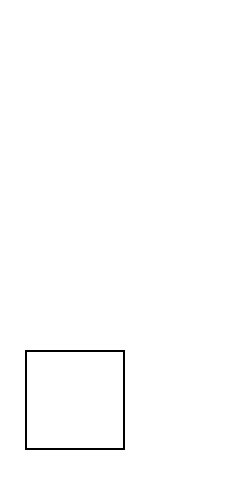}}%
    \put(0.26659063,0.36704104){\color[rgb]{0,0,0}\makebox(0,0)[lt]{\lineheight{1.25}\smash{\begin{tabular}[t]{l}$\beta$\end{tabular}}}}%
    \put(0,0){\includegraphics[width=\unitlength,page=2]{conjugation-step-5.pdf}}%
  \end{picture}%
\endgroup%
 &
\begingroup%
  \makeatletter%
  \providecommand\color[2][]{%
    \errmessage{(Inkscape) Color is used for the text in Inkscape, but the package 'color.sty' is not loaded}%
    \renewcommand\color[2][]{}%
  }%
  \providecommand\transparent[1]{%
    \errmessage{(Inkscape) Transparency is used (non-zero) for the text in Inkscape, but the package 'transparent.sty' is not loaded}%
    \renewcommand\transparent[1]{}%
  }%
  \providecommand\rotatebox[2]{#2}%
  \newcommand*\fsize{\dimexpr\f@size pt\relax}%
  \newcommand*\lineheight[1]{\fontsize{\fsize}{#1\fsize}\selectfont}%
  \ifx\svgwidth\undefined%
    \setlength{\unitlength}{72bp}%
    \ifx\svgscale\undefined%
      \relax%
    \else%
      \setlength{\unitlength}{\unitlength * \real{\svgscale}}%
    \fi%
  \else%
    \setlength{\unitlength}{\svgwidth}%
  \fi%
  \global\let\svgwidth\undefined%
  \global\let\svgscale\undefined%
  \makeatother%
  \begin{picture}(1,2)%
    \lineheight{1}%
    \setlength\tabcolsep{0pt}%
    \put(0,0){\includegraphics[width=\unitlength,page=1]{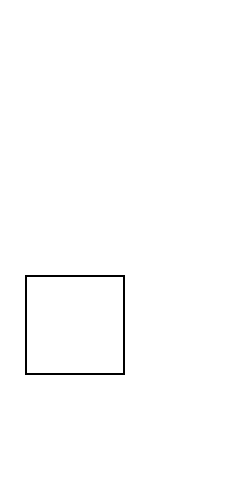}}%
    \put(0.26659063,0.66704113){\color[rgb]{0,0,0}\makebox(0,0)[lt]{\lineheight{1.25}\smash{\begin{tabular}[t]{l}$\beta$\end{tabular}}}}%
    \put(0,0){\includegraphics[width=\unitlength,page=2]{conjugation-step-6.pdf}}%
  \end{picture}%
\endgroup%

    \end{tabularx}
    \caption{The steps to prove conjugation invariance for $n=2$ and $\alpha=\sigma_1$.}
    \label{fig:conjugation-steps}
\end{figure}

\newpage
\appendix

\section{Homological Algebra}

In this section we review a few lemmas in homological algebra that will aid in our calculations. We note that all four lemmas are true for filtered complexes, replacing maps with filtered maps and quasi-isomorphisms with filtered quasi-isomorphisms. Additionally, if we instead assume that filtered maps have filtration degree 1, then these lemmas still hold, replacing $\cone(f)$ with $\cone_1(f)$ and filtered quasi-isomorphism with $E_1$-quasi-isomorphism.

\begin{lemma}{{\cite[Lemma 4.2]{barnatan2006}}}
\label{lemma:gaussian-elimination}
If $\varphi: A \to B$ is an isomorphism of complexes, then the double complexes
\begin{center}
\begin{tikzcd}[ampersand replacement=\&, row sep = huge, column sep = huge]
C \ar[r,"\pmat{\alpha \\ \beta}"] \& A \directsum D \ar[r,"\pmat{\varphi & \delta \\ \gamma & \epsilon}"] \& B \directsum E \ar[r,"\pmat{\mu & \nu }"] \& F
\end{tikzcd}
\end{center}
and
\begin{center}
\begin{tikzcd}[ampersand replacement=\&, row sep = huge, column sep = huge]
C \ar[r,"\beta"] \& D \ar[r,"\epsilon-\gamma\varphi^{-1}\delta"] \& E \ar[r,"\nu"] \& F
\end{tikzcd}
\end{center}
are quasi-isomorphic.
\end{lemma}

This lemma is proved for the very general case of additive categories in \cite{barnatan2006}, and is well known in the specific case of free modules over a ring as the ``cancellation lemma'' or ``reduction algorithm''. We require it to prove invariance in \cref{sec:invariance}, as it greatly simplifies calculations involving cubes of resolutions.

\begin{lemma}\label{lemma:homotopic-cones}
Let $f: A \to B$ be a map of complexes, and suppose that $A \iso A' \directsum A''$ and $B \iso B' \directsum B''$, where $A''$ and $B''$ are acyclic. Let $\iota: A' \inj A$ and $\pi: B \surj B'$ be the associated inclusion and projection maps, respectively. Then $\cone(f) \homotopic \cone(\pi \comp f \comp \iota)$.
\end{lemma}

\begin{proof}
First, we note that $\cone(\iota)$ is acyclic. One way to see this is via a cancellation argument: we have that $\cone(\iota) \iso \left(A' \to A' \directsum A''\right)$, which is quasi-isomorphic to $A''$ by \cref{lemma:gaussian-elimination}. Similarly, we get that $\cone(\pi)$, being quasi-isomorphic to $B''$, is acyclic as well.

For any two maps $\alpha: X \to Y$ and $\beta: Y \to Z$ of complexes, we have a long exact sequence relating the homology groups of $\cone(\alpha)$, $\cone(\beta)$, and $\cone(\beta\comp\alpha)$ (for example, via the octahedral axiom for triangulated categories applied to the derived category of $R$-modules). Therefore, we get that $\cone(f) \homotopic \cone(f \comp \iota) \homotopic \cone(\pi \comp f \comp \iota)$.
\end{proof}

\begin{lemma}\label{lemma:homotopic-cubes}
Let $A,B,C,D$ be complexes, and suppose that $A \iso A' \directsum A''$ and $D \iso D' \directsum D''$, where $A''$ and $D''$ are acyclic. Let $\iota: A' \inj A$ and $\pi: D \surj D'$ be the associated inclusion and projection maps, respectively.  Then the following two cube of resolutions complexes have the same homotopy type:

\begin{tabularx}{\textwidth}{XX}
\begin{center}
\begin{tikzcd}[ampersand replacement = \&, row sep = huge, column sep = huge]
A \ar[r,"f_1"] \ar[d,"g_1"] \& B \ar[d,"g_2"] \\
C \ar[r,"f_2"] \& D
\end{tikzcd}
\end{center}
&
\begin{center}
\begin{tikzcd}[ampersand replacement = \&, row sep = huge, column sep = huge]
A' \ar[r,"f_1 \comp \iota"] \ar[d,"g_1 \comp \iota"] \& B \ar[d,"\pi \comp g_2"] \\
C \ar[r,"\pi \comp f_2"] \& D' \nospaceperiod
\end{tikzcd}
\end{center}
\end{tabularx}
\end{lemma}

\begin{proof}
We know that the inclusion $\cone(g_1 \comp \iota) \inj \cone(g_1)$ and the projection $\cone(g_2) \surj \cone(\pi \comp g_2)$ are quasi-isomorphisms by the proof of \cref{lemma:homotopic-cones}.
\begin{center}
\begin{tikzcd}[ampersand replacement = \&, row sep = huge, column sep = huge]
A' \ar[r,"\iota"] \ar[d,"g_1\comp\iota"] \& A \ar[d,"g_1"] \& B \ar[r,"\id_B"] \ar[d,"g_2"] \& B \ar[d,"\pi\comp g_2"] \\
C \ar[r,"\id_C"] \& C \& D \ar[r,"\pi"] \& D'
\end{tikzcd}
\end{center}
We can also view the maps $f_1: A \to B$ and $f_2: C \to D$ as components of a map $f: \cone(g_1) \to \cone(g_2)$. Therefore, we can compose $f$ with the inclusion and projection to get a single map $f': \cone(g_1 \comp \iota) \to \cone(\pi \comp g_2)$. By the same long exact sequence logic as before, the cone of this map has the same homotopy type as $f$, i.e.\
\begin{equation*}
    \cone(f') = \cone(\cone(g_1 \comp \iota) \to \cone(\pi \comp g_2)) \homotopic \cone(\cone(g_1) \to \cone(g_2)) = \cone(f)
\end{equation*}
We conclude by noting that the complex on the left in \cref{lemma:homotopic-cubes} is $\cone(f)$, and the complex on the right is $\cone(f')$.
\end{proof}

While the above lemma is phrased only for squares, it can be iterated to reduce summands of higher-dimensional cubes as well.

Since our complexes in this paper are often constructed as mapping cones, it will help to know when a quasi-isomorphism is induced by maps on the components of the cone.

\begin{lemma}\label{lemma:cone-components}
Suppose that we have the following commutative diagram of chain maps.
\begin{center}
\begin{tikzcd}[ampersand replacement = \&, row sep = huge, column sep = huge]
A_0 \ar[r, "g"] \ar[d, "f_0"] \& A_1 \ar[d, "f_1"] \\
B_0 \ar[r, "g'"] \& B_1
\end{tikzcd}
\end{center}
Let $A = \cone(g)$ and $B = \cone(g')$, so that we get a map $f: A \to B$ with components $f_0$ and $f_1$. If $f_0$ and $f_1$ are quasi-isomorphisms, then so is $f$.
\end{lemma}

\begin{proof}
By properties of the mapping cone, $f$ induces a map of short exact sequences (with grading shifts suppressed)
\begin{center}
\begin{tikzcd}[ampersand replacement = \&, row sep = huge, column sep = huge]
0 \ar[r] \& A_1 \ar[r] \ar[d,"f_0"] \& A \ar[r] \ar[d,"f"] \& A_0 \ar[d,"f_1"] \ar[r] \& 0 \\
0 \ar[r] \& B_1 \ar[r] \& B \ar[r] \& B_0 \ar[r] \& 0
\end{tikzcd}
\end{center}
We can look at the induced map of long exact sequences in homology
\begin{center}
\begin{tikzcd}[ampersand replacement = \&, row sep = huge, column sep = large]
\dots \ar[r] \& H_*(A_1) \ar[r] \ar[d,"H_*(f_1)"] \& H_*(A) \ar[r] \ar[d,"H_*(f)"] \& H_*(A_0) \ar[d,"H_*(f_0)"]  \ar[r] \& \dots \\
\dots \ar[r] \& H_*(B_1) \ar[r] \& H_*(B) \ar[r] \& H_*(B_0)  \ar[r] \& \dots
\end{tikzcd}
\end{center}
to conclude that $H_*(f)$ must be an isomorphism as well, so $f$ is a quasi-isomorphism.
\end{proof}

To prove the filtered generalizations of \cref{lemma:cone-components}, one replaces $H_*(-)$ with $E_1(-)$ or $E_2(-)$ (see \cite[Exercise 5.4.4]{weibel1994}).

\bibliographystyle{plain}
\bibliography{master}

\end{document}